\documentclass[11pt]{amsart}
\usepackage{graphicx}
\newtheorem{thm}{Theorem}[section]
\newtheorem{cor}[thm]{Corollary}

\newtheorem{lem}[thm]{Lemma}
\newtheorem{prop}[thm]{Proposition}
\theoremstyle{definition}
\newtheorem{defn}[thm]{Definition}
\theoremstyle{remark}
\newtheorem{rem}[thm]{Remark}
\theoremstyle{conclusion}

\numberwithin{equation}{section}
\begin{document}
\title[ Continuous Dependence of NLS]
{ Continuous Dependence of Cauchy Problem For Nonlinear
Schr\"{o}dinger Equation in $H^{s}$ }

\author{Wei Dai,  Weihua Yang and Daomin Cao }

\address{Institute of Applied Mathematics, AMSS, Chinese Academy of Sciences, Beijing 100190, P. R. China}
\email{daiwei@amss.ac.cn}

\address{Department of Mathematics and Natural Science, Beijing University of Technology,
 Beijing 100124, P. R. China}
\email{whyang@bjut.edu.cn}

\address{Institute of Applied Mathematics, AMSS, Chinese Academy of Sciences, Beijing 100190, P. R. China}
\email{dmcao@amt.ac.cn}

%\thanks{}
\begin{abstract}
We consider the Cauchy problem for the nonlinear Schr\"{o}dinger
equation $i \partial_{t}u+ \Delta u=\lambda_{0}u+\lambda_{1}|u|^\alpha
u$ in $\mathbb{R}^{N}$, where $\lambda_{0},\lambda_{1}\in\mathbb{C}$, in $H^s$ subcritical and critical case: $0<\alpha\leq\frac{4}{N-2s}$ when $1<s<\frac{N}{2}$ and $0<\alpha<+\infty$ when $s\geq\frac{N}{2}$. We show that the solution depends continuously on the initial value in the standard sense in $H^{s}(\mathbb{R}^{N})$ if $\alpha$ satisfies certain assumptions.
\end{abstract}
\maketitle {\small {\bf Keywords:} Nonlinear Schr\"{o}dinger equation; Continuous dependence; Cauchy problem.\\

}

\section{INTRODUCTION}

In this paper we are concerned with the following Cauchy problem for the nonlinear Schr\"{o}dinger equation\\
\begin{equation}\label{eq1}
    \left\{
  \begin{array}{ll}
    i \partial_{t}u+ \Delta u=g(u), \\
    u(0,x)=\phi(x),
  \end{array}
\right.
\end{equation}
where $\phi\in H^s(\mathbb{R}^{N})$, $N\geq1$ and $s\geq \min\{1,N/2\}$. We assume that the nonlinearity $g=g_{0}+g_{1}$, where $g_{0}(u)=\lambda_{0}u$ with $\lambda_{0}\in\mathbb{C}$ and the nonlinear term $g_{1}$ is $H^{s}$-subcritical or critical and of class $\mathcal{C}(\alpha,s)$ for some $0<\alpha<\infty$.

\begin{defn}\label{class}
Let $f:\mathbb{C}\rightarrow\mathbb{C}$, $0<\alpha<\infty$, $[s]$ denotes the largest integer which is less than $s$ and $k$-th order complex partial derivative $D^{k}$ with $k\in\mathbb{N}$ be defined under the identification $\mathbb{C}=\mathbb{R}^{2}$(see Section 2). When $0<s\leq N/2$, we say that $f$ is of class $\mathcal{C}(\alpha,s)$ provided it satisfies both the two assumptions: \\
(i) \,$f\in C^{[s]+1}(\mathbb{C},\mathbb{C})$ with $f(0)=0$; \\
(ii) if $f(z)$ is a polynomial in $z$ and $\bar{z}$, then $1<deg(f)=1+\alpha\leq 1+\frac{4}{N-2s}$($1<deg(f)=1+\alpha<\infty$, if $s=N/2$); if $f$ is not a polynomial, then
\begin{equation}\label{class1}
|D^{k}f(u)|\leq C|u|^{\alpha+1-k}
\end{equation}
for any $0\leq k\leq[s]+1$ and $u\in\mathbb{C}$, and
\begin{equation}\label{class2}
|D^{[s]+1}(f(u)-f(v))|\leq C|u-v|^{\min(\alpha-[s],1)}(|u|+|v|)^{\max(0,\alpha-1-[s])}
\end{equation}
for any $u,v\in\mathbb{C}$, where $[s]<\alpha\leq\frac{4}{N-2s}$($[s]<\alpha<\infty$, if $s=N/2$).
When $s>N/2$, we say that $f$ is of class $\mathcal{C}(\alpha,s)$ if $f\in C^{[s]+1}(\mathbb{C},\mathbb{C})$ with $f(0)=0$.
\end{defn}

\begin{rem}\label{model}
Assume that if $s<N/2$, $0<\alpha\leq\frac{4}{N-2s}$ and if $s\geq N/2$, $0<\alpha<\infty$. If $\alpha$ is not an even integer, assume further that $\alpha>[s]$. One easily verifies that the nonlinearity $g(u)=\lambda|u|^{\alpha}u$ with $\lambda\in\mathbb{C}$ is a model case of class $\mathcal{C}(\alpha,s)$(see \cite{G,N1,N2}).
\end{rem}

The Cauchy problem (\ref{eq1}) in the Sobolev space $H^s(\mathbb{R}^{N})$($s\geq0$) has been quite extensively studied (see \cite{b2,b5,b6,F1,b3,b1,b7,Ponce4,b17,P,b4}). For $s=0,1,2$, the local well-posedness for Cauchy problem (\ref{eq1}) with local nonlinearity $g(u)$ has been studied(see \cite{b2,b6,b1,b4}). For general $0<s<N/2$, local well-posedness results for (\ref{eq1}) were established when $g(u)=\lambda|u|^{\alpha}u$, where $\lambda\in\mathbb{C}$ and $\alpha>0$ satisfying certain regularity assumptions, and the following result is well known now(see theorem 1.1 in \cite{b5}, see also theorem 4.9.9 and theorem 4.9.10 in \cite{b6}).
\begin{thm}\label{thm0}(\cite{b5,b6})
Assume $N\geq 3$, $1<s<\frac{N}{2}$. Let $g(u)=\lambda |u|^\alpha
u$ with $\lambda \in \mathbb{C }$ and
\begin{equation}\label{g1'}
0\leq\alpha\leq\frac{4}{N-2s}.
\end{equation}
If $\alpha$ is not an even integer, suppose further that
 \begin{equation}\label{eq6}
 \alpha>[s],
 \end{equation}
where $[s]$ denotes the largest integer which is less than $s$($[s]=s-1$ if $s$ is an integer).
Then for any given $\phi\in H^{s}(\mathbb{R}^{N})$, there exist
$T_{max},T_{min}\in (0,\infty]$ and a unique, maximal solution $u\in
C((-T_{min}, T_{max}), H^{s}(\mathbb{R}^{N}))$ of Cauchy problem
$(\ref{eq1})$. Moreover, the following properties hold:

$(i)$ $u\in L^{q}_{loc}((-T_{min}(\phi),T_{max}(\phi)),B^{s}_{r,2}(\mathbb{R}^{N}))$
for every admissible pair $(q,r)$.

$(ii)$ $u$ depends continuously on $\phi$ in the following sense. There exists $0<T<T_{max}(\phi),T_{min}(\phi)$ such
that if $\phi_n\rightarrow \phi$ in $H^s$ and if $u_n$ denote the solution of $(\ref{eq1})$ with the initial value $\phi_n$, then $0<T<T_{max}(\phi_n),T_{min}(\phi_n)$ for all sufficiently large $n$ and $u_n$ is bounded in $L^{q}((-T,T),B^{s}_{r,2}(\mathbb{R}^{N}))$ for every admissible pair $(q,r)$. Moreover, $u_n\rightarrow u$ in
$L^{q}((-T,T),B^{s-\varepsilon}_{r,2}(\mathbb{R}^{N}))$ for all $\varepsilon>0$ and admissible pair $(q,r)$ as $n\to\infty$. In particular, $u_n\rightarrow u$ in $C([-T,T],H^{s-\varepsilon}(\mathbb{R}^{N}))$ for all $\varepsilon>0$.
\end{thm}

\begin{rem} \label{rem12}
In general, for $0<s<\infty$ and nonlinearity $g$ of class $\mathcal{C}(\alpha,s)$, the local well-posedness for Cauchy problem $(\ref{eq1})$ was also established by T. Kato in \cite{b7}, and the solution $u$ of Cauchy problem $(\ref{eq1})$ depends continuously on initial value $\phi$ in the sense of (ii) of Theorem \ref{thm0}(see \cite{b7}).
\end{rem}

Let $(e^{it\Delta})_{t\in \mathbb{R}}$ be the Schr\"{o}dinger group. The existence of solutions in Theorem \ref{thm0} is established by using fixed point theorem to the equivalent integral equation(Duhamel's formula)
\begin{equation}\label{eq2}
    u(t)=e^{it\Delta}\phi-i\int^{t}_{0}e^{i(t-\tau)\Delta}g(u(\tau)) d\tau
\end{equation}
in an appropriate space. So one would expect that the dependence of the solution on the initial value is locally Lipschitz in spaces of order $s$ differentiability. However, the metric space in which one applies Banach's fixed point theorem involves Sobolev (or Besov) norms of order s, while the distance only involves Lebesgue norms(The reason for that choice of the distance is that the nonlinearity need not be locally Lipschitz for Sobolev or Besov norms of positive order). Thus the flow is locally Lipschitz for Lebesgue norms, and continuous dependence in the sense of Theorem \ref{thm0} (ii) follows by interpolation inequalities(see \cite{b5,b7}).

As for the statement of continuous dependence, it was pointed out in \cite{b6} that the above results are weaker than what would be ``standard", that is, $\varepsilon=0$(see remark 4.9.6 for subcritical case and theorem 4.9.10 for critical case in \cite{b6}, see also \cite{b5,b7}). We wonder if $\varepsilon=0$ is allowed, that is, if the solution depends continuously on the initial data in the sense that the local solution flow is continuous $H^{s}(\mathbb{R}^{N})\rightarrow H^{s}(\mathbb{R}^{N})$, it is not known until very recently. A positive answer was given recently by Cazenave, Fang and Han in \cite{CFZ}, where under the condition that $0<s<\min\{1,\,\frac{N}{2}\}$, the continuous dependence in $H^s$ in the standard sense was obtained.

In addition, when $s=0,1,2$(subcritical case), or $s=m$, $m>N/2$ is an integer(with nonlinearity $g(u)$ satisfying $g\in C^{m}(\mathbb{C}, \mathbb{C}$) in the real sense and $g(0)=0$), the standard continuous dependence in $H^{s}(\mathbb{R}^{N})$ has been proved(see \cite{b5,b6,K-M,b1,Tao2}).

In this paper, we  mainly consider nonlinearity $g(u)=\lambda_{0}u+g_{1}(u)$ with $g_{1}\in\mathcal{C}(\alpha,s)$(specially, a typical form $g(u)=\lambda_{0}u+\lambda_{1}|u|^{\alpha}u$) and address the question of $H^s(\mathbb{R}^{N})$ continuous dependence for non-integer and integer $s\geq\min\{1,\,\frac{N}{2}\}$. We show that continuous dependence holds in $H^{s}(\mathbb{R}^{N})$ in the standard sense in virtually all the cases where local existence of solution to (\ref{eq1}) is known. Under more restrictive conditions, we show that the dependence is locally Lipschitz(see Theorem \ref{th1} and Corollary \ref{cor1}).

To state our result and explain our strategy of its proof let us give some of the notation to be used in this paper first. We use the notation $[x]$ to denote the largest integer which is less than $x$, and the remainder part of $x$ is denoted by $\{x\}$, namely, $0<\{x\}:=x-[x]\leq 1$, in particular, if $x$ is an integer, then $[x]=x-1$ and $\{x\}=1$. We denote by $p'$ the conjugate of the exponent $p\in[1,\infty]$ defined by $\frac{1}{p}+\frac{1}{p'}=1$. We will use the usual notation for various complex-valued function spaces: Lebesgue space $L^{r}=L^{r}(\mathbb{R}^{N})$, Sobolev spaces $H^{s,r}=H^{s,r}(\mathbb{R}^{N}):=(I-\Delta)^{-s/2}L^{r}$, homogeneous Sobolev spaces $\dot{H}^{s,r}=\dot{H}^{s,r}(\mathbb{R}^{N}):=(-\Delta)^{-s/2}L^{r}$, Besov spaces $B^{s}_{r,b}=B^{s}_{r,b}(\mathbb{R}^{N})$ and homogeneous Besov spaces $\dot{B}^{s}_{r,b}=\dot{B}^{s}_{r,b}(\mathbb{R}^{N})$. For the definitions of these spaces and the corresponding interpolation and embedding properties, refer to \cite{b8,b9,b10,b11}. For any interval $I\subset\mathbb{R}$ and any Banach space $X$ just mentioned, we denote by $C(I,X)$ the space of strongly continuous functions from $I$ to $X$ and by $L^{q}(I,X)$ the space of measurable functions $u$ from $I$ to $X$ such that $\|u(\cdot)\|_{X}\in L^{q}(I)$. As usual, we define the ``admissible pair" as below, which plays an important role in our space-time estimates.

\begin{defn}
We say that a pair $(q,r)$ is admissible if
\begin{equation}\label{eq3}
    \frac{2}{q}=\delta(r)=N(\frac{1}{2}-\frac{1}{r})
\end{equation}
and $2\leq r\leq\frac{2N}{N-2}$ ($2\leq r\leq\infty$ if $N=1$, $2\leq r<\infty$ if $N=2$).\\
Note that if $(q,r)$ is an admissible pair, then $2\leq
q\leq\infty$, the pair $(\infty,2)$ is always admissible, and the
pair $(2,\frac{2N}{N-2})$ is admissible if $N\geq3$.
\end{defn}

In this paper we always assume the following condition for $\alpha$ and $s$,
\begin{equation}\label{eq4}
\left\{
\begin{array}{ll}
0<\alpha\leq \frac{4}{N-2s}, &  if\,\, 0\leq s<N/2,  \\
0<\alpha<\infty, &  if\,\, s\geq N/2.\\
\end{array}
\right.
\end{equation}
Moreover, since we are working in spaces of order $s$ differentiability, we need the nonlinear map $u\mapsto |u|^{\alpha}u$ to have certain regularity, which will sometimes be expressed by the condition given by \eqref{eq6}. \\

The main result in this paper is the following.

\begin{thm}\label{th1}
Assume $N\geq1$, $s>1$ or $s\geq N/2$. Let $g(u)=g_{0}(u)+g_{1}(u)$, where $g_{0}(u)=\lambda_{0}u$ with $\lambda_{0}\in\mathbb{C}$ and $g_{1}$ is of class $\mathcal{C}(\alpha,s)$. Then for any given $\phi\in
H^{s}(\mathbb{R}^{N})$, the corresponding solution $u$ of Cauchy problem $(\ref{eq1})$ obtained in \cite{b5,b7} depends continuously on the initial value $\phi$ in the following sense: \\
$(i)$ the mappings $\phi\mapsto T_{min}(\phi),T_{max}(\phi)$ are lower semi-continuous $H^{s}(\mathbb{R}^{N})\rightarrow (0,\infty]$, \\
$(ii)$ for every interval $[-S,T]\subset (-T_{min}(\phi),T_{max}(\phi))$, and every admissible pair $(q,r)$, if $\phi_{n}\rightarrow \phi$ in $H^{s}(\mathbb{R}^{N})$ and if $u_{n}$ denotes the solution of (\ref{eq1}) with the initial value $\phi_{n}$, then $u_{n}\rightarrow u$ in $L^{q}((-S,T),B^{s}_{r,2}(\mathbb{R}^{N}))$ as $n\rightarrow\infty$. In particular, $u_{n}\rightarrow u$ in $C([-S,T],H^{s}(\mathbb{R}^{N}))$.

In addition, for $s\leq N/2$, if $g_{1}(u)$ is a polynomial in $u$ and $\bar{u}$, or if $g_{1}$ is not a polynomial, we assume further that $\alpha\geq [s]+1$; or for $s>N/2$, if $g_{1}\in C^{[s]+2}(\mathbb{C},\mathbb{C})$; then the dependence is locally Lipschitz.
\end{thm}

By Remark \ref{model}, Theorem \ref{th1} applies in particular to the model case $g(u)=\lambda_{0}u+\lambda_{1}|u|^{\alpha}u$ with $\lambda_{0},\lambda_{1}\in\mathbb{C}$ and $\alpha>0$. We summarize the corresponding results in the following corollary.

\begin{cor}\label{cor1}
Assume $N\geq1$, $s>1$ or $s\geq N/2$. Let $g(u)=\lambda_{0}u+\lambda_{1}|u|^{\alpha}u$, where $\lambda_{0},\lambda_{1}\in\mathbb{C}$. Suppose that $\alpha>0$ satisfies $(\ref{eq4})$, and suppose further that $(\ref{eq6})$ holds if $\alpha$ is not an even integer. Then for any given $\phi\in H^{s}(\mathbb{R}^{N})$, the corresponding solution $u$ of Cauchy problem $(\ref{eq1})$ obtained in \cite{b5,b7} depends continuously on the initial value $\phi$ in the sense of Theorem \ref{th1}. In addition, if $\alpha$ is an even integer, or if $\alpha$ is not an even integer, we assume further that $\alpha\geq [s]+1$, then the dependence is locally Lipschitz.
\end{cor}

\begin{rem}\label{r1.6'}
The conclusions of Theorem \ref{th1} and Corollary \ref{cor1} also contains cases that $s$ equals an integer, note that in such cases $[s]=s-1$.
\end{rem}

\begin{rem}\label{r1.6''}
Note that Corollary $\ref{cor1}$ applies in particular to the single power case $g(u)=\lambda|u|^{\alpha}u$ with $\lambda\in \mathbb{C}$ and $\alpha>0$.
\end{rem}

%\subsection{ Outline of the proof}
The key ingredient in the proof of Theorem \ref{th1} is an estimate showing that $g_{1}(u_{n})-g_{1}(u)$ is bounded in Besov spaces by a Lipschitz term(i.e., with a factor $u_{n}-u$ in some Besov space of order $s$) plus a lower order Kato's remainder term $K(u_{n},u)$ of H\"{o}lder type. The Lipschitz term is easily absorbed by the left-hand side of the inequality, and what we need to do is showing that the Kato's remainder term $K(u_{n},u)$ converges to 0. Note that convergence in Lebesgue spaces holds by the contraction mapping estimates, thus for $s\geq N/2$, we can take advantage of the embedding $H^{s}\hookrightarrow L^{\infty}$(resp., $H^{s,r}\hookrightarrow L^{\infty}$ with $r>2$, if $s=\frac{N}{2}$) to obtain the conclusions. Our proof of the covergence of the remainder term $K(u_{n},u)$ for $s<N/2$ is partly based on the method used by T. Kato to prove continuous dependence in $H^{1}(\mathbb{R}^{N})$, which is actually a contradiction argument by using dominated convergence theorem; we also apply the idea presented by T. Tao and M. Visan in \cite{Tao2} of using intermediate Besov spaces of lower order(see Lemma \ref{lem6} below), which is based on Strichartz's estimates for non-admissible pairs and can be extended to solve the critical case(see Claim \ref{claim1}, Claim \ref{claim2}, note that $\gamma=\alpha+2$ in the critical case). Therefore, by applying Strichartz's estimates, the main difficulty in proving Theorem $\ref{th1}$ is to estimate terms like $\|g_{1}(u)-g_{1}(v)\|_{\dot{B}^{s}_{r,2}}$, where $\frac{2N}{N+2}\leq r\leq 2$. Because $\dot{H}^{s,r}\hookrightarrow \dot{B}^{s}_{r,2}$ for $\frac{2N}{N+2}\leq r\leq 2$, we will estimate $\|g_{1}(u)-g_{1}(v)\|_{\dot{H}^{s,r}}$ instead.

\begin{rem}\label{rem100}
In fact, when $s\leq N/2$, we can write the Kato's remainder term $K(u_{n},u)$ in a general form(see \eqref{Kato} and \eqref{Kato2} in Section 3)
\[K(u_{n},u)=\|u\|^{[s]+1}_{B^{s}_{\rho,2}}\|D^{[s]+1}g_{1}(u_{n})-D^{[s]+1}g_{1}(u)\|_{L^{\frac{\rho^{\ast}_{0}}{\alpha-[s]}}}.\]
Therefore, in order to obtain the continuity of the solution map $\phi\mapsto u$, we can also only assume (\ref{class1}) in definition \ref{class}(without assuming (\ref{class2})) and apply dominated convergence theorem to the Kato's remainder term $K(u_{n},u)$(except when $s$ equals an integer in the critical case, see the proof of Claim \ref{claim2}, in which we need the assumption (\ref{class2})). It's obvious that assumption (\ref{class1}) is enough for us to construct the dominating function, but we won't konw when the flow is locally Lipschitz. Here the general assumption (\ref{class2}) is presented as an alternative criterion between H\"{o}lder continuity and Lipschitz continuity of the solution map. Therefore, when $0<s<\min\{1,\,\frac{N}{2}\}$, if we assume further that $|g'(u)-g'(v)|\leq C|u-v|^{\min(\alpha,1)}(|u|+|v|)^{\max(0,\alpha-1)}$ for any $u,v\in\mathbb{C}$, then the continuous dependence obtained by Cazenave, Fang and Han in Theorem 1.2 in \cite{CFZ} will be locally Lipschitz if $\alpha\geq 1$, this gives an answer to the question raised by the authors in \cite{CFZ}.
\end{rem}

The rest of this paper is organized as follows. In section 2 we give some useful notation and preliminary knowledge. In section 3 we prove Theorem $\ref{th1}$ by dealing with four cases. \\

\section{Notation and preliminary knowledge}

In this section we collect some preliminary knowledge which will be
applied  to show Theorem $\ref{th1}$ in section 3.

\subsection{Some notation}
Throughout this paper, we use the following notation. $\bar{z}$ is the conjugate of the complex number $z$, $\Re z$ and $\Im z$ are respectively the real the imaginary part of the complex number $z$. All function spaces involved are spaces of complex valued functions. Let $L^{q}_{t}(\mathbb{R},L^{r}_{x}(\mathbb{R}^{N}))$ denote the Banach space with norm
\[\|u\|_{L^{q}(\mathbb{R},L^{r})}=(\int_{\mathbb{R}}(\int_{\mathbb{R}^{N}}|u(t,x)|^{r}dx)^{q/r}dt)^{1/q},\]
with the usual modifications when $q$ or $r$ is infinity, or when the domain $\mathbb{R}\times\mathbb{R}^{N}$ is replaced by a smaller region of spacetime such as $I\times\mathbb{R}^{N}$.

We use the convention $\langle x\rangle=(1+|x|^{2})^{1/2}$ and $\langle \nabla\rangle=\sqrt{I-\Delta}$.

We use the notation $[x]$ to denote the largest integer which is less than $x$, and the remainder part of $x$ is denoted by $\{x\}$, namely, $0<\{x\}=x-[x]\leq 1$.

The positive part of a real number $x$ is defined by $(x)^{+}$, that is, $(x)^{+}:=\max\{x,0\}$, the negative part of $x$ is defined by $(x)^{-}:=-\min\{x,0\}$.

For a function $f(z)$ defined for a single complex variable $z$ and
for a positive integer $k$, we denote by $f^{(k)}_{z\backslash\bar{z}}(z)$ one of the $k+1$ derivatives of
this form: $\frac{\partial^{k}f}{\partial z^{k}}, \frac{\partial^{k}f}{\partial z^{k-1}\partial\bar{z}}, \cdots, \frac{\partial^{k}f}{\partial z\partial \bar{z}^{k-1}}, \frac{\partial^{k}f}{\partial\bar{z}^{k}}$, that is, $f^{(k)}_{z\backslash\bar{z}}(z):=D^{k}f(z)$, where we use the convention $a\setminus b$ to denote the logical relationship ``$a$ or $b$". $f^{(1)}_{z\backslash\bar{z}}$ will be denoted by $f_{z\backslash\bar{z}}$ simply, where $f_{z}$, $f_{\bar{z}}$ are the usual complex derivatives(under the identification $\mathbb{C}=\mathbb{R}^{2}$) defined by
\[f_{z}:=D_{z}f=\frac{1}{2}(\frac{\partial f}{\partial x}-i\frac{\partial f}{\partial y}), \,\,\,\, f_{\bar{z}}:=D_{\bar{z}}f=\frac{1}{2}(\frac{\partial f}{\partial x}+i\frac{\partial f}{\partial y}).\]

For $0\leq s< N/2$ and $\alpha$ satisfying \eqref{eq4} there is a
particular admissible pair $(\gamma,\rho)$, defined by
\begin{equation}\label{eq7}
    \gamma=\frac{4(\alpha+2)}{\alpha(N-2s)},\rho=\frac{N(\alpha+2)}{N+s\alpha},
\end{equation}
which will play an important  role in our discussion.

In what follows positive constants will be denoted by $C$ and will change from line to line. If necessary, by $C_{\star,\cdots,\star}$ we denote positive constants depending only on the quantities appearing in subscript continuously.

The following proposition concerns the various parameters that occur in our paper.
\begin{prop}\label{prop2}
Suppose that $\alpha, s, \rho$, and $\gamma$ satisfy \eqref{eq4} and
\eqref{eq7}. For any $2\leq r< N/s$ and $0\leq j\leq [s]$, we define indices $r^{\ast}_{j}$ by
\[\frac{1}{r^{\ast}_{j}}=\frac{1}{r}-\frac{s-j}{N}.\]
Then:\\
{\rm(i)} $2<\rho<\frac{2N}{N-2}$, $2<\gamma<\infty$; \\
{\rm(ii)} $\rho<N/s$ and $\rho^{\ast}_{j}> \rho$ for any $0\leq j\leq [s]$; \\
{\rm (iii)} $\frac{1}{\rho'}+\frac{\alpha j}{N}=\frac{\alpha}{\rho^{\ast}_{j}}+\frac{1}{\rho}$, and in particular $\frac{1}{\rho'}=\frac{\alpha}{\rho^{\ast}_{0}}+\frac{1}{\rho}$; \\
{\rm (iv)} $\frac{1}{\gamma'}=\frac{4-\alpha(N-2s)}{4}+\frac{\alpha+1}{\gamma}\geq \alpha/\gamma+1/\gamma$, with equality if and only if $\alpha=\frac{4}{N-2s}$; \\
{\rm (v)} $\frac{1}{r^{\ast}_{j}}=\frac{j}{s}\cdot \frac{1}{r}+(1-\frac{j}{s})\frac{1}{r^{\ast}_{0}}$, from which we can deduce that $B^{s}_{r,2}\hookrightarrow H^{j,\,r^{\ast}_{j}}$.
\end{prop}
\begin{proof} The proof of (i)-(iv) needs only simple calculations. For (v), it's easy to note that
$\frac{1}{r^{\ast}_{j}}=\frac{j}{s}\cdot\frac{1}{r}+(1-\frac{j}{s})\frac{1}{r^{\ast}_{0}}$,
$B^{s}_{r,2}\hookrightarrow H^{s,r}$ and $B^{s}_{r,2}\hookrightarrow
L^{r^{\ast}_{0}}$. Then by Gagliardo-Nirenberg's inequality and
H\"{o}lder's inequality, we deduce that
\[\||\nabla|^{j}f\|_{L^{r^{\ast}_{j}}}\leq C'\||\nabla|^{s}f\|^{j/s}_{L^{r}}\cdot \|f\|^{1-j/s}_{L^{r^{\ast}_{0}}}\leq C\|f\|_{B^{s}_{r,2}},\]
\[\|f\|_{L^{r^{\ast}_{j}}}\leq C''\|f\|^{j/s}_{L^{r}}\cdot \|f\|^{1-j/s}_{L^{r^{\ast}_{0}}}\leq C\|f\|_{B^{s}_{r,2}},\]
for every $f\in B^{s}_{r,2}$. Thus, we have
$B^{s}_{r,2}\hookrightarrow H^{j,\,r^{\ast}_{j}}$.
\end{proof}

\subsection{ Basic harmonic analysis}
Let $\wp$ be an annulus in momentum space $\mathbb{R}^{N}_{\xi}$
given by
\[\wp=\{\xi\,:\, \frac{3}{4}\leq |\xi|\leq\frac{8}{3}\},\]
and let $\chi(\xi)$ be a radial bump function supported in the ball $\{\xi\in\mathbb{R}^{N}: |\xi|\leq 4/3\}$ and equal to 1 on the ball
 $\{\xi\in\mathbb{R}^{N}: |\xi|\leq 3/4\}$, $\hat{\varphi}(\xi)=\chi(\xi/2)-\chi(\xi)$ be a radial bump function supported in the annulus
  $\wp$ respectively, such that \[0\leq\chi(\xi),\,\,\widehat{\varphi}(\xi)\leq 1.\]
So we can define decomposition of the whole momentum space $\mathbb{R}^{N}_{\xi}$,
\[\left\{
  \begin{array}{ll}
    \chi(\xi)+\sum_{j\geq 0}\widehat{\varphi}(2^{-j}\xi)=1, & \forall \xi \in \mathbb{R}_\xi^N  \,\,\,\,\, (inhomogeneous),\\
    \,&\\
   \sum_{j\in \mathbb{Z}}\widehat{\varphi}(2^{-j}\xi)=1, & \forall \xi \in \mathbb{R}_\xi^N\setminus \{0\} \,\,\, (homogeneous),
  \end{array}
\right.\]
with the following properties: \\
i) $supp\, \widehat{\varphi}(2^{-j}\cdot)\bigcap
supp\,\widehat{\varphi}(2^{-j'}\cdot)=\phi$,
for $\forall j,j'\in\mathbb{Z}$  such that $|j-j'|\geq 2$; \\
ii) $supp\, \chi(\cdot)\bigcap
supp\,\widehat{\varphi}(2^{-j}\cdot)=\phi$, for $\forall j\geq 1$.

Now setting $h(x)=\mathcal{F}^{-1}(\chi(\xi))$,
$\varphi(x)=\mathcal{F}^{-1}(\widehat{\varphi}(\xi))$, then we can define the Fourier multipliers as following\\
$\left\{
\begin{array}{ll}
\Delta_{j}u=\mathcal{F}^{-1}(\widehat{\varphi}(2^{-j}\xi)\widehat{u}(\xi))
=2^{jN}\int_{R^N}\varphi(2^jy)u(x-y)dy, &j\in\mathbb{Z},\\
\,&\\
S_{<j}u=\mathcal{F}^{-1}(\chi(2^{-j}\xi)\widehat{u}(\xi))=2^{jN} \int_{R^N}h(2^jy)u(x-y)dy, &j\in\mathbb{Z},\\
\,&\\
S_{\geq j}u=u-S_{<j}u=\mathcal{F}^{-1}((1-\chi(2^{-j}\xi))\widehat{u}(\xi)), &j\in\mathbb{Z}.
\end{array}
\right.$ \\

Similarly, we can define $S_{\leq j}$, $S_{>j}$ and $S_{i\leq \cdot<j}:=S_{<j}-S_{<i}$. Due to identity $\widehat{\varphi}(\xi/2^j)=\chi(\xi/2^{j+1})-\chi(\xi/2^j)$, we have $\Delta_j u=(S_{<j+1}-S_{<j})u$ for any $j\in\mathbb{Z}$. So for any $u\in \mathcal{S}'(\mathbb{R}^{N})$, we have inhomogeneous Littlewood-Paley dyadic decomposition $u=S_{<0}u+\sum_{j\geq0}\Delta_{j}u$ and homogeneous decomposition $u=\sum_{j\in\mathbb{Z}}\Delta_{j}u$ in $\mathcal{S}'(\mathbb{R}^{N})$, respectively.

As with all Fourier multipliers, the Littlewood-Paley operators
commute with the Schr\"{o}dinger group $e^{it\Delta}$, as well as
with differential operators such as $i\partial_{t}+\Delta$. We will
use basic properties of these operators frequently, for instance, in
the following lemma(see \cite{b22,Tao2}).
\begin{lem}\label{bernstein}(Bernstein estimates). For $1\leq p\leq q\leq \infty$, $s\geq0$, we have
\[\||\nabla|^{\pm s}\Delta_{j}f\|_{L^{p}_{x}(\mathbb{R}^{N})}\sim 2^{\pm js}\|\Delta_{j}f\|_{L^{p}_{x}(\mathbb{R}^{N})},\]
\[\||\nabla|^{s}S_{\leq j}f\|_{L^{p}_{x}(\mathbb{R}^{N})}\leq C2^{js}\|S_{\leq j}f\|_{L^{p}_{x}(\mathbb{R}^{N})},\]
\[\|\Delta_{j}f\|_{L^{q}_{x}(\mathbb{R}^{N})}\leq C2^{j(\frac{N}{p}-\frac{N}{q})}\|\Delta_{j}f\|_{L^{p}_{x}(\mathbb{R}^{N})},\]
\[\|S_{\leq j}f\|_{L^{q}_{x}(\mathbb{R}^{N})}\leq C2^{j(\frac{N}{p}-\frac{N}{q})}\|S_{\leq j}f\|_{L^{p}_{x}(\mathbb{R}^{N})},\]
\[\|S_{\geq j}f\|_{L^{p}_{x}(\mathbb{R}^{N})}\leq C2^{-js}\||\nabla|^{s}S_{\geq j}f\|_{L^{p}_{x}(\mathbb{R}^{N})}.\]
\end{lem}
The next result, which can be found in \cite{b12}(see lecture note 6, corollary 7.2 in \cite{b12}), provides important tools for dealing with the fractional power appearing in the nonlinearity.
\begin{lem}\label{lem8}
Let $u$ be $\emph{Schwartz}$. If $F$ is a Lipschitz nonlinearity then
\[\|F(u)\|_{H^{s,q}(\mathbb{R}^{N})}\leq C_{s,q,N}\|u\|_{H^{s,q}(\mathbb{R}^{N})}\]
for all $1<q<\infty$ and $0<s<1$. If instead F is a power-type nonlinearity with exponent $p\geq 1$, that is, $|F(z)-F(\omega)|\leq C_{p}(|z|^{p-1}+|\omega|^{p-1})|z-\omega|$ for all $z,\omega\in\mathbb{C}$, then
\[\|F(u)\|_{H^{s,q}(\mathbb{R}^{N})}\leq C_{s,q,N}\|u\|^{p-1}_{L^{r}(\mathbb{R}^{N})}\|u\|_{H^{s,t}(\mathbb{R}^{N})}\]
whenever $1<q,r,t<\infty$ and $0<s<1$  such that $1/q=(p-1)/r+1/t$.
\end{lem}

For any $u,v \in S'(\mathbb{R}^{N})$, we have Littlewood-Paley decomposition
\[u=\sum_{j}\Delta_{j}u,\,\,\,\, v=\sum_{j'}\Delta_{j'}v,\,\,\, uv= \sum_{j,j'}\Delta_{j}u \Delta_{j'}v.\]
Therefore, the Bony's paraproducts decomposition of the product $uv$ can be expressed as(see \cite{b23})
\[uv=\sum_{j}S_{<j-1}u\Delta_{j}v+\sum_{j}S_{<j-1}v\Delta_{j}u+\sum_{|j-j'|\leq 1}\Delta_{j}u\Delta_{j'}v.\]
By applying the paraproduct estimates in Besov spaces, one can get the following Morse type product estimates(see R. Danchin \cite{b22}, proposition 1.4.3).
\begin{lem}\label{lem7}
(Morse type inequality). Let $s>0$ and $1\leq p,r \leq\infty$. Then $B_{p,r}^s\cap L^{\infty}$ is an algebra, and there exists constant
$C_{s}>0$ depending on $s$, such that the following estimate holds
\[\|uv\|_{B_{p,r}^s}\leq C_{s}(\|u\|_{L^{\infty}}\|v\|_{B_{p,r}^s}+\|v\|_{L^{\infty}}\|u\|_{B_{p,r}^s}).\]
\end{lem}
We have the following alternative definition of the Besov spaces(see \cite{b10,b11}), which is very convenient for nonlinear estimates.
\begin{lem}\label{alternative}
Assume that $s>0$ is not an integer. Let $1\leq p\leq\infty$ and $1\leq q<\infty$, then
\begin{equation}\label{alter definition}
    \|f\|_{B^{s}_{p,q}}\sim\|f\|_{L^{p}}+\sum_{j=1}^{N}(\int_{0}^{\infty}(t^{[s]-s}\sup_{|y|\leq t}\|\Delta_{y}\partial_{x_{j}}^{[s]}f\|_{L^{p}})^{q}\frac{dt}{t})^{1/q},
\end{equation}
where $[s]$ denotes the largest integer which is less than $s$ and $\Delta_{y}$ denotes the first order difference operator.
\end{lem}
From theorem 1, p. 119 in Stein \cite{b14}, we have
\begin{lem}\label{HLS}(Hardy-Littlewood-Sobolev fractional integration inequality).
Let $N\geq1$, $0<s<N$, and $1<p<q<\infty$ be such that
$\frac{N}{q}=\frac{N}{p}-s$. \\
Then we have
\[\left\||\nabla|^{-s}f\right\|_{L^{q}_{x}(\mathbb{R}^{N})}\sim\left\|\int_{\mathbb{R}^{N}}\frac{f(x)}{|x-y|^{N-s}}dx\right\|_{L^{q}_{y}(\mathbb{R}^{N})}
\leq C_{N,s,p,q}\|f\|_{L^{p}_{x}(\mathbb{R}^{N})}\] for all $f\in
L^{p}_{x}(\mathbb{R}^{N})$.
\end{lem}

\subsection{ Estimates of products in fractional order Sobolev spaces}
\begin{defn}
 Let $s\geq 0$. A homogeneous symbol of order $s$ is a symbol $m: \mathbb{R}^{N}\rightarrow \mathbb{C}$ which obeys the
 estimates: $|\nabla^{j}m(\xi)|\leq C_{j,s,N}|\xi|^{s-j}$ for all $j\geq 0$ and $\xi\neq 0$; and an inhomogeneous symbol of
 order $s$ is a symbol $m: \mathbb{R}^{N}\rightarrow \mathbb{C}$ which obeys the estimates: $|\nabla^{j}m(\xi)|\leq C_{j,s,N}\langle\xi\rangle^{s-j}$.
 The corresponding Fourier multipliers $m(D)$ are referred to as homogeneous and inhomogeneous Fourier multipliers of order $s$ respectively.
\end{defn}
Here we mainly consider the homogeneous Fourier multiplier $\dot{D}^{s}=|\nabla|^{s}$ and we need estimate of $\|\dot{D}^{s}(fg)\|_{L^{r}}$. Now let $0<s<1$ and $s_{1}, s_{2}>0$ such that $s_{1}+s_{2}=s$, applying Bony's paraproducts decomposition and Kato-Ponce type commutator identities(see \cite{b23,b12}), we have the fractional Leibnitz rule:
\[\dot{D}^{s}(fg)-(\dot{D}^{s}f)g-f(\dot{D}^{s}g)=\pi(|\nabla|^{s_{1}}f,|\nabla|^{s_{2}}g)\]
for some linear combination $\pi$ of Coifman-Meyer and residual paraproducts.

Thus one can deduce, from Moving derivatives rule and Coifman-Meyer and Residual multiplier theorem(see
\cite{b12}), the following product rule in fractional order Sobolev spaces(for the proof, refer to \cite{b20,b21}).
\begin{prop}\label{prop1}(Kato-Ponce inequality). Let $0<s<1$ and let $1<p_{1},q_{1},p_{2},q_{2},r<\infty$ be such that $\frac{1}{r}=\frac{1}{p_{1}}+\frac{1}{q_{1}}=\frac{1}{p_{2}}+\frac{1}{q_{2}}$. Then:
\[\|fg\|_{\dot{H}^{s,r}(\mathbb{R}^{N})}\leq C_{p_{1},p_{2},r,s,N}\left\{\|f\|_{\dot{H}^{s,p_{1}}(\mathbb{R}^{N})}\|g\|_{L^{q_{1}}
(\mathbb{R}^{N})}+\|f\|_{L^{p_{2}}(\mathbb{R}^{N})}\|g\|_{\dot{H}^{s,q_{2}}(\mathbb{R}^{N})}\right\}\]
for all \emph{Schwartz} $f,g$.
\end{prop}

\subsection{Properties of the Schr\"{o}dinger group $(e^{it\Delta})_{t\in\mathbb{R}}$}
We denote by $(e^{it\Delta})_{t\in\mathbb{R}}$ the Schr\"{o}dinger group, which is isometric on $H^{s}$ and $\dot{H}^{s}$ for every $s\geq0$, and satisfies the Dispersive estimate and Strichartz's estimates(for more details, see Keel and Tao \cite{Tao1}). We will use freely the well-known properties of of the Schr\"{o}dinger group $(e^{it\Delta})_{t\in\mathbb{R}}$(see \cite{b6} for an account of these properties). Here we only mention specially the Strichartz's estimates for non-admissible pair as below, which can be found in \cite{b6,Tao2}.
\begin {lem}\label{lem9}
(Strichartz's estimates for non-admissible pairs). Let $I$ be an interval of $\mathbb{R}$ (bounded or not), set $J=\bar{I}$, let $t_{0}\in J$, and consider $\Phi$ defined by: $t\mapsto \Phi_{f}(t)=\int^{t}_{t_{0}}e^{i(t-s)\Delta}f(s)ds$ for $t\in I$. Assume $2<r<2N/(N-2)$($2<r\leq \infty$ if $N=1$) and let $1<a,\tilde{a}<\infty$ satisfy \[\frac{1}{\tilde{a}}+\frac{1}{a}=\delta(r)=N(\frac{1}{2}-\frac{1}{r}).\]
It follows that $\Phi_{f}\in L^{a}(I,L^{r}(\mathbb{R}^{N}))$ for every $f\in L^{\tilde{a}'}(I,L^{r'}(\mathbb{R}^{N}))$. Moreover, there exists a constant C independent of $I$ such that
\[\|\Phi_{f}\|_{L^{a}(I,L^{r})}\leq C\|f\|_{L^{\tilde{a}'}(I,L^{r'})},\]
for every $f\in L^{\tilde{a}'}(I,L^{r'}(\mathbb{R}^{N}))$.
\end{lem}

\section{PROOF OF THEOREM \ref{th1}}
In this section we will assume $N\geq1$, $s>1$ or $s\geq N/2$. Let $g(u)=g_{0}(u)+g_{1}(u)$, where $g_{0}(u)=\lambda_{0}u$ with $\lambda_{0}\in\mathbb{C}$ and $g_{1}$ is of class $\mathcal{C}(\alpha,s)$(see Definition \ref{class}). \\

\emph{Proof of Theorem \ref{th1}.}

It follows from \cite{b6}(see also \cite{b5,b7}) that given $\phi\in H^{s}(\mathbb{R}^{N})$,
there exist $T_{max},T_{min}\in (0,\infty]$ such that Cauchy problem
(\ref{eq1}) has a unique, maximal solution $u\in
C((-T_{min},T_{max}),H^{s}(\mathbb{R}^{N}))$. There exists
$0<T<T_{max},T_{min}$ such that if $\phi_{n}\rightarrow \phi$ in
$H^{s}(\mathbb{R}^{N})$ and if $u_{n}$ denotes the solution of
(\ref{eq1}) with the initial value $\phi_{n}$, then
$\|\phi_{n}\|_{H^{s}}\leq 2\|\phi\|_{H^{s}}$ for $n$ large, and we
have $0<T<T_{max}(\phi_{n}),T_{min}(\phi_{n})$ for all sufficiently
large $n$ and $u_{n}$ is bounded in
$L^{q}((-T,T),B^{s}_{r,2}(\mathbb{R}^{N}))$ for any admissible pair
$(q,r)$.

Since $u, u_{n}$ satisfy integral equation
\[u(t)=e^{it\Delta}\phi-i\int^{t}_{0}e^{i(t-\tau)\Delta}g(u(\tau)) d\tau,\]
\[u_{n}(t)=e^{it\Delta}\phi_{n}-i\int^{t}_{0}e^{i(t-\tau)\Delta}g(u_{n}(\tau)) d\tau,\]
respectively. It follows directly that
\begin{equation}\label{eq8}
    u_{n}(t)-u(t)=e^{it\Delta}(\phi_{n}-\phi)-i\int^{t}_{0}e^{i(t-\tau)\Delta}[g(u_{n}(\tau))-g(u(\tau))] d\tau.
\end{equation}
Moreover, $u_{n}\rightarrow u$ in $L^{q}((-T,T),L^{r}(\mathbb{R}^{N}))$ as $n\rightarrow \infty$. In particular,
\begin{equation}\label{weak convergence}
    u_n\rightarrow u \,\,\,\, in \,\,\,\, L^{q}((-T,T),B^{s-\varepsilon}_{r,2}(\mathbb{R}^{N}))\cap C([-T,T],H^{s-\varepsilon}(\mathbb{R}^{N}))
\end{equation}
for all admissible pair $(q,r)$ and $\varepsilon>0$. As a consequence of (\ref{weak convergence}), we can deduce easily from Proposition \ref{prop2} that, if $s<\frac{N}{2}$,
\begin{equation}\label{equa100}
    u_{n}\rightarrow u \quad in \quad L^{\gamma-\delta}((-T,T),H^{j,\,\rho^{\ast}_{j}})
\end{equation}
as $n\rightarrow \infty$, for all $0\leq j\leq [s]$ and $0<\delta\leq \gamma-1$,
and
\begin{equation}\label{equa101}
    u_{n}\rightarrow u \quad in \quad L^{p}((-T,T),H^{j,\,2^{\ast}_{j}})
\end{equation}
as $n\rightarrow \infty$, for all $0\leq j\leq [s]$ and $1\leq p< \infty$, where $\frac{1}{2^{\ast}_{j}}=\frac{1}{2}-\frac{s-j}{N}$.
Indeed, given $\eta>0$ and small, by Proposition \ref{prop2}, we have $B^{s-\frac{\eta N}{\rho(\rho+\eta)}}_{\rho+\eta,2}\hookrightarrow H^{j,\,\rho^{\ast}_{j}}$ for all $0\leq j\leq [s]$. If $\eta$ is sufficiently small, $\rho+\eta<\frac{2N}{N-2}$ so that there exists $\gamma_{\eta}$ such that $(\gamma_{\eta},\rho+\eta)$ is an admissible pair, and we deduce from (\ref{weak convergence}) that $u_{n}\rightarrow u$ in $L^{\gamma_{\eta}}((-T,T),H^{j,\,\rho^{\ast}_{j}})$ for all $0\leq j\leq [s]$. Since $\gamma_{\eta}\rightarrow \gamma$ as $\eta\rightarrow 0$, we conclude that (\ref{equa100}) holds. The convergence (\ref{equa101}) can be obtained similarly.

We are to prove that as $n\rightarrow \infty$, for every admissible pair $(q,r)$,
\begin{equation}\label{equa101*}
u_{n}\rightarrow u\,\,\,\, {\rm in} \,\,L^{q}((-T,T),B^{s}_{r,2}(\mathbb{R}^{N})).
\end{equation}

Suppose that this is done then Theorem \ref{th1} follows by iterating this property in order to cover any compact subset of $(-T_{min},T_{max})$ in the subcritical case or a standard compact argument in the critical case. To show (\ref{equa101*}), let us define $I=(-T,T)$, and
$$X_{n}(I):=\|u_{n}(t)-u(t)\|_{L^{\infty}(I,H^{s})}+\|u_{n}(t)-u(t)\|_{L^{\gamma}(I,B^{s}_{\rho,2})}.$$

We will carry out our proof of \eqref{equa101*} by discussing the following four cases respectively:

\noindent Case I.\,\,\,\,$1<s<N/2$ and $\alpha<\frac{4}{N-2s}$;

\noindent Case II.\,\,$1<s<N/2$ and $\alpha=\frac{4}{N-2s}$;

\noindent Case III.\, The borderline case $s=N/2$, $\alpha<\infty$;

\noindent Case IV.\,\,\,$s> N/2$.\\

Let us start with Case I. $1<s<N/2$ and $\alpha<\frac{4}{N-2s}$.\\

First, we will prove a refinement of the convergence \eqref{equa100}.
\begin{lem}\label{lem6}
If we assume further that $\alpha> \frac{4\{s\}}{N-2s}$, then as $n\rightarrow\infty$,
\begin{equation}\label{equa106}
    u_{n}\rightarrow u \quad in \quad L^{\gamma}((-T,T),H^{j,\rho^{\ast}_{j}}(\mathbb{R}^{N}))
\end{equation}
for every $0\leq j\leq [s]$.
\end{lem}
\begin{proof}
To prove Lemma \ref{lem6}, we resort to Strichartz's estimates for non-admissible pairs(see Lemma \ref{lem9}) and the fact $u_{n}\rightarrow u$ in $L^{p}((-T,T),H^{[s],2^{\ast}_{[s]}})$ for every $1\leq p<\infty$(see \eqref{equa101}). One can easily verify that the Dispersive and Strichartz's estimates for $(e^{it\Delta})_{t\in\mathbb{R}}$ and $(e^{it(\Delta-\lambda_{0})})_{t\in\mathbb{R}}$ are the same modulo at most a factor $e^{|\Im\lambda_{0}|T}$(see \cite{O}), thus we may assume $\lambda_{0}=0$ here without loss of generality.

For any $2\leq x<\frac{2N}{N-2}$ and $0<\eta\leq 1$, let $(y,x)$ be an admissible pair, we define index $\tau_{x,\eta}$ by:
\[\frac{1}{\tau_{x,\eta}}=1-\alpha(\frac{1}{x}-\frac{s}{N})-(1-\eta)(\frac{1}{x}-\frac{\{s\}}{N})-\eta(\frac{1}{2}-\frac{\{s\}}{N}).\]
Note that for $\alpha> \frac{4\{s\}}{N-2s}$, $1/\tau_{x,\eta}\rightarrow 1-\alpha(\frac{1}{2}-\frac{s}{N})-(\frac{1}{2}-\frac{\{s\}}{N})<\frac{N-2\{s\}}{2N}$ as $\eta\uparrow 1$ and $x\downarrow 2$, $1/\tau_{x,\eta}\rightarrow 1-\frac{[s]}{N}-(\alpha+1)(\frac{1}{2}-\frac{1}{N}-\frac{s}{N})^{+}>1/\rho^{\ast}_{[s]}>\frac{N-2\{s\}-2}{2N}$ as $\eta\downarrow 0$ and $x\uparrow \min\{\frac{2N}{N-2},\frac{N}{s}\}$, so we can choose $2\leq x=\rho_{1}<\min\{\frac{2N}{N-2},\frac{N}{s}\}$, $y=\gamma_{1}$ and $0<\eta=\eta_{0}\leq 1$ such that $\frac{2N}{N-2\{s\}}<\tau=\tau_{\rho_{1},\eta_{0}}<\rho^{\ast}_{[s]}$, and there exists some $2<\tilde{\tau}<\rho$ such that $\frac{1}{\tau}=\frac{1}{\tilde{\tau}}-\frac{\{s\}}{N}$, which implies $B^{\{s\}}_{\tilde{\tau},2}(\mathbb{R}^{N})\hookrightarrow L^{\tau}(\mathbb{R}^{N})$.

Let $\frac{2}{a}=\frac{N}{2}-\frac{N}{\tau}-\{s\}$ and
$\frac{2}{\tilde{a}}=\frac{N}{2}-\frac{N}{\tau}+\{s\}$. Then
$(a,\tilde{\tau})$ is an admissible pair and
\begin{equation}\label{equa107}
    \frac{1}{\tau'}=\alpha(\frac{1}{\rho_{1}}-\frac{s}{N})+(1-\eta_{0})(\frac{1}{\rho_{1}}-\frac{\{s\}}{N})+\eta_{0}(\frac{1}{2}-\frac{\{s\}}{N}),
\end{equation}
\begin{equation}\label{equa108}
    \frac{1}{\tilde{a}'}=\frac{\alpha+1-\eta_{0}}{\gamma_{1}}+1-\frac{\alpha(N-2s)}{4}.
\end{equation}
Set $p=\frac{4}{4-\alpha(N-2s)}$, then $p\in (1,\infty)$. We recall the formula
\begin{equation*}
\begin{array}{ll}
\||\nabla|^{[s]}(g(u_{n})-g(u))\|_{L^{\tau'}}&\\
\leq \sum_{k=1}^{[s]}\sum_{|\beta_1|+\cdots
+|\beta_k|=[s],|\beta_j|\geq 1}\{\|(g_{z\backslash \bar{z}}^{(k)}(u_n)-g_{z\backslash
\bar{z}}^{(k)}(u))\prod_{j=1}^k\partial^{\beta_j}(u\backslash \bar{u})\|_{L^{\tau'}}&\\
\qquad+\sum_{(w_{1},\cdots,w_{k})}\|g_{z\backslash\bar{z}}^{(k)}(u_{n})\prod_{j=1}^k\partial^{\beta_j}w_j\|_{L^{\tau'}}\},&
\end{array}
\end{equation*}
where all the $w_j$'s are equal to $u_{n},\bar{u}_{n}$ or $u,\bar{u}$, except one which is equal to $u_{n}-u$ or its conjugate. Since $g_{1}$ is of class $\mathcal{C}(\alpha,s)$, we deduce from the above formula, (\ref{class1}), Strichartz's estimates for non-admissible pairs(Lemma \ref{lem9}), the indices and embedding relationships obtained above and H\"{o}lder's inequality that
\begin{equation}\label{nonadmissible-strichartz}
\begin{array}{ll}
  &\||\nabla|^{[s]}(u_{n}-u)\|_{L^{a}(I,L^{\tau})}\leq C\|\phi_{n}-\phi\|_{H^{s}} \\
   &+ CT^{\frac{1-\eta_{0}}{p}}(\|u_{n}\|^{\alpha+1-\eta_{0}}_{L^{\gamma_{1}}(I,B^{s}_{\rho_{1},2})}+
   \|u\|^{\alpha+1-\eta_{0}}_{L^{\gamma_{1}}(I,B^{s}_{\rho_{1},2})})\|\langle\nabla\rangle^{[s]}(u_{n}-u)\|^{\eta_{0}}_{L^{p}(I,L^{\frac{2N}{N-2\{s\}}})}.
\end{array}
\end{equation}
Therefore, from (\ref{nonadmissible-strichartz}) and the following facts,
\[\|u\|_{L^{\gamma_{1}}(I,B^{s}_{\rho_{1},2})}+\sup_{n\geq 1}\|u_{n}\|_{L^{\gamma_{1}}(I,B^{s}_{\rho_{1},2})}< \infty,\]
\[\phi_{n}\rightarrow \phi\,\,\, in \,\,H^{s},\,\,\,\, u_{n}\rightarrow u\,\, in \,\,L^{p}((-T,T),H^{[s],2^{\ast}_{[s]}}),\]
we infer that
\begin{equation*}
    \||\nabla|^{[s]}(u_{n}-u)\|_{L^{a}((-T,T),L^{\tau}(\mathbb{R}^{N}))}\rightarrow 0, \quad as \quad n\rightarrow \infty.
\end{equation*}
Note that $\frac{2}{a}=\frac{N}{2}-\frac{N}{\tau}-\{s\}$ and $\frac{2}{\gamma}=\frac{N}{2}-\frac{N}{\rho^{\ast}_{[s]}}-\{s\}$, then $u_{n}\rightarrow u$ in $L^{\gamma}((-T,T),H^{[s],\rho^{\ast}_{[s]}})$ as $n\rightarrow \infty$ can be obtained from interpolation easily. In fact, we have $u_{n}\rightarrow u$ in $L^{q}((-T,T),H^{[s],\frac{rN}{N-r\{s\}}})$ as $n\rightarrow \infty$ for every $r\in( 2,\frac{2N}{N-2})$ and admissible pair $(q,r)$. Note that by Proposition \ref{prop2}, we have embedding $H^{j,\rho^{\ast}_{j}}(\mathbb{R}^{N})\hookrightarrow H^{[s],\rho^{\ast}_{[s]}}(\mathbb{R}^{N})$ for every $0\leq j\leq[s]$, this closes our proof.
\end{proof}

We will discuss Case I by the following two subcases respectively.
\[
\left\{
\begin{array}{ll}
subcase\,1:\,\,\, 1<s\leq 2 \,\,\, and \,\,\, s<N/2,&\\
\,&\\
subcase\,2:\,\,\, 2<s<N/2.&
\end{array}
 \right.
\]

First, we consider simpler subcase 1: $1<s\leq 2$ and $s<N/2$. \\

By Strichartz's estimate, we have
\begin{equation}\label{stri}
\begin{array}{ll}
 X_{n}&=\|u_{n}(t)-u(t)\|_{L^{\infty}(I,H^{s})}
+\|u_{n}(t)-u(t)\|_{L^{\gamma}(I,B^{s}_{\rho,2})}\\
&\leq
C\|\phi_{n}-\phi\|_{H^{s}}+CT\|u_{n}(t)-u(t)\|_{L^{\infty}(I,H^{s})}+E_{1}+E_{2},
\end{array}
\end{equation}
where $E_1=\|g_{1}(u_{n})-g_{1}(u)\|_{L^{\gamma'}(I,L^{\rho'})}$, $E_2=\|g_{1}(u_{n})-g_{1}(u)\|_{L^{\gamma'}(I,\dot{B}^{s}_{\rho',2})}$.

For $E_1$, we have
\begin{equation}\label{I2-1}
E_1\leq
CT^{\frac{4-\alpha(N-2s)}{4}}(\|u_{n}\|^{\alpha}_{L^{\gamma}(I,{B^{s}_{\rho,2}})}+\|u\|^{\alpha}_{L^{\gamma}(I,{B^{s}_{\rho,2}})})\|u_{n}
-u\|_{L^{\gamma}(I,{L^{\rho}})}.
\end{equation}

As to $E_2$, noting that $\dot{H}^{s,\rho'}\hookrightarrow
\dot{B}^{s}_{\rho',2}$, we have
\begin{equation}\label{II2-1}
\|g_{1}(u_{n})-g_{1}(u)\|_{\dot{B}^{s}_{\rho',2}}\leq
C\||\nabla|^{s}(g_{1}(u_{n})-g_{1}(u))\|_{L^{\rho'}},
\end{equation}
and
\begin{equation}\label{II2-2}
\begin{array}{ll}
\qquad\||\nabla|^{s}(g_{1}(u_{n})-g_{1}(u))\|_{L^{\rho'}}&\\
\qquad =\||\nabla|^{\{s\}}[\nabla
u_{n}g_{1z}(u_{n})+\nabla\bar{u}_{n}g_{1\bar{z}}(u_{n})-\nabla
ug_{1z}(u)
-\nabla\bar{u}g_{1\bar{z}}(u)]\|_{L^{\rho'}}&\\
\qquad\leq
\||\nabla|^{\{s\}}[g_{1z}(u_{n})\nabla(u_{n}-u)]\|_{L^{\rho'}}
+\||\nabla|^{\{s\}}[(g_{1z}(u_{n})-g_{1z}(u))\nabla u]\|_{L^{\rho'}} &\\
\qquad\qquad+\||\nabla|^{\{s\}}[g_{1\bar{z}}(u_{n})\nabla(\bar{u}_{n}-\bar{u})]
\|_{L^{\rho'}}+\||\nabla|^{\{s\}}[(g_{1\bar{z}}(u_{n})-g_{1\bar{z}}(u))\nabla \bar{u}]\|_{L^{\rho'}}&\\
\qquad=:A_1+A_2+A_3+A_4.&
\end{array}
\end{equation}

In the sequel we apply Proposition \ref{prop2} and Proposition \ref{prop1} to estimate $A_1,A_2,A_3,A_4$. \\

Since $g_{1}$ is of class $\mathcal{C}(\alpha,s)$, we deduce from (\ref{class1}), Proposition \ref{prop2} and
Proposition \ref{prop1} that
\begin{eqnarray*}
% \nonumber to remove numbering (before each equation)
A_1&=&\||\nabla|^{\{s\}}[g_{1z}(u_{n})\nabla(u_{n}-u)]\|_{L^{\rho'}}\\
&\leq&C\{\||u_{n}|^{\alpha}\|_{L^{\rho^{\ast}_{0}/\alpha}}\|u_{n}
-u\|_{\dot{H}^{s,\rho}}+\|g_{1z}(u_{n})\|_{\dot{H}^{\{s\},\sigma}}
\|\nabla(u_{n}-u)\|_{L^{\rho^{\ast}_{1}}}\},
\end{eqnarray*}
where $1/\rho'=1/\rho+\alpha/\rho^{\ast}_{0}=1/\sigma+1/\rho^{\ast}_{1}$.

It follows that $1/\sigma=(\alpha-1)/\rho^{\ast}_{0}+1/t$, where $t=\frac{N\rho}{N-\rho}$. Then by (\ref{class1}) and Lemma \ref{lem8}, we infer that
\begin{equation}\label{A1}
\|g_{1z}(u_{n})\|_{\dot{H}^{\{s\},\sigma}}\leq
C_{\{s\},\sigma,N}\|u_{n}\|^{\alpha-1}_{L^{\rho^{\ast}_{0}}}\|u_{n}\|_{H^{\{s\},t}}.
\end{equation}

Thus we deduce from $B^{s}_{\rho,2}\hookrightarrow L^{\rho^{\ast}_{0}}$, $B^{s}_{\rho,2}\hookrightarrow H^{s,\rho}$, $B^{s}_{\rho,2}\hookrightarrow H^{1,\rho^{\ast}_{1}}$ and $B^{s}_{\rho,2}\hookrightarrow H^{\{s\},t}$ that
\begin{equation}\label{A2}
A_1\leq C\|u_{n}\|^{\alpha}_{B^{s}_{\rho,2}}\|u_{n}-u\|_{B^{s}_{\rho,2}}.
\end{equation}

From Proposition \ref{prop2} and Proposition \ref{prop1}, we have
\begin{equation}\label{B1}
\begin{array}{ll}
 A_2=\||\nabla|^{\{s\}}[(g_{1z}(u_{n})-g_{1z}(u))\nabla u]\|_{L^{\rho'}}&\\
\quad\leq
C\||\nabla|^{\{s\}}(g_{1z}(u_{n})-g_{1z}(u))\|_{L^{\sigma}}\|\nabla
u\|_{L^{\rho^{\ast}_{1}}}\\
\qquad+C\|g_{1z}(u_{n})-g_{1z}(u)\|_{L^{\rho^{\ast}_{0}/\alpha}}\|u\|_{\dot{H}^{s,\rho}},&\\
\end{array}
\end{equation}
where $1/\rho'=1/\rho+\alpha/\rho^{\ast}_{0}=1/\sigma+1/\rho^{\ast}_{1}$.

It follows easily from (\ref{class1}) and $B^{s}_{\rho,2}\hookrightarrow L^{\rho^{\ast}_{0}}$ that
\[\|g_{1z}(u_{n})-g_{1z}(u)\|_{L^{\rho^{\ast}_{0}/\alpha}}\leq C(\|u_{n}\|^{\alpha-1}_{B^{s}_{\rho,2}}
+\|u\|^{\alpha-1}_{B^{s}_{\rho,2}})\|u_{n}-u\|_{B^{s}_{\rho,2}}.\]

By Hardy-Littlewood-Sobolev inequality(which will be called HLS inequality simply hereafter) and Proposition \ref{prop2}, we have
\begin{equation*}
    \||\nabla|^{\{s\}}(g_{1z}(u_{n})-g_{1z}(u))\|_{L^{\sigma}}\leq C\|\nabla(g_{1z}(u_{n})-g_{1z}(u))\|_{L^{\nu}},
\end{equation*}
where $1/\nu=1/\sigma-\frac{s-2}{N}=1/\rho^{\ast}_{1}+(\alpha-1)/\rho^{\ast}_{0}$.

Since $g_{1}$ is of class $\mathcal{C}(\alpha,s)$, by (\ref{class1}), we have pointwise estimate:
\begin{equation*}
|\nabla(g_{1z}(u_{n})-g_{1z}(u))|\leq |D^{2}(g_{1}(u_{n})-g_{1}(u))|\cdot|\nabla u|+C|\nabla(u_{n}-u)|\cdot|u_{n}|^{\alpha-1}.
 \end{equation*}
So by (\ref{class2}), $B^{s}_{\rho,2}\hookrightarrow L^{\rho^{\ast}_{0}}$, $B^{s}_{\rho,2}\hookrightarrow H^{1,\rho^{\ast}_{1}}$ and H\"{o}lder's inequality, we get: \\
(i) if $g_{1}(u)$ is a polynomial in $u$ and $\bar{u}$, or if not, we assume further that $\alpha\geq [s]+1=2$, then
\begin{equation}\label{B-4}
\begin{array}{ll}
&\|\nabla(g_{1z}(u_{n})-g_{1z}(u))\|_{L^{\nu}}\\
\leq& C\{(\|u_{n}\|^{\alpha-2}_{L^{\rho^{\ast}_{0}}}+
\|u\|^{\alpha-2}_{L^{\rho^{\ast}_{0}}})\|u_{n}-u\|_{L^{\rho^{\ast}_{0}}}\|\nabla
u\|_{L^{\rho^{\ast}_{1}}}+\|\nabla(u_{n}-u)\|_{L^{\rho^{\ast}_{1}}}
\|u_{n}\|^{\alpha-1}_{L^{\rho^{\ast}_{0}}}\}\\
\leq & C(\|u_{n}\|^{\alpha-1}_{B^{s}_{\rho,2}}
+\|u\|^{\alpha-1}_{B^{s}_{\rho,2}})\|u_{n}-u\|_{B^{s}_{\rho,2}};
\end{array}
\end{equation}
(ii) if $g_{1}$ is not a polynomial and $1=[s]<\alpha<2$, then
\begin{equation}\label{B5}
\begin{array}{ll}
&\|\nabla(g_{1z}(u_{n})-g_{1z}(u))\|_{L^{\nu}} \\
\leq
&C\{\|D^{2}(g_{1}(u_{n})-g_{1}(u))\|_{L^{\frac{\rho^{\ast}_{0}}{\alpha-1}}}\|u\|_{B^{s}_{\rho,2}}
+\|u_{n}-u\|_{B^{s}_{\rho,2}}\|u_{n}\|^{\alpha-1}_{B^{s}_{\rho,2}}\}.
\end{array}
\end{equation}

By \eqref{B1}, \eqref{B-4} and \eqref{B5} we have the following estimates for $A_2$: \\
(i) if $g_{1}(u)$ is a polynomial in $u$ and $\bar{u}$, or if not, we assume further that $\alpha\geq [s]+1=2$, then
\begin{equation}\label{B6}
A_2\leq
C(\|u_{n}\|^{\alpha}_{B^{s}_{\rho,2}}+\|u\|^{\alpha}_{B^{s}_{\rho,2}})\|u_{n}-u\|_{B^{s}_{\rho,2}};
\end{equation}
(ii) otherwise,
\begin{equation}\label{B7}
A_2\leq C\{\|D^{2}(g_{1}(u_{n})-g_{1}(u))\|_{L^{\frac{\rho^{\ast}_{0}}{\alpha-1}}}\|u\|^{2}_{B^{s}_{\rho,2}}+
(\|u_{n}\|^{\alpha}_{B^{s}_{\rho,2}}+\|u\|^{\alpha}_{B^{s}_{\rho,2}})\|u_{n}-u\|_{B^{s}_{\rho,2}}\}.
\end{equation}

The estimate of $A_3$ is similar to $A_1$. From (\ref{class1}), Proposition \ref{prop2} and Proposition \ref{prop1}, we obtain
\begin{equation}\label{C1}
\begin{array}{ll}
A_3=\||\nabla|^{\{s\}}[g_{1\bar{z}}(u_{n})\nabla(\bar{u}_{n}-\bar{u})]\|_{L^{\rho'}}&\\
\qquad\leq C
\{\||u_{n}|^{\alpha}\|_{L^{\rho^{\ast}_{0}/\alpha}}\|u_{n}-u\|_{\dot{H}^{s,\rho}}
+\|g_{1\bar{z}}(u_{n})\|_{\dot{H}^{\{s\},\sigma}}\|\nabla(u_{n}-u)\|_{L^{\rho^{\ast}_{1}}}\},&
\end{array}
\end{equation}
where $1/\rho'=1/\rho+\alpha/\rho^{\ast}_{0}=1/\sigma+1/\rho^{\ast}_{1}$.

Then completely similar to the estimate of $A_1$, we deduce from (\ref{class1}), Lemma \ref{lem8} and the embedding relationships that
\begin{equation}\label{C3}
A_3\leq C\|u_{n}\|^{\alpha}_{B^{s}_{\rho,2}}\|u_{n}-u\|_{B^{s}_{\rho,2}}.
\end{equation}

Now we turn to $A_4$, from Proposition \ref{prop2} and Proposition \ref{prop1}, we get
\begin{equation}
\begin{array}{ll}
A_4=\||\nabla|^{\{s\}}[(g_{1\bar{z}}(u_{n})-g_{1\bar{z}}(u))\nabla
\bar{u}]\|_{L^{\rho'}}&\\
\qquad\leq
C\||\nabla|^{\{s\}}(g_{1\bar{z}}(u_{n})-g_{1\bar{z}}(u))\|_{L^{\sigma}}\|\nabla
u\|_{L^{\rho^{\ast}_{1}}}&\\
\qquad\qquad+C\|g_{1\bar{z}}(u_{n})-g_{1\bar{z}}(u)\|_{L^{\rho^{\ast}_{0}/\alpha}}\|u\|_{\dot{H}^{s,\rho}},&\\
\end{array}
\end{equation}
where $1/\rho'=1/\rho+\alpha/\rho^{\ast}_{0}=1/\sigma+1/\rho^{\ast}_{1}$.

Then completely similar to the estimate of $A_2$, we deduce from (\ref{class1}), (\ref{class2}), Proposition \ref{prop2}, H\"{o}lder's and HLS inequalities that \\
(i) if $g_{1}(u)$ is a polynomial in $u$ and $\bar{u}$, or if not, we assume further that $\alpha\geq [s]+1=2$, then
\begin{equation}\label{B8}
A_4\leq
C(\|u_{n}\|^{\alpha}_{B^{s}_{\rho,2}}+\|u\|^{\alpha}_{B^{s}_{\rho,2}})\|u_{n}-u\|_{B^{s}_{\rho,2}};
\end{equation}
(ii) otherwise,
\begin{equation}\label{B9}
A_4\leq C\{\|D^{2}(g_{1}(u_{n})-g_{1}(u))\|_{L^{\frac{\rho^{\ast}_{0}}{\alpha-1}}}\|u\|^{2}_{B^{s}_{\rho,2}}+
(\|u_{n}\|^{\alpha}_{B^{s}_{\rho,2}}+\|u\|^{\alpha}_{B^{s}_{\rho,2}})\|u_{n}-u\|_{B^{s}_{\rho,2}}\}.
\end{equation}

Finally by the estimates of $A_1$,$A_2$,$A_3$ and $A_4$, applying H\"{o}lder's inequality on time, we get estimate of $E_2=\|g_{1}(u_{n})-g_{1}(u)\|_{L^{\gamma'}(I,\dot{B}^{s}_{\rho',2})}$ as following:\\
(i)\, if $g_{1}(u)$ is a polynomial in $u$ and $\bar{u}$, or if not, we assume further that $\alpha\geq [s]+1=2$, then
\[E_2\leq CT^{\sigma}(\|u_{n}\|^{\alpha}_{L^{\gamma}(I,{B^{s}_{\rho,2}})}
+\|u\|^{\alpha}_{L^{\gamma}(I,{B^{s}_{\rho,2}})})\|u_{n}-u\|_{L^{\gamma}(I,{B^{s}_{\rho,2}})};\]
(ii)\, if $g_{1}$ is not a polynomial and $1=[s]<\alpha<2$, then
\[E_2\leq C\{T^{\sigma}(\|u_{n}\|^{\alpha}_{L^{\gamma}(I,{B^{s}_{\rho,2}})}+\|u\|^{\alpha}_{L^{\gamma}(I,{B^{s}_{\rho,2}})})\|u_{n}-u\|_{L^{\gamma}(I,{B^{s}_{\rho,2}})}
+\|K(u_{n},u)\|_{L^{\gamma'}(I)}\};\]
where $\sigma=\frac{4-\alpha(N-2s)}{4}$ and the Kato's remainder term $K(u_{n},u)$ is defined by
\begin{equation}\label{Kato}
    K(u_{n},u)=\|u\|^{2}_{B^{s}_{\rho,2}}\|D^{2}g_{1}(u_{n})-D^{2}g_{1}(u)\|_{L^{\frac{\rho^{\ast}_{0}}{\alpha-1}}}.
\end{equation}

As $\|u_{n}\|_{L^{\gamma}(I,{B^{s}_{\rho,2}})}$ is bounded, there exists
\[M=\|u\|_{L^{\gamma}(I,B^{s}_{\rho,2})}+\sup_{n\geq 1}\|u_{n}\|_{L^{\gamma}(I,B^{s}_{\rho,2})}< \infty\]
such that $\|u\|_{L^{\gamma}(I,{B^{s}_{\rho,2}})}\leq M$ and $\|u_{n}\|_{L^{\gamma}(I,{B^{s}_{\rho,2}})}\leq M$.
Inserting the estimates of $E_1$ and $E_2$ into the original Strichartz estimate (\ref{stri}), we get the following two cases:\\
Case (i) \,if $g_{1}(u)$ is a polynomial in $u$ and $\bar{u}$, or if not, we assume further that $\alpha\geq [s]+1=2$, then
\begin{equation}\label{X1}
    X_{n}\leq C\|\phi_{n}-\phi\|_{H^{s}}+CTX_{n}+CT^{\sigma}M^{\alpha}X_{n};
\end{equation}
Case (ii) \,if $g_{1}$ is not a polynomial and $1=[s]<\alpha<2$, then
\begin{equation}\label{X2}
    X_{n}\leq C\|\phi_{n}-\phi\|_{H^{s}}+C\{TX_{n}+T^{\sigma}M^{\alpha}X_{n}+\|K(u_{n},u)\|_{L^{\gamma'}((-T,T))}\};
\end{equation}
where $\sigma=\frac{4-\alpha(N-2s)}{4}$. Here we need assumption $\alpha<\frac{4}{N-2s}$ such that $\sigma>0$.

For case $(i)$, we can choose $T$ sufficiently small so that $C(T+T^{\sigma}M^{\alpha})\leq\frac{1}{2}$, and deduce from \eqref{X1} that as $n\rightarrow \infty$,
\[X_{n}=\|u_{n}(t)-u(t)\|_{L^{\infty}(I,H^{s})}+\|u_{n}(t)-u(t)\|_{L^{\gamma}(I,B^{s}_{\rho,2})}\leq C\|\phi_{n}-\phi\|_{H^{s}}\rightarrow 0,\]
so the solution flow is locally Lipschitz.

For case $(ii)$, we can use the Kato's method, but we will apply a direct way of proving based on Lemma \ref{lem6} first. By (\ref{Kato}), (\ref{class2}) and applying H\"{o}lder's estimates, we obtain
\begin{equation}\label{Kato1}
    \|K(u_{n},u)\|_{L^{\gamma'}(I)}\leq CT^{\sigma}M^{2}\|u_{n}-u\|^{\alpha-1}_{L^{\gamma}(I,L^{\rho^{\ast}_{0}})},
\end{equation}
where $\sigma=\frac{4-\alpha(N-2s)}{4}$. Using \eqref{Kato1} combined with \eqref{X2}, we get
\begin{equation}\label{Xn2}
X_{n}\leq
C\{\|\phi_{n}-\phi\|_{H^{s}}+\|u_{n}-u\|^{\alpha-1}_{L^{\gamma}((-T,T),L^{\rho^{\ast}_{0}})}\},
\end{equation}
for $T$ sufficiently small. As a consequence of Lemma \ref{lem6}, we know if $\alpha> \frac{4\{s\}}{N-2s}$, then
\[\|u_{n}-u\|_{L^{\gamma}((-T,T),L^{\rho^{\ast}_{0}})}\rightarrow 0, \,\,\, as \,\,\, n\rightarrow\infty,\]
which yields the desired convergence: $X_{n}\rightarrow 0$ as $n\rightarrow \infty$, by \eqref{Xn2}.

To deal with the cases that $1=[s]<\alpha\leq\frac{4\{s\}}{N-2s}$, we resort to Kato's method(see \cite{CFZ}). It follows from (\ref{class1}), \eqref{Kato} and H\"{o}lder's estimates that
\begin{equation}\label{control kato}
    K(u_{n},u)\leq C\|u\|^{2}_{B^{s}_{\rho,2}}(\|u_{n}\|^{\alpha-1}_{L^{\rho^{\ast}_{0}}}+\|u\|^{\alpha-1}_{L^{\rho^{\ast}_{0}}}).
\end{equation}
We can also deduce from \eqref{X2} that
\[X_{n}\leq C\{\|\phi_{n}-\phi\|_{H^{s}}+\|K(u_{n},u)\|_{L^{\gamma'}(-T,T)}\}\]
for $T$ sufficiently small. Therefore, to prove $X_n\to 0$ it suffices to show that
\begin{equation}\label{Kato3}
    \|K(u_{n},u)\|_{L^{\gamma'}(-T,T)}\rightarrow 0, \quad as \quad n\rightarrow
    \infty.
\end{equation}
To show (\ref{Kato3}) we argue by contradiction. Suppose that there
exist $\varepsilon>0$, and a subsequence, which is still  denoted by
$\{u_{n}\}_{n\geq 1}$, such that
\begin{equation}\label{Kato4}
    \|K(u_{n},u)\|_{L^{\gamma'}(-T,T)}\geq \varepsilon.
\end{equation}
Since $\frac{(\alpha-1)\gamma}{\gamma-3}<\gamma$,  using
(\ref{equa100})(taking $j=0$ there) and by possibly extracting a subsequence, we can
assume that $u_{n}\rightarrow u$ a.e. on $(-T,T)\times
\mathbb{R}^{N}$ and that
\[\|u_{n+1}-u_{n}\|_{L^{\frac{(\alpha-1)\gamma}{\gamma-3}}(I,L^{\rho^{\ast}_{0}}(\mathbb{R}^{N}))}<2^{-n}.\]
So if we let $\omega=|u_{1}|+\sum_{n=1}^{\infty}|u_{n+1}-u_{n}|$, then $\omega\in
L^{\frac{(\alpha-1)\gamma}{\gamma-3}}(I,L^{\rho^{\ast}_{0}}(\mathbb{R}^{N}))$
and $|u_{n}|\leq \omega$ a.e. on $(-T,T)\times
\mathbb{R}^{N}$. Consequently, by \eqref{control kato} and Young's inequality, there is a constant $C>0$ such that
\[K(u_{n},u)^{\gamma'}\leq C(\|\omega\|^{\frac{(\alpha-1)\gamma}{\gamma-3}}_{L^{\rho^{\ast}_{0}}}+\|u\|^{\frac{(\alpha-1)\gamma}{\gamma-3}}_{L^{\rho^{\ast}_{0}}}
+\|u\|^{\gamma}_{B^{s}_{\rho,2}})\in L^{1}((-T,T)).\]

On the other hand, note that by (\ref{class1}), we have
\[|D^{2}g_{1}(u_{n})-D^{2}g_{1}(u))|^{\frac{\rho^{\ast}_{0}}{\alpha-1}}\leq C(|u_{n}|^{\rho^{\ast}_{0}}+|u|^{\rho^{\ast}_{0}});\]
note also that since $g_{1}\in C^{2}(\mathbb{C},\mathbb{C})$, $D^{2}g_{1}$ is continuous, hence we can deduce from $u_{n}\rightarrow u$ that $D^{2}g_{1}(u_{n})\rightarrow D^{2}g_{1}(u)$. Therefore, by dominated convergence theorem and a standard contradiction argument as above, after possibly extracting a subsequence again, we can obtain from $u_{n}(t)\rightarrow u(t)$ in $L^{\rho^{\ast}_{0}}(\mathbb{R}^{N})$ for a.a. $t\in(-T,T)$ that $\|D^{2}g_{1}(u_{n})-D^{2}g_{1}(u)\|_{L^{\frac{\rho^{\ast}_{0}}{\alpha-1}}}\rightarrow0$ for a.a. $t\in(-T,T)$, which implies that $K(u_{n},u)\rightarrow 0$ as $n\rightarrow \infty$ for a.a. $t\in(-T,T)$ by (\ref{Kato}).

Now we can use dominated convergence theorem to infer that
\[\|K(u_{n},u)\|_{L^{\gamma'}(-T,T)}\rightarrow 0,\]
which contradicts (\ref{Kato4}). Hence (\ref{Kato3}) holds.

Thus we have proved $X_{n}\rightarrow 0$ as $n\rightarrow \infty$ for $T$ sufficiently small in both case (i) and (ii). The convergence for arbitrary admissible pair $(q,r)$ follows from Strichartz's estimates. The conclusions $(i)$ and $(ii)$ of Theorem \ref{th1} follow by iterating this property to cover any compact subset of $(-T_{min},T_{max})$. Moreover, if $g_{1}(u)$ is a polynomial in $u$ and $\bar{u}$, or if not, we assume further that $\alpha\geq [s]+1=2$, then the continuous dependence is locally Lipschitz. \\

Now we turn to subcase 2: \,$2<s<\frac{N}{2}$.\\

Applying Strichartz's estimate to the difference equation (\ref{eq8}), we find that the main difficulty is still how to estimate $\|g_{1}(u_{n})-g_{1}(u)\|_{L^{\gamma'}(I,\dot{B}^{s}_{\rho',2})}$, or more precisely, how to estimate $\||\nabla|^{s}(g_{1}(u_{n})-g_{1}(u))\|_{L^{\rho'}}$. By the $L^{p}$ boundedness of Riesz transforms, we have the following formula
\begin{equation*}
\begin{array}{ll}
\||\nabla|^{s}(g_{1}(u_{n})-g_{1}(u))\|_{L^{\rho'}}
=\||\nabla|^{s-[s]} [|\nabla|^{[s]}(g_{1}(u_{n})-g_{1}(u))]\|_{L^{\rho'}} &\\
\leq \sum_{k=1}^{[s]}\sum_{|\beta_1|+\cdots+|\beta_k|=[s],|\beta_j|\geq 1}\||\nabla|^{\{s\}}[ (g_{1,z\backslash\bar{z}}^{(k)}(u_n)-g_{1,z\backslash\bar{z}}^{(k)}(u))\prod_{j=1}^k\partial^{\beta_j}(u\backslash\bar{u})\\
\qquad+g_{1,z\backslash \bar{z}}^{(k)}(u_{n})\sum_{(w_{1},\cdots,w_{k})}\prod_{j=1}^k\partial^{\beta_j}w_j]\|_{L^{\rho'}},&\\
\end{array}
\end{equation*}
where $k\in \{1,\cdots,[s]\}$ and the $\beta_{j}$'s are multi-indices such that $[s]=|\beta_{1}|+\cdots+|\beta_{k}|$ and $|\beta_{j}|\geq 1$ for $j=1,\cdots,k$, all the $w_j$'s are equal to $u_{n},\bar{u}_{n}$ or $u,\bar{u}$, except one which is equal to $u_{n}-u$ or its conjugate. \\
We deduce from the above formula and Proposition \ref{prop1} that
\begin{eqnarray*}
&&\||\nabla|^{s}(g_{1}(u_{n})-g_{1}(u))\|_{L^{\rho'}}\\
&\leq&C\sum_{k=1}^{[s]}\sum_{\sum_{j=1}^{k}|\beta_j|=[s],|\beta_j|\geq1}
\{\||\nabla|^{\{s\}}(g_{1,z\backslash\bar{z}}^{(k)}(u_n)-g_{1,z\backslash\bar{z}}^{(k)}(u))\|_{L^{\sigma_{1}}}
\|\prod_{j=1}^k\partial^{\beta_j}u\|_{L^{t_{1}}}\\
&&+\|g_{1,z\backslash\bar{z}}^{(k)}(u_n)-g_{1,z\backslash\bar{z}}^{(k)}(u)\|_{L^{\sigma_{2}}}
\||\nabla|^{\{s\}}(\prod_{j=1}^k\partial^{\beta_j}u)\|_{L^{t_{2}}}\\
&&+\||\nabla|^{\{s\}}(g_{1,z\backslash\bar{z}}^{(k)}(u_{n}))\|_{L^{\sigma_{3}}}
\sum_{(w_{1},\cdots,w_{k})}\|\prod_{j=1}^k\partial^{\beta_j}w_j\|_{L^{t_{3}}}\\
&&+\|g_{1,z\backslash\bar{z}}^{(k)}(u_{n})\|_{L^{\sigma_{4}}}
\sum_{(w_{1},\cdots,w_{k})}\||\nabla|^{\{s\}}(\prod_{j=1}^k\partial^{\beta_j}w_j)\|_{L^{t_{4}}}\}\\
&=:&B_1+B_2+B_3+B_4,
\end{eqnarray*}
where the exponents $\sigma_{i}$, $t_{i}$($i=1,\cdots,4$) depend on $k$ and satisfy $1/\rho'=1/\sigma_{1}+1/t_{1}=1/\sigma_{2}+1/t_{2}=1/\sigma_{3}+1/t_{3}=1/\sigma_{4}+1/t_{4}$ with $1/t_{1}=1/t_{3}=\sum_{j=1}^{k}1/\rho^{\ast}_{|\beta_{j}|}$, $\frac{1}{\sigma_{2}}=\frac{1}{\sigma_{4}}=\frac{\alpha+1-k}{\rho^{\ast}_{0}}$. \\

First let us estimate $B_1$. Note that $1/t_{1}=\sum_{j=1}^{k}1/\rho^{\ast}_{|\beta_{j}|}$, by using H\"{o}lder's inequality and Proposition \ref{prop2}, we get for any $k\in\{1,\cdots,[s]\}$,
\[\|\prod_{j=1}^k\partial^{\beta_j}u\|_{L^{t_{1}}}\leq C\prod_{j=1}^k \|u\|_{H^{|\beta_{j}|,\rho^{\ast}_{|\beta_{j}|}}}\leq C \|u\|^{k}_{B^{s}_{\rho,2}}.\]

By HLS inequality(see Lemma \ref{HLS}) and Proposition \ref{prop2}, we have
\begin{equation}\label{equa109}
    \||\nabla|^{\{s\}}(g_{1,z\backslash \bar{z}}^{(k)}(u_n)-g_{1,z\backslash \bar{z}}^{(k)}(u))\|_{L^{\sigma_{1}}}
\leq C\|\nabla(g_{1,z\backslash \bar{z}}^{(k)}(u_n)-g_{1,z\backslash \bar{z}}^{(k)}(u))\|_{L^{\nu_{1}}},
\end{equation}
where $1/\nu_{1}=1/\sigma_{1}-\frac{s-[s]-1}{N}=1/\rho'-\sum_{j=1}^{k}1/\rho^{\ast}_{|\beta_{j}|}-\frac{s-[s]-1}{N}
=\frac{1}{\rho^{\ast}_{1}}+\frac{\alpha-k}{\rho^{\ast}_{0}}$.\\
Since $g_{1}$ is of class $\mathcal{C}(\alpha,s)$, by (\ref{class1}), we have pointwise estimate:
\[|\nabla(g_{1,z\backslash \bar{z}}^{(k)}(u_n)-g_{1,z\backslash \bar{z}}^{(k)}(u))|\leq C|u_{n}|^{\alpha-k}|\nabla(u_{n}-u)|+|D^{k+1}(g_{1}(u_{n})-g_{1}(u))|\cdot|\nabla u|,\]
for any $k\in\{1,\cdots,[s]\}$.

Noting that $\frac{1}{\nu_{1}}=\frac{1}{\rho^{\ast}_{1}}+\frac{\alpha-k}{\rho^{\ast}_{0}}$, by (\ref{equa109}), (\ref{class1}), (\ref{class2}), $B^{s}_{\rho,2}\hookrightarrow L^{\rho^{\ast}_{0}}$, $B^{s}_{\rho,2}\hookrightarrow H^{1,\rho^{\ast}_{1}}$ and H\"{o}lder's inequality, we get two cases:\\
(i) if $g_{1}(u)$ is a polynomial in $u$ and $\bar{u}$, or if not, we assume further that $\alpha\geq [s]+1$, then for any $k\in\{1,\cdots,[s]\}$,
\[\||\nabla|^{\{s\}}(g_{1,z\backslash \bar{z}}^{(k)}(u_n)-g_{1,z\backslash \bar{z}}^{(k)}(u))\|_{L^{\sigma_{1}}}\leq C(\|u_{n}\|^{\alpha-k}_{B^{s}_{\rho,2}}+\|u\|^{\alpha-k}_{B^{s}_{\rho,2}})\|u_{n}-u\|_{B^{s}_{\rho,2}},\]
thus
\[B_1\leq C(\|u_{n}\|^{\alpha}_{B^{s}_{\rho,2}}+\|u\|^{\alpha}_{B^{s}_{\rho,2}})\|u_{n}-u\|_{B^{s}_{\rho,2}};\]
(ii) if $g_{1}$ is not a polynomial and $[s]<\alpha<[s]+1$, then for any $k\in\{1,\cdots,[s]-1\}$, we have
\[\||\nabla|^{\{s\}}(g_{1,z\backslash \bar{z}}^{(k)}(u_n)-g_{1,z\backslash \bar{z}}^{(k)}(u))\|_{L^{\sigma_{1}}}\leq C(\|u_{n}\|^{\alpha-k}_{B^{s}_{\rho,2}}+\|u\|^{\alpha-k}_{B^{s}_{\rho,2}})\|u_{n}-u\|_{B^{s}_{\rho,2}},\]
except when $k=[s]$, we have
\begin{eqnarray*}
% \nonumber to remove numbering (before each equation)
\begin{array}{ll}
\||\nabla|^{\{s\}}(g_{1,z\backslash \bar{z}}^{(k)}(u_n)-g_{1,z\backslash \bar{z}}^{(k)}(u))\|_{L^{\sigma_{1}}}&\\
\qquad\leq C\|\nabla(g_{1,z\backslash \bar{z}}^{(k)}(u_n)-g_{1,z\backslash \bar{z}}^{(k)}(u))\|_{L^{\nu_{1}}}&\\
\qquad\leq
C\{(\|u_{n}\|^{\alpha-[s]}_{B^{s}_{\rho,2}}+\|u\|^{\alpha-[s]}_{B^{s}_{\rho,2}})\|u_{n}-u\|_{B^{s}_{\rho,2}}&\\
\qquad\qquad+\|D^{[s]+1}(g_{1}(u_{n})-g_{1}(u))\|_{L^{\frac{\rho^{\ast}_{0}}{\alpha-[s]}}}\|u\|_{B^{s}_{\rho,2}}\},&\\
\end{array}
\end{eqnarray*}
thus
\[
\begin{array}{ll}
B_1\leq C\{(\|u_{n}\|^{\alpha}_{B^{s}_{\rho,2}}
+\|u\|^{\alpha}_{B^{s}_{\rho,2}})\|u_{n}-u\|_{B^{s}_{\rho,2}}
&\\
\qquad\qquad+\|D^{[s]+1}(g_{1}(u_{n})-g_{1}(u))\|_{L^{\frac{\rho^{\ast}_{0}}{\alpha-[s]}}}\|u\|^{[s]+1}_{B^{s}_{\rho,2}}\}.
\end{array}
\]

Let us now estimate $B_2$. Since $g_{1}$ is of class $\mathcal{C}(\alpha,s)$, by (\ref{class1}) and the mean value theorem, we get for any $k\in\{1,\cdots,[s]\}$,
\[|g_{1,z\backslash \bar{z}}^{(k)}(u_n)-g_{1,z\backslash \bar{z}}^{(k)}(u)|\leq C(|u_{n}|^{\alpha-k}+|u|^{\alpha-k})|u_{n}-u|.\]
Noting that $\frac{1}{\sigma_{2}}=\frac{\alpha+1-k}{\rho^{\ast}_{0}}$, by H\"{o}lder's inequality and $B^{s}_{\rho,2}\hookrightarrow L^{\rho^{\ast}_{0}}$, we infer that for any $k\in\{1,\cdots,[s]\}$,
\begin{equation}\label{eq15}
% \nonumber to remove numbering (before each equation)
\|g_{1,z\backslash \bar{z}}^{(k)}(u_n)-g_{1,z\backslash\bar{z}}^{(k)}(u)\|_{L^{\sigma_{2}}}\leq C(\|u_{n}\|^{\alpha-k}_{B^{s}_{\rho,2}}+\|u\|^{\alpha-k}_{B^{s}_{\rho,2}})\|u_{n} -u\|_{B^{s}_{\rho,2}}.
\end{equation}

When $k=1$, we deduce from Proposition \ref{prop2} that $t_{2}=\rho$, thus we have
\begin{equation}\label{eq k1}
\||\nabla|^{\{s\}}(\prod_{j=1}^k\partial^{\beta_j}u)\|_{L^{t_{2}}}\leq C\|u\|_{B^{s}_{\rho,2}}.
\end{equation}

When $k\in\{2,\cdots,[s]\}$, by HLS inequality and Proposition \ref{prop2}, we have
\[\||\nabla|^{\{s\}}(\prod_{j=1}^k\partial^{\beta_j}u)\|_{L^{t_{2}}}\leq C\|\nabla(\prod_{j=1}^k\partial^{\beta_j}u)\|_{L^{\nu_{2}}},\]
where $1/\nu_{2}=1/t_{2}-\frac{s-[s]-1}{N}=1/\rho'-\frac{\alpha+1-k}{\rho^{\ast}_{0}}-\frac{s-[s]-1}{N}=\sum_{j=1}^{k}1/\rho^{\ast}_{|\beta_{j}|}+1/N$.\\
For fixed $i=1,\cdots,k$, define $\hat{\beta}_{ji}=\beta_{j}$ for
$j\neq i$,  and $|\hat{\beta}_{ji}|=|\beta_{j}|+1$ for $j=i$. For a
fixed $i=1,\cdots,k$,
$1/\nu_{2}=\sum_{j=1}^{k}1/\rho^{\ast}_{|\hat{\beta}_{ji}|}$, by
applying H\"{o}lder's inequality, we deduce that for any $k\in\{2,\cdots,[s]\}$,
\begin{equation}\label{eq16}
% \nonumber to remove numbering (before each equation)
\begin{array}{ll}\||\nabla|^{\{s\}}(\prod_{j=1}^k\partial^{\beta_j}u)\|_{L^{t_{2}}}
\leq C\|\nabla(\prod_{j=1}^k\partial^{\beta_j}u)\|_{L^{\nu_{2}}}\leq C\sum_{i=1}^{k}\|\prod_{j=1}^k\partial^{\hat{\beta}_{ji}}u\|_{L^{\nu_{2}}}&\\
\qquad\leq C\sum_{i=1}^{k}\prod_{j=1}^k\|u\|_{H^{\hat{\beta}_{ji},\rho^{\ast}_{|\hat{\beta}_{ji}|}}}\leq C\|u\|^{k}_{B^{s}_{\rho,2}}.&
\end{array}
\end{equation}
Here we have used the fact that $B^{s}_{\rho,2}\hookrightarrow H^{\hat{\beta}_{ji},\rho^{\ast}_{|\hat{\beta}_{ji}|}}$, which can been deduced from Proposition \ref{prop2}. From (\ref{eq15}), (\ref{eq k1}) and (\ref{eq16}), we get
\[B_2\leq C(\|u_{n}\|^{\alpha}_{B^{s}_{\rho,2}}+\|u\|^{\alpha}_{B^{s}_{\rho,2}})\|u_{n}-u\|_{B^{s}_{\rho,2}}.\]

Next we estimate $B_3$. First note that $1/t_{3}=\sum_{j=1}^{k}1/\rho^{\ast}_{|\beta_{j}|}$, by using H\"{o}lder's inequality and Proposition \ref{prop2}, we get
\[\|\prod_{j=1}^k\partial^{\beta_j}w_j\|_{L^{t_{3}}}\leq C\prod_{j=1}^k \|w_j\|_{H^{|\beta_{j}|,\rho^{\ast}_{|\beta_{j}|}}}\leq C\prod_{j=1}^k\|w_j\|_{B^{s}_{\rho,2}}.\]
Noting that the $w_j$'s are equal to $u_{n},\bar{u}_{n}$ or
$u,\bar{u}$, except one which is equal to $u_{n}-u$ or its
conjugate, we have for any $k\in\{1,\cdots,[s]\}$,
\begin{equation}\label{eq17}
  \|\prod_{j=1}^k\partial^{\beta_j}w_j\|_{L^{t_{3}}}\leq C(\|u_{n}\|^{k-1}_{B^{s}_{\rho,2}}
+\|u\|^{k-1}_{B^{s}_{\rho,2}})\|u_{n}-u\|_{B^{s}_{\rho,2}}.
\end{equation}

Since $1/\sigma_{3}=1/\rho'-\sum_{j=1}^{k}1/\rho^{\ast}_{|\beta_{j}|}=\frac{\alpha-k}{\rho^{\ast}_{0}}+\frac{1}{\rho}-\frac{[s]}{N}$,
we deduce from (\ref{class1}) and Lemma \ref{lem8} that for any $k\in\{1,\cdots,[s]\}$,
\[\||\nabla|^{\{s\}}(g_{1,z\backslash \bar{z}}^{(k)}(u_{n}))\|_{L^{\sigma_{3}}}\leq C\|u_{n}\|^{\alpha-k}_{L^{\rho^{\ast}_{0}}}\|u_{n}\|_{\dot{H}^{\{s\},t}},\]
where $t=\frac{\rho N}{N-[s]\rho}$. It's easy to see that $B^{s}_{\rho,2}\hookrightarrow \dot{H}^{\{s\},t}$, thus
\begin{equation}\label{eq18}
 \||\nabla|^{\{s\}}(g_{1,z\backslash \bar{z}}^{(k)}(u_{n}))\|_{L^{\sigma_{3}}}\leq C\|u_{n}\|^{\alpha+1-k}_{B^{s}_{\rho,2}}.
\end{equation}
From (\ref{eq17}) and (\ref{eq18}), we get
\[B_3\leq C(\|u_{n}\|^{\alpha}_{B^{s}_{\rho,2}}
+\|u\|^{\alpha}_{B^{s}_{\rho,2}})\|u_{n}-u\|_{B^{s}_{\rho,2}}.\]

Finally we estimate $B_4$. Note that $\frac{1}{\sigma_{4}}=\frac{\alpha+1-k}{\rho^{\ast}_{0}}$, we deduce from $(\ref{class1})$ and H\"{o}lder's inequality that for any $k\in\{1,\cdots,[s]\}$,
\begin{equation}\label{eq19}
\|g_{1,z\setminus \bar{z}}^{(k)}(u_{n})\|_{L^{\sigma_{4}}}\leq C\|u_{n}\|^{\alpha+1-k}_{L^{\rho^{\ast}_{0}}}\leq C\|u_{n}\|^{\alpha+1-k}_{B^{s}_{\rho,2}}.
\end{equation}

When $k=1$, we deduce from Proposition \ref{prop2} that $t_{4}=\rho$, note also that $|\beta_{1}|=[s]$ and $w_{1}$ is equal to $u_{n}-u$ or its conjugate, so we have
\begin{equation}\label{equa k1}
\||\nabla|^{\{s\}}(\prod_{j=1}^k\partial^{\beta_j}w_j)\|_{L^{t_{4}}}\leq C\|u_{n}-u\|_{B^{s}_{\rho,2}}.
\end{equation}

When $k\in\{2,\cdots,[s]\}$, by HLS inequality and Proposition \ref{prop2}, we have
\[\||\nabla|^{\{s\}}(\prod_{j=1}^k\partial^{\beta_j}w_j)\|_{L^{t_{4}}}\leq C\|\nabla(\prod_{j=1}^k\partial^{\beta_j}w_j)\|_{L^{\nu_{4}}},\]
where $1/\nu_{4}=1/t_{4}-\frac{s-[s]-1}{N}=1/\rho'-\frac{\alpha+1-k}{\rho^{\ast}_{0}}-\frac{s-[s]-1}{N}=\sum_{j=1}^{k}1/\rho^{\ast}_{|\beta_{j}|}+1/N$.\\
For fixed $i=1,\cdots,k$, define $\hat{\beta}_{ji}=\beta_{j}$ for
$j\neq i$,  and $|\hat{\beta}_{ji}|=|\beta_{j}|+1$ for $j=i$. For a
fixed $i=1,\cdots,k$,
$1/\nu_{4}=\sum_{j=1}^{k}1/\rho^{\ast}_{|\hat{\beta}_{ji}|}$, by
applying H\"{o}lder's inequality, we deduce that for any $k\in\{2,\cdots,[s]\}$,
\begin{equation}\label{eq20}
\begin{array}{ll}
\||\nabla|^{\{s\}}(\prod_{j=1}^k\partial^{\beta_j}w_j)\|_{L^{t_{4}}}\leq C\|\nabla(\prod_{j=1}^k\partial^{\beta_j}w_j)\|_{L^{\nu_{4}}}&\\
\qquad\leq C\sum_{i=1}^{k}\|\prod_{j=1}^k\partial^{\hat{\beta}_{ji}}w_j\|_{L^{\nu_{4}}}\leq C\sum_{i=1}^{k}\prod_{j=1}^k\|w_j\|_{H^{\hat{\beta}_{ji},\rho^{\ast}_{|\hat{\beta}_{ji}|}}}&\\
\qquad\qquad\leq C\prod_{j=1}^k\|w_j\|_{B^{s}_{\rho,2}}\leq C(\|u_{n}\|^{k-1}_{B^{s}_{\rho,2}}+\|u\|^{k-1}_{B^{s}_{\rho,2}})\|u_{n}-u\|_{B^{s}_{\rho,2}}.&\\
\end{array}
\end{equation}
In the above estimate we have used the fact $B^{s}_{\rho,2}\hookrightarrow H^{\hat{\beta}_{ji},\rho^{\ast}_{|\hat{\beta}_{ji}|}}$, which is deduced from Proposition \ref{prop2}. From (\ref{eq19}), (\ref{equa k1}) and (\ref{eq20}), we get
\begin{equation}\label{B4}
B_4\leq C(\|u_{n}\|^{\alpha}_{B^{s}_{\rho,2}}+\|u\|^{\alpha}_{B^{s}_{\rho,2}})\|u_{n}-u\|_{B^{s}_{\rho,2}}.
\end{equation}

Combining the estimates for $B_1$, $B_2$, $B_3$ and $B_4$, we have the following two cases for the estimate of $\||\nabla|^{s}(g_{1}(u_{n})-g_{1}(u))\|_{L^{\rho'}}$:

(i) if $g_{1}(u)$ is a polynomial in $u$ and $\bar{u}$, or if not, we assume further that $\alpha\geq [s]+1$, then
\[\||\nabla|^{s}(g_{1}(u_{n})-g_{1}(u))\|_{L^{\rho'}} \leq
C(\|u_{n}\|^{\alpha}_{B^{s}_{\rho,2}}
+\|u\|^{\alpha}_{B^{s}_{\rho,2}})\|u_{n}-u\|_{B^{s}_{\rho,2}}.\]

(ii) if $g_{1}$ is not a polynomial and $[s]<\alpha<[s]+1$, then
\begin{equation}\label{nonlinear}
\begin{array}{ll}
 \||\nabla|^{s}(g_{1}(u_{n})-g_{1}(u))\|_{L^{\rho'}}\leq
C\{(\|u_{n}\|^{\alpha}_{B^{s}_{\rho,2}}
+\|u\|^{\alpha}_{B^{s}_{\rho,2}})\|u_{n}-u\|_{B^{s}_{\rho,2}}&\\
\qquad\qquad\qquad\qquad\qquad+\|D^{[s]+1}(g_{1}(u_{n})-g_{1}(u))\|_{L^{\frac{\rho^{\ast}_{0}}{\alpha-[s]}}}\|u\|^{[s]+1}_{B^{s}_{\rho,2}}\}.&
\end{array}
\end{equation}
We can see that the difference between two nonlinear interactions can be estimated by a Lipschitz term plus a lower order term, the Kato's remainder term $K(u_{n},u)$ can be defined by
\begin{equation}\label{Kato2}
    K(u_{n},u)=\|D^{[s]+1}g_{1}(u_{n})-D^{[s]+1}g_{1}(u)\|_{L^{\frac{\rho^{\ast}_{0}}{\alpha-[s]}}}\|u\|^{[s]+1}_{B^{s}_{\rho,2}}.
\end{equation}

Since $\|u_{n}\|_{L^{\gamma}(I,{B^{s}_{\rho,2}})}$ is bounded, there exists
\[M=\|u\|_{L^{\gamma}(I,B^{s}_{\rho,2})}+\sup_{n\geq 1}\|u_{n}\|_{L^{\gamma}(I,B^{s}_{\rho,2})}< \infty\]
such that $\|u\|_{L^{\gamma}(I,{B^{s}_{\rho,2}})}\leq M$ and $\|u_{n}\|_{L^{\gamma}(I,{B^{s}_{\rho,2}})}\leq M$. Thus, by using H\"{o}lder's inequality on time and the original Strichartz estimate to the difference equation, we
can obtain estimate for $X_n$ by considering the following two cases respectively: \\
Case (i) \,\, if $g_{1}(u)$ is a polynomial in $u$ and $\bar{u}$, or if not, we assume further that $\alpha\geq [s]+1$, then
\begin{equation}\label{X3}
    X_n\leq C\|\phi_{n}-\phi\|_{H^{s}}+CTX_n+CT^{\sigma}M^{\alpha}X_n;
\end{equation}
Case (ii) \, if $g_{1}$ is not a polynomial and $[s]<\alpha<[s]+1$, then
\begin{equation}\label{X4}
 X_n\leq C\|\phi_{n}-\phi\|_{H^{s}}+C\{TX_n+T^{\sigma}M^{\alpha}X_n+\|K(u_{n},u)\|_{L^{\gamma'}(-T,T)}\};
\end{equation}
where $\sigma=\frac{4-\alpha(N-2s)}{4}$. Here we need assumption $\alpha<\frac{4}{N-2s}$ such that $\sigma>0$.

For case $(i)$, we can choose $T$ sufficiently small so that $C(T+T^{\sigma}M^{\alpha})\leq\frac{1}{2}$, and deduce from \eqref{X3} that as $n\rightarrow \infty$,
\[X_{n}=\|u_{n}(t)-u(t)\|_{L^{\infty}(I,H^{s})}+\|u_{n}(t)-u(t)\|_{L^{\gamma}(I,B^{s}_{\rho,2})}\leq C\|\phi_{n}-\phi\|_{H^{s}}\rightarrow 0,\]
so the solution flow is locally Lipschitz.

For case $(ii)$, we can use the Kato's method, but we will apply a direct way of proving based on Lemma \ref{lem6} first. By (\ref{Kato2}), (\ref{class2}) and applying H\"{o}lder's estimates, we obtain
\begin{equation}\label{KATO}
\|K(u_{n},u)\|_{L^{\gamma'}(I)}\leq CT^{\sigma}M^{[s]+1}\|u_{n}-u\|^{\alpha-[s]}_{L^{\gamma}(I,L^{\rho^{\ast}_{0}})},
\end{equation}
where $\sigma=\frac{4-\alpha(N-2s)}{4}$. Using \eqref{KATO} combined with \eqref{X4}, we get
\begin{equation}\label{Xn1}
X_{n}\leq C\{\|\phi_{n}-\phi\|_{H^{s}}+\|u_{n}-u\|^{\alpha-[s]}_{L^{\gamma}((-T,T),L^{\rho^{\ast}_{0}})}\},
\end{equation}
for $T$ sufficiently small. As a consequence of Lemma \ref{lem6}, we know if $\alpha> \frac{4\{s\}}{N-2s}$, then
\[\|u_{n}-u\|_{L^{\gamma}((-T,T),L^{\rho^{\ast}_{0}})}\rightarrow 0, \,\,\, as \,\,\, n\rightarrow\infty,\]
which yields the desired convergence: $X_{n}\rightarrow 0$ as $n\rightarrow \infty$, by \eqref{Xn1}.

To deal with the cases that $[s]<\alpha\leq\frac{4\{s\}}{N-2s}$, we resort to Kato's method.  It follows from (\ref{class1}), \eqref{Kato2} and H\"{o}lder's estimates that
\begin{equation}\label{control remainder}
K(u_{n},u)\leq C\|u\|^{[s]+1}_{B^{s}_{\rho,2}}(\|u_{n}\|^{\alpha-[s]}_{L^{\rho^{\ast}_{0}}}+\|u\|^{\alpha-[s]}_{L^{\rho^{\ast}_{0}}}).
\end{equation}
We can also deduce from \eqref{X4} that
\[X_{n}\leq C\{\|\phi_{n}-\phi\|_{H^{s}}+\|K(u_{n},u)\|_{L^{\gamma'}(-T,T)}\}\]
for $T$ sufficiently small. Therefore, to show  $X_{n}\to 0$ it suffices to show that
\begin{equation}\label{Kato5}
    \|K(u_{n},u)\|_{L^{\gamma'}(-T,T)}\rightarrow 0, \,\,\, as\,\,\, n\rightarrow \infty.
\end{equation}
To show (\ref{Kato5}) we argue by contradiction. Suppose that there exist
$\varepsilon>0$, and a subsequence,
which is still denoted by $\{u_{n}\}_{n\geq 1}$, such that
\begin{equation}\label{Kato6}
    \|K(u_{n},u)\|_{L^{\gamma'}(-T,T)}\geq \varepsilon.
\end{equation}
Since $\frac{(\alpha-[s])\gamma}{\gamma-[s]-2}<\gamma$, using
(\ref{equa100})(taking $j=0$ there) and by possibly extracting a subsequence, we can
assume that $u_{n}\rightarrow u$ a.e. on $(-T,T)\times\mathbb{R}^{N}$ and that
\[\|u_{n+1}-u_{n}\|_{L^{\frac{(\alpha-[s])\gamma}{\gamma-[s]-2}}(I,L^{\rho^{\ast}_{0}}(\mathbb{R}^{N}))}<2^{-n}.\]
As before(subcase 1), if we set $\omega=|u_{1}|+\sum_{n=1}^{\infty}|u_{n+1}-u_{n}|$, then $\omega\in
L^{\frac{(\alpha-[s])\gamma}{\gamma-[s]-2}}(I,L^{\rho^{\ast}_{0}}(\mathbb{R}^{N}))$ and $|u_{n}|\leq \omega$ a.e. on $(-T,T)\times\mathbb{R}^{N}$. Therefore, by \eqref{control remainder} and Young's inequality, there is a constant $C>0$ such that
\[K(u_{n},u)^{\gamma'}\leq C(\|\omega\|^{\frac{(\alpha-[s])\gamma}{\gamma-[s]-2}}_{L^{\rho^{\ast}_{0}}}+\|u\|^{\frac{(\alpha-[s])\gamma}{\gamma-[s]-2}}_{L^{\rho^{\ast}_{0}}}
+\|u\|^{\gamma}_{B^{s}_{\rho,2}})\in L^{1}((-T,T)).\]

On the other hand, note that by (\ref{class1}), we have
\[|D^{[s]+1}(g_{1}(u_{n})-g_{1}(u))|^{\frac{\rho^{\ast}_{0}}{\alpha-[s]}}\leq C(|u_{n}|^{\rho^{\ast}_{0}}+|u|^{\rho^{\ast}_{0}});\]
note also that since $g_{1}\in C^{[s]+1}(\mathbb{C},\mathbb{C})$, $D^{[s]+1}g_{1}$ is continuous, hence we can deduce from $u_{n}\rightarrow u$ that $D^{[s]+1}g_{1}(u_{n})\rightarrow D^{[s]+1}g_{1}(u)$. Therefore, by dominated convergence theorem and a standard contradiction argument as above, after possibly extracting a subsequence again, we can obtain from $u_{n}(t)\rightarrow u(t)$ in $L^{\rho^{\ast}_{0}}(\mathbb{R}^{N})$ for a.a. $t\in(-T,T)$ that $\|D^{[s]+1}(g_{1}(u_{n})-g_{1}(u))\|_{L^{\frac{\rho^{\ast}_{0}}{\alpha-[s]}}}\rightarrow0$ for a.a. $t\in(-T,T)$, which implies that $K(u_{n},u)\rightarrow 0$ as $n\rightarrow \infty$ for a.a. $t\in(-T,T)$ by (\ref{Kato2}).

Now we can use dominated convergence theorem to infer that
$$\|K(u_{n},u)\|_{L^{\gamma'}(-T,T)}\rightarrow 0,$$
which contradicts (\ref{Kato6}). Hence (\ref{Kato5}) holds.

Thus we have proved $X_{n}\rightarrow 0$ as $n\rightarrow \infty$ if $T$ is sufficiently small in both case (i) and (ii). The convergence for arbitrary admissible pair $(q,r)$ follows from Strichartz's estimates. The conclusions $(i)$ and $(ii)$ of Theorem \ref{th1} follow by iterating this property to cover any compact subset of $(-T_{min},T_{max})$. Moreover, if $g_{1}(u)$ is a polynomial in $u$ and $\bar{u}$, or if not, we assume further that $\alpha\geq [s]+1$, then the continuous dependence is locally Lipschitz. \\
\\
Case II. \,$1<s<N/2$ and $\alpha=\frac{4}{N-2s}$.\\

One can easily verify that the Dispersive and Strichartz's estimates for $(e^{it\Delta})_{t\in\mathbb{R}}$ and $(e^{it(\Delta-\lambda_{0})})_{t\in\mathbb{R}}$ are the same modulo at most a factor $e^{|\Im\lambda_{0}|T}$(see \cite{O}), thus we may assume $\lambda_{0}=0$ here without loss of generality, then the nonlinearity $g=g_{1}$ is of class $\mathcal{C}(\alpha,s)$(see Definition \ref{class}).

From (\ref{class1}), Proposition \ref{prop2}, Lemma \ref{alternative} and H\"{o}lder's estimates, it follows that
\[\|g(u)\|_{L^{\gamma'}(I,B^{s}_{\rho',2})}\leq C\|u\|^{\alpha+1}_{L^{\gamma}(I,B^{s}_{\rho,2})},\] and
\[\|g(u_{n})-g(u)\|_{L^{\gamma'}(I,L^{\rho'})}\leq C(\|u_{n}\|^{\alpha}_{L^{\gamma}(I,B^{s}_{\rho,2})}
+\|u\|^{\alpha}_{L^{\gamma}(I,B^{s}_{\rho,2})})\|u_{n}-u\|_{L^{\gamma}(I,L^{\rho})}.\]
Then one can solve the Cauchy problem $(1.1)$ by using the fixed point argument in the set
$$E=\{u\in
L^{\gamma}(I,B^{s}_{\rho,2}(\mathbb{R}^{N}));
\|u\|_{L^{\gamma}(I,B^{s}_{\rho,2})}\leq \eta\}$$
equipped with the distance $d(u,v)=\|u-v\|_{L^{\gamma}(I,L^{\rho})}$, where $\eta>0$ is sufficiently small. By applying Strichartz's estimates in Besov spaces, we get
\[\|u\|_{L^{q}(I,B^{s}_{r,2})}\leq \|e^{i\cdot \Delta}\phi\|_{L^{q}(I,B^{s}_{r,2})}+C\|u\|^{\alpha+1}_{L^{\gamma}(I,B^{s}_{\rho,2})},\]

\[\|u_{n}\|_{L^{q}(I,B^{s}_{r,2})}\leq \|e^{i\cdot \Delta}\phi_{n}\|_{L^{q}(I,B^{s}_{r,2})}+C\|u_{n}\|^{\alpha+1}_{L^{\gamma}(I,B^{s}_{\rho,2})},\]
and
\begin{equation*}
\begin{array}{ll}
  \|u_{n}-u\|_{L^{q}(I,L^{r})}\leq \|e^{i\cdot \Delta}(\phi_{n}-\phi)\|_{L^{q}(I,L^{r})}&\\
  \qquad\qquad\qquad\qquad+ C(\|u_{n}\|^{\alpha}_{L^{\gamma}(I,B^{s}_{\rho,2})}+\|u\|^{\alpha}_{L^{\gamma}(I,B^{s}_{\rho,2})})
  \|u_{n}-u\|_{L^{\gamma}(I,L^{\rho})},&
  \end{array}
\end{equation*}
where $I=(-T,T)$, $(q,r)$ is an arbitrary admissible pair and $C$ is a positive constant independent of $T$ and $n$. Considering $K$ larger than the constant $C$ appearing in the above three estimates for the particular choice of the admissible pair $(q,r)=(\gamma,\rho)$, there exists $\delta>0$ small enough such that
\begin{equation}\label{delta0}
    K(4\delta)^{\alpha}<1/2.
\end{equation}

Since by Strichartz's estimates in Besov spaces, for any admissible pair $(q,r)$, $\|e^{i\cdot\Delta}\phi\|_{L^{q}(I,B^{s}_{r,2})}\leq C\|\phi\|_{H^{s}}$, we have
\[\|e^{i\cdot\Delta}\phi_{n}\|_{L^{q}(I,B^{s}_{r,2})}\leq \|e^{i\cdot\Delta}\phi\|_{L^{q}(I,B^{s}_{r,2})}+C\|\phi_{n}-\phi\|_{H^{s}}.\]
Thus by dominated convergence theorem, if we choose $T>0$ sufficiently small such that $\|e^{i\cdot\Delta}\phi\|_{L^{\gamma}(I,B^{s}_{\rho,2})}<\delta$, then we also have
\[\|e^{i\cdot\Delta}\phi_{n}\|_{L^{\gamma}(I,B^{s}_{\rho,2})}<\delta\]
for $n$ sufficiently large. So we can construct solutions $u$ and $u_{n}$($n$ large enough) on interval $I$ by using the fixed point argument in the set
$$E=\{u\in L^{\gamma}(I,B^{s}_{\rho,2}(\mathbb{R}^{N})); \|u\|_{L^{\gamma}(I,B^{s}_{\rho,2})}\leq 2\delta\}$$
equipped with the distance $d(u,v)=\|u-v\|_{L^{\gamma}(I,L^{\rho})}$, which means $T_{max}(\phi)$, $T_{min}(\phi)>T$ and $T_{max}(\phi_{n}),T_{min}(\phi_{n})>T$ for $n$ large. Applying continuity arguments, we deduce easily that
\begin{equation}\label{continuity1}
    max\{\|u\|_{L^{\gamma}(I,B^{s}_{\rho,2})},\|u_{n}\|_{L^{\gamma}(I,B^{s}_{\rho,2})}\}\leq 2\delta,
\end{equation}
for all sufficiently large $n$. Furthermore, by Strichartz's estimates in Besov spaces, we see that given any admissible pair $(q,r)$, there exists a constant $C_{q,r}$ such that
\begin{equation}\label{continuity2}
    max\{\|u\|_{L^{q}(I,B^{s}_{r,2})},\|u_{n}\|_{L^{q}(I,B^{s}_{r,2})}\}\leq C_{q,r}\delta,
\end{equation}
for all sufficiently large $n$.

\newtheorem{clm}[thm]{Claim}

\begin{clm}\label{claim1}(Non-integer case). We claim that by possibly choosing $T$ smaller, for any non-integer $s\in(1,N/2)$, we have as $n\rightarrow\infty$,
\begin{equation}\label{critical}
u_{n}\rightarrow u \quad in \quad L^{\gamma}((-T,T),H^{j,\rho^{\ast}_{j}}(\mathbb{R}^{N}))
\end{equation}
for every $0\leq j\leq [s]$.
\end{clm}
\begin{proof}
To prove our claim, fixing $\varepsilon>0$ such that \[0<\varepsilon
N<min\left\{1-\{s\},\frac{2}{N-2s}\right\},\] and for $N\geq 3$, we
define non-admissible pair $(a,\tau)$ by
\begin{equation}\label{nonadmissible}
    \tau=\frac{2N}{N-2+2\varepsilon N},\,\, a=\frac{2}{1-\{s\}-\varepsilon N},\,\, \tilde{a}=\frac{2}{1+\{s\}-\varepsilon N},
\end{equation}
and admissible pair $(\gamma_{1},\rho_{1})$ by
\begin{equation}\label{admissible}
    \gamma_{1}=\frac{4}{\varepsilon N(N-2s)},\,\, \rho_{1}=\frac{2}{1-\varepsilon (N-2s)}.
\end{equation}

It follows from \eqref{nonadmissible} and \eqref{admissible} that
\begin{equation}
    \frac{1}{a}+\frac{1}{\tilde{a}}=N(\frac{1}{2}-\frac{1}{\tau}),\,\, \frac{2}{a}=N(\frac{1}{2}-\frac{1}{\tau})-\{s\},\,\,
     \frac{1}{\tilde{a}'}=\frac{\alpha}{\gamma_{1}}+\frac{1}{a}.
\end{equation}
Since $2<\rho_{1}<N/s$, we can define indices $(\rho_{1})^{\ast}_{j}$ by $\frac{1}{(\rho_{1})^{\ast}_{j}}=\frac{1}{\rho_{1}}-\frac{s-j}{N}$ for any $0\leq j\leq [s]$. Consequently, we have the embedding $B^{s}_{\rho_{1},2}\hookrightarrow H^{j,(\rho_{1})^{\ast}_{j}}$ for any $0\leq j\leq [s]$ and $\frac{1}{\tau'}=\frac{\alpha}{(\rho_{1})^{\ast}_{0}}+\frac{1}{\tau}$. Set $\tilde{\tau}=\frac{2N}{N-2+2\{s\}+2\varepsilon N}$, then $(a,\tilde{\tau})$ is an admissible pair and $B^{\{s\}}_{\tilde{\tau},2}\hookrightarrow L^{\tau}$.

We recall the formula
\begin{equation*}
\begin{array}{ll}
\||\nabla|^{[s]}(g(u_{n})-g(u))\|_{L^{\tau'}}&\\
\leq\sum_{k=1}^{[s]}\sum_{|\beta_1|+\cdots+|\beta_k|=[s],|\beta_j|\geq1}
\{\|(g_{z\backslash\bar{z}}^{(k)}(u_n)-g_{z\backslash\bar{z}}^{(k)}(u))\prod_{j=1}^k\partial^{\beta_j}(u\backslash\bar{u})\|_{L^{\tau'}}&\\
\qquad+\sum_{(w_{1},\cdots,w_{k})}\|g_{z\backslash\bar{z}}^{(k)}(u_{n})\prod_{j=1}^k\partial^{\beta_j}w_j\|_{L^{\tau'}}\},&
\end{array}
\end{equation*}
where all the $w_j$'s are equal to $u_{n},\bar{u}_{n}$ or $u,\bar{u}$, except one which is equal to $u_{n}-u$ or its conjugate. Since $g$ is of class $\mathcal{C}(\alpha,s)$, we can deduce from the above formula, (\ref{class1}), Strichartz's estimates for non-admissible pairs(Lemma \ref{lem9}), the indices and embedding relationships obtained above and H\"{o}lder's inequality that
\begin{eqnarray*}
% \nonumber to remove numbering (before each equation)
  &&\||\nabla|^{[s]}(u_{n}-u)\|_{L^{a}(I,L^{\tau})}\\
  &\leq& C\|e^{i\cdot\Delta}(\phi_{n}-\phi)\|_{L^{a}(I,B^{s}_{\tilde{\tau},2})}+C\||\nabla|^{[s]}(g(u_{n})-g(u))\|_{L^{\tilde{a}'}(I,L^{\tau'})}\\
&\leq&
C\|\phi_{n}-\phi\|_{H^{s}}+C(\|u_{n}\|^{\alpha}_{L^{\gamma_{1}}(I,B^{s}_{\rho_{1},2})}+\|u\|^{\alpha}_{L^{\gamma_{1}}(I,B^{s}_{\rho_{1},2})})
\||\nabla|^{[s]}(u_{n}-u)\|_{L^{a}(I,L^{\tau})}.
\end{eqnarray*}
Thus by (\ref{continuity2}), we have for all sufficiently large $n$,
\[\||\nabla|^{[s]}(u_{n}-u)\|_{L^{a}(I,L^{\tau})}\leq C\|\phi_{n}-\phi\|_{H^{s}}+C\delta^{\alpha}\||\nabla|^{[s]}(u_{n}-u)\|_{L^{a}(I,L^{\tau})}.\]
Therefore, by possibly choosing $\delta$ smaller, and choosing $T$ smaller if necessary, such that
\begin{equation}\label{delta1}
    C\delta^{\alpha}<1/2,
\end{equation}
we deduce that as $n\rightarrow \infty$, $\||\nabla|^{[s]}(u_{n}-u)\|_{L^{a}(I,L^{\tau})}\rightarrow 0$.

Since $\frac{2}{\gamma}=\frac{N}{2}-\frac{N}{\rho^{\ast}_{[s]}}-\{s\}$ and $\frac{2}{a}=\frac{N}{2}-\frac{N}{\tau}-\{s\}$, so by (\ref{continuity2}), there always exists some admissible pair $(q,r)$ and index $r^{\ast}_{[s]}$ defined by $\frac{1}{r^{\ast}_{[s]}}=\frac{1}{r}-\frac{\{s\}}{N}$ such that either $\rho\in (r,\tilde{\tau}]$ or $\rho\in[\tilde{\tau},r)$, and
\[\sup_{n\geq 1}\||\nabla|^{[s]}(u_{n}-u)\|_{L^{q}(I,L^{r^{\ast}_{[s]}})}<\infty,\]
then $\||\nabla|^{[s]}(u_{n}-u)\|_{L^{\gamma}(I,L^{\rho^{\ast}_{[s]}})}\rightarrow0$, as $n\rightarrow \infty$ follows from interpolation. Hence, by the convergence of Lebesgue norms, we have $u_{n}\rightarrow u$ in $L^{\gamma}(I,H^{[s],\rho^{\ast}_{[s]}})$ as $n\rightarrow \infty$. Note that by Proposition \ref{prop2}, we have embedding $H^{j,\rho^{\ast}_{j}}(\mathbb{R}^{N})\hookrightarrow H^{[s],\rho^{\ast}_{[s]}}(\mathbb{R}^{N})$ for every $0\leq j\leq[s]$, this closes our proof.
\end{proof}

\begin{clm}\label{claim2}(Integer case). We claim that for any integer $s\in(1,N/2)$, assume $g(u)$ is not a polynomial in $u$ and $\bar{u}$ and $s-1=[s]<\alpha<[s]+1=s$, then by possibly choosing $T$ smaller, we have as $n\rightarrow\infty$,
\begin{equation}\label{critical-1}
u_{n}\rightarrow u \quad in \quad L^{\gamma}((-T,T),H^{j,\rho^{\ast}_{j}}(\mathbb{R}^{N}))
\end{equation}
for every $0\leq j\leq [s]=s-1$.
\end{clm}
\begin{proof}
We will prove our claim in the spirit of Tao and Visan \cite{Tao2}. First, we should note that the restrictions in Claim \ref{claim2} admit only one case, that is, $s=2$, $N=7$ and $\alpha=4/3$, so the nonlinearity $g$ is of class $\mathcal{C}(\frac{4}{3},2)$. Second, we have to avoid taking full derivatives of order 2, since this is what turns the nonlinearity from Lipschitz into just H\"{o}lder continuous of order $\frac{1}{3}$, we need to take fewer than $\frac{4}{3}$ but more than $1$ derivatives instead. To this end, we will use the norms $U=U(I\times\mathbb{R}^{7})$ and $V=V(I\times\mathbb{R}^{7})$ defined by
\begin{equation}\label{exotic norms 1}
    \|u\|_{U}=(\sum_{k\in\mathbb{Z}}2^{\frac{k}{2}}\|\Delta_{k}(\nabla u)\|_{L^{16}(I,L^{8/3}(\mathbb{R}^{7}))}^{2})^{1/2},
\end{equation}
\begin{equation}\label{exotic norms 2}
    \|f\|_{V}=(\sum_{k\in\mathbb{Z}}2^{\frac{k}{2}}\|\Delta_{k}(\nabla f)\|_{L^{\frac{16}{3}}(I,L^{8/5}(\mathbb{R}^{7}))}^{2})^{1/2},
\end{equation}
which require roughly $\frac{5}{4}$ degrees of differentiability. Choosing a particular admissible pair $(q_{1},r_{1})=(\frac{32}{3},\frac{112}{53})$, we will also need the Strichartz space $\dot{W}=\dot{W}(I\times\mathbb{R}^{N})$ defined as the closure of the test functions under the norm
\begin{equation}\label{Strichartz norms}
    \|u\|_{\dot{W}}=\|u\|_{L^{q_{1}}(I,\dot{B}^{2}_{r_{1},2}(\mathbb{R}^{7}))}.
\end{equation}
As the Littlewood-Paley operators commute with the free evolution, we can deduce from Strichartz's estimates for non-admissible pairs(Lemma \ref{lem9}) that, there exists a constant $C$ independent of $I$ such that for any $f\in V$,
\begin{equation*}
    \|\Delta_{k}(\nabla \int^{t}_{t_{0}}e^{i(t-s)\Delta}f(s)ds)\|_{L^{16}(I,L^{8/3}(\mathbb{R}^{7}))}\leq C\|\Delta_{k}(\nabla f)\|_{L^{\frac{16}{3}}(I,L^{8/5}(\mathbb{R}^{7}))}.
\end{equation*}
Squaring the above inequality, multiplying by $2^{\frac{k}{2}}$, and summing over all integer $k$'s, we obtain the exotic Strichartz estimate, that is, for any $f\in V$,
\begin{equation}\label{exotic Strichartz estimate}
    \|\int^{t}_{t_{0}}e^{i(t-s)\Delta}f(s)ds\|_{U}\leq C\|f\|_{V}.
\end{equation}

The proof of Lemma 3.4 in \cite{Tao2} can be adapted to give us the following frequency-localized nonlinear estimate
\begin{eqnarray}\label{f-l nonlinear}
&& \max\{\|\Delta_{k}(g^{(2)}_{z\backslash\bar{z}}(w)v\nabla u)\|_{L^{\frac{16}{3}}_{t}L^{\frac{8}{5}}_{x}},\|\Delta_{k}(g^{(2)}_{z\backslash\bar{z}}(w)\nabla v u)\|_{L^{\frac{16}{3}}_{t}L^{\frac{8}{5}}_{x}}\} \\
\nonumber &\leq&C\|w\|_{\dot{W}}^{1/3}\|v\|_{\dot{W}}\sum_{l\in\mathbb{Z}}\min\{1,2^{\frac{l-k}{3}}\}\|\Delta_{l}(\nabla u)\|_{L^{16}_{t}L^{8/3}_{x}}
\end{eqnarray}
for any $k\in\mathbb{Z}$, where all space-time norms are on $I\times\mathbb{R}^{7}$.

Indeed, by scaling, we only need to show (\ref{f-l nonlinear}) for $k=0$. To this end, by Littlewood-Paley decomposition(see Section 2), we first slip $w$ into its low-frequency part $w_{lo}=S_{<0}w$ and high-frequency part $w_{hi}=S_{\geq0}w$. Since $g$ is of class $\mathcal{C}(\frac{4}{3},2)$, we deduce from (\ref{class2}) that
\begin{equation}\label{pointwise-H2}
    g^{(2)}_{z\backslash\bar{z}}(w)=g^{(2)}_{z\backslash\bar{z}}(w_{lo})+O(|w_{hi}|^{\frac{1}{3}}),
\end{equation}
where $O$ denotes the Landau's symbol. By Minkowski's inequality, we have
\begin{equation}\label{F-L11}
\|\Delta_{0}(g^{(2)}_{z\backslash\bar{z}}(w)v\nabla u)\|_{L^{\frac{16}{3}}_{t}L^{\frac{8}{5}}_{x}}\leq \sum_{l\in\mathbb{Z}}\|\Delta_{0}(g^{(2)}_{z\backslash\bar{z}}(w)v\Delta_{l}(\nabla u))\|_{L^{\frac{16}{3}}_{t}L^{\frac{8}{5}}_{x}}.
\end{equation}
To bound the right-hand side of (\ref{F-L11}), by using (\ref{class1}), (\ref{Strichartz norms}) and H\"{o}lder's inequality, note that $\dot{B}^{2}_{r_{1},2}(\mathbb{R}^{7})\hookrightarrow L_{x}^{\frac{16}{3}}(\mathbb{R}^{7})$, we get
\begin{equation}\label{F-L12}
\sum_{l\geq-2}\|\Delta_{0}(g^{(2)}_{z\backslash\bar{z}}(w)v\Delta_{l}(\nabla u))\|_{L^{\frac{16}{3}}_{t}L^{\frac{8}{5}}_{x}}\leq C\|w\|_{\dot{W}}^{1/3}\|v\|_{\dot{W}}\sum_{l\geq-2}\|\Delta_{l}(\nabla u)\|_{L^{16}_{t}L^{8/3}_{x}}.
\end{equation}
On the other hand, by H\"{o}lder's and Bernstein estimates(see Lemma \ref{bernstein}), we have
\begin{eqnarray}\label{F-L13}
\nonumber &&\sum_{l\leq-3}\|\Delta_{0}(g^{(2)}_{z\backslash\bar{z}}(w)v\Delta_{l}(\nabla u))\|_{L^{\frac{16}{3}}_{t}L^{\frac{8}{5}}_{x}}\\
&\leq&C\|S_{\geq-1}(g^{(2)}_{z\backslash\bar{z}}(w)v)\|_{L^{8}_{t}L^{\frac{84}{25}}_{x}}\sum_{l\leq-3}\|\Delta_{l}(\nabla u)\|_{L^{16}_{t}L^{\frac{168}{55}}_{x}}\\
\nonumber &\leq&C\|S_{\geq-1}(g^{(2)}_{z\backslash\bar{z}}(w)v)\|_{L^{8}_{t}L^{\frac{84}{25}}_{x}}\sum_{l\leq-3}2^{\frac{l}{3}}
\|\Delta_{l}(\nabla u)\|_{L^{16}_{t}L^{8/3}_{x}}.
\end{eqnarray}
Since by (\ref{pointwise-H2}), we get
\begin{equation}\label{F-L14}
\|S_{\geq-1}(g^{(2)}_{z\backslash\bar{z}}(w)v)\|_{L^{8}_{t}L^{\frac{84}{25}}_{x}}\leq
\|S_{\geq-1}(g^{(2)}_{z\backslash\bar{z}}(w_{lo})v)\|_{L^{8}_{t}L^{\frac{84}{25}}_{x}}+
\|S_{\geq-1}(O(|w_{hi}|^{\frac{1}{3}})v)\|_{L^{8}_{t}L^{\frac{84}{25}}_{x}};
\end{equation}
moreover, by H\"{o}lder's inequality, Bernstein estimates and the embedding $\dot{B}^{2}_{r_{1},2}(\mathbb{R}^{7})\hookrightarrow\dot{H}^{1,\frac{112}{37}}(\mathbb{R}^{7})\hookrightarrow L_{x}^{\frac{16}{3}}(\mathbb{R}^{7})$, we have
\begin{equation}\label{F-L15}
\|S_{\geq-1}(O(|w_{hi}|^{\frac{1}{3}})v)\|_{L^{8}_{t}L^{\frac{84}{25}}_{x}}\leq C\|\nabla(w_{hi})\|^{1/3}_{L^{\frac{32}{3}}_{t}L^{\frac{112}{37}}_{x}}\|v\|_{L^{\frac{32}{3}}_{t}L^{\frac{16}{3}}_{x}}
\leq C\|w\|_{\dot{W}}^{1/3}\|v\|_{\dot{W}};
\end{equation}
therefore, we only have the task of estimating $\|S_{\geq-1}(g^{(2)}_{z\backslash\bar{z}}(w_{lo})v)\|_{L^{8}_{t}L^{\frac{84}{25}}_{x}}$.

One easily verifies that for any spatial function $F$ and $1\leq p\leq\infty$, we have the bound(see \cite{Tao2})
\begin{equation}\label{F-L16}
\|S_{\geq-1}F\|_{L_{x}^{p}}\leq C\sup_{|h|\leq1}\|\tau_{h}F-F\|_{L_{x}^{p}},
\end{equation}
where $\tau_{h}F(x):=F(x-h)$ is the translation operator. Note that $g$ is of class $\mathcal{C}(\frac{4}{3},2)$, by (\ref{class2}), we have
\begin{equation}\label{F-L17}
\tau_{h}g^{(2)}_{z\backslash\bar{z}}(w_{lo})-g^{(2)}_{z\backslash\bar{z}}(w_{lo})=O(|\tau_{h}w_{lo}-w_{lo}|^{\frac{1}{3}}).   \end{equation}
Thus we can deduce from the identity
\begin{equation}\label{F-L18}
    \tau_{h}(fg)-fg=(\tau_{h}f-f)\tau_{h}g+(\tau_{h}g-g)f
\end{equation}
for any functions $f$, $g$, (\ref{class1}), (\ref{F-L16}), (\ref{F-L17}), H\"{o}lder's inequality and the fundamental theorem of calculus that
\begin{eqnarray}\label{F-L19}
&&\|S_{\geq-1}(g^{(2)}_{z\backslash\bar{z}}(w_{lo})v)\|_{L^{\frac{84}{25}}_{x}}\\
\nonumber&\leq&C\sup_{|h|\leq1}\{\|\tau_{h}w_{lo}-w_{lo}\|^{\frac{1}{3}}_{L_{x}^{\frac{112}{37}}}\|\tau_{h}v\|_{L_{x}^{\frac{16}{3}}}+
\|\tau_{h}v-v\|^{\frac{1}{3}}_{L_{x}^{\frac{112}{37}}}\|v\|^{\frac{2}{3}}_{L_{x}^{\frac{16}{3}}}\|w_{lo}\|^{\frac{1}{3}}_{L_{x}^{\frac{16}{3}}}\}\\
\nonumber&\leq&C\{\|\nabla w_{lo}\|^{\frac{1}{3}}_{L_{x}^{\frac{112}{37}}}\|v\|_{L_{x}^{\frac{16}{3}}}+\|\nabla v\|^{\frac{1}{3}}_{L_{x}^{\frac{112}{37}}}\|v\|^{\frac{2}{3}}_{L_{x}^{\frac{16}{3}}}\|w_{lo}\|^{\frac{1}{3}}_{L_{x}^{\frac{16}{3}}}\}\\
\nonumber&\leq&C\|w\|_{\dot{B}^{2}_{r_{1},2}}^{1/3}\|v\|_{\dot{B}^{2}_{r_{1},2}}
\end{eqnarray}
and hence, it follows from H\"{o}lder's inequality on time that
\begin{equation}\label{F-L110}
\|S_{\geq-1}(g^{(2)}_{z\backslash\bar{z}}(w_{lo})v)\|_{L^{8}_{t}L^{\frac{84}{25}}_{x}}\leq C\|w\|_{\dot{W}}^{1/3}\|v\|_{\dot{W}}.
\end{equation}

By combining the estimates (\ref{F-L11})-(\ref{F-L15}) and (\ref{F-L110}), we infer that
\begin{equation}\label{F-L10}
\|\Delta_{0}(g^{(2)}_{z\backslash\bar{z}}(w)v\nabla u)\|_{L^{\frac{16}{3}}_{t}L^{\frac{8}{5}}_{x}}\leq C\|w\|_{\dot{W}}^{1/3}\|v\|_{\dot{W}}\sum_{l\in\mathbb{Z}}\min\{1,2^{\frac{l}{3}}\}\|\Delta_{l}(\nabla u)\|_{L^{16}_{t}L^{8/3}_{x}}.
\end{equation}

Similarly, we can estimate $\|\Delta_{0}(g^{(2)}_{z\backslash\bar{z}}(w)\nabla vu)\|_{L^{\frac{16}{3}}_{t}L^{\frac{8}{5}}_{x}}$. First, by Minkowski's inequality, we have
\begin{equation}\label{F-L21}
\|\Delta_{0}(g^{(2)}_{z\backslash\bar{z}}(w)\nabla vu)\|_{L^{\frac{16}{3}}_{t}L^{\frac{8}{5}}_{x}}\leq \sum_{l\in\mathbb{Z}}\|\Delta_{0}(g^{(2)}_{z\backslash\bar{z}}(w)\nabla v\Delta_{l}u)\|_{L^{\frac{16}{3}}_{t}L^{\frac{8}{5}}_{x}}.
\end{equation}
To bound the right-hand side of (\ref{F-L21}), on one hand, by using (\ref{class1}), (\ref{Strichartz norms}), H\"{o}lder's inequality and the embedding $\dot{B}^{2}_{r_{1},2}(\mathbb{R}^{7})\hookrightarrow\dot{H}^{1,\frac{112}{37}}(\mathbb{R}^{7})\hookrightarrow L_{x}^{\frac{16}{3}}(\mathbb{R}^{7})$, $\dot{H}^{1,\frac{8}{3}}(\mathbb{R}^{7})\hookrightarrow L_{x}^{\frac{56}{13}}(\mathbb{R}^{7})$, we get
\begin{equation}\label{F-L22}
\sum_{l\geq-2}\|\Delta_{0}(g^{(2)}_{z\backslash\bar{z}}(w)\nabla v\Delta_{l}u)\|_{L^{\frac{16}{3}}_{t}L^{\frac{8}{5}}_{x}}\leq C\|w\|_{\dot{W}}^{1/3}\|v\|_{\dot{W}}\sum_{l\geq-2}\|\Delta_{l}(\nabla u)\|_{L^{16}_{t}L^{8/3}_{x}};
\end{equation}
on the other hand, by H\"{o}lder's and Bernstein estimates, we have
\begin{eqnarray}\label{F-L23}
\nonumber &&\sum_{l\leq-3}\|\Delta_{0}(g^{(2)}_{z\backslash\bar{z}}(w)\nabla v\Delta_{l}u)\|_{L^{\frac{16}{3}}_{t}L^{\frac{8}{5}}_{x}}\\
&\leq&C\|S_{\geq-1}(g^{(2)}_{z\backslash\bar{z}}(w)\nabla v)\|_{L^{8}_{t}L^{\frac{84}{37}}_{x}}\sum_{l\leq-3}\|\Delta_{l}u\|_{L^{16}_{t}L^{\frac{168}{31}}_{x}}\\
\nonumber &\leq&C\|S_{\geq-1}(g^{(2)}_{z\backslash\bar{z}}(w)\nabla v)\|_{L^{8}_{t}L^{\frac{84}{37}}_{x}}\sum_{l\leq-3}2^{\frac{l}{3}}
\|\Delta_{l}(\nabla u)\|_{L^{16}_{t}L^{8/3}_{x}}.
\end{eqnarray}
Since by (\ref{pointwise-H2}), we get
\begin{eqnarray}\label{F-L24}
&&\|S_{\geq-1}(g^{(2)}_{z\backslash\bar{z}}(w)\nabla v)\|_{L^{8}_{t}L^{\frac{84}{37}}_{x}}\\
\nonumber&\leq&\|S_{\geq-1}(g^{(2)}_{z\backslash\bar{z}}(w_{lo})\nabla v)\|_{L^{8}_{t}L^{\frac{84}{37}}_{x}}+
\|S_{\geq-1}(O(|w_{hi}|^{\frac{1}{3}})\nabla v)\|_{L^{8}_{t}L^{\frac{84}{37}}_{x}};
\end{eqnarray}
moreover, by H\"{o}lder's inequality, Bernstein estimates and the embedding $\dot{B}^{2}_{r_{1},2}(\mathbb{R}^{7})\hookrightarrow\dot{H}^{1,\frac{112}{37}}(\mathbb{R}^{7})$, we have
\begin{equation}\label{F-L25}
\|S_{\geq-1}(O(|w_{hi}|^{\frac{1}{3}})\nabla v)\|_{L^{8}_{t}L^{\frac{84}{37}}_{x}}\leq C\|\nabla(w_{hi})\|^{1/3}_{L^{\frac{32}{3}}_{t}L^{\frac{112}{37}}_{x}}\|\nabla v\|_{L^{\frac{32}{3}}_{t}L^{\frac{112}{37}}_{x}}
\leq C\|w\|_{\dot{W}}^{1/3}\|v\|_{\dot{W}};
\end{equation}
therefore, we only have the task of estimating $\|S_{\geq-1}(g^{(2)}_{z\backslash\bar{z}}(w_{lo})\nabla v)\|_{L^{8}_{t}L^{\frac{84}{37}}_{x}}$.

We can deduce from (\ref{class1}), (\ref{F-L16}), (\ref{F-L17}), (\ref{F-L18}), H\"{o}lder's inequality and the fundamental theorem of calculus that
\begin{eqnarray}\label{F-L26}
&&\|S_{\geq-1}(g^{(2)}_{z\backslash\bar{z}}(w_{lo})\nabla v)\|_{L^{\frac{84}{37}}_{x}}\\
\nonumber&\leq&C\sup_{|h|\leq1}\{\|\Delta_{h}w_{lo}\|^{\frac{1}{3}}_{L_{x}^{\frac{112}{37}}}\|\nabla v\|_{L_{x}^{\frac{112}{37}}}+\|\Delta_{h}(\nabla v)\|^{\frac{1}{3}}_{L_{x}^{\frac{112}{53}}}\|\nabla v\|^{\frac{2}{3}}_{L_{x}^{\frac{112}{37}}}\|w_{lo}\|^{\frac{1}{3}}_{L_{x}^{\frac{16}{3}}}\}\\
\nonumber&\leq&C\{\|\nabla w_{lo}\|^{\frac{1}{3}}_{L_{x}^{\frac{112}{37}}}\|\nabla v\|_{L_{x}^{\frac{112}{37}}}+\|v\|^{\frac{1}{3}}_{\dot{H}^{2,r_{1}}}\|\nabla v\|^{\frac{2}{3}}_{L_{x}^{\frac{112}{37}}}\|w_{lo}\|^{\frac{1}{3}}_{L_{x}^{\frac{16}{3}}}\}\\
\nonumber&\leq&C\|w\|_{\dot{B}^{2}_{r_{1},2}}^{1/3}\|v\|_{\dot{B}^{2}_{r_{1},2}},
\end{eqnarray}
where $\Delta_{h}F:=\tau_{h}F-F$ is the first order difference operator. Hence, it follows from H\"{o}lder's inequality on time that
\begin{equation}\label{F-L27}
\|S_{\geq-1}(g^{(2)}_{z\backslash\bar{z}}(w_{lo})\nabla v)\|_{L^{8}_{t}L^{\frac{84}{37}}_{x}}\leq C\|w\|_{\dot{W}}^{1/3}\|v\|_{\dot{W}}.
\end{equation}

By combining the estimates (\ref{F-L21})-(\ref{F-L25}) and (\ref{F-L27}), we infer that
\begin{equation}\label{F-L20}
\|\Delta_{0}(g^{(2)}_{z\backslash\bar{z}}(w)\nabla vu)\|_{L^{\frac{16}{3}}_{t}L^{\frac{8}{5}}_{x}}\leq C\|w\|_{\dot{W}}^{1/3}\|v\|_{\dot{W}}\sum_{l\in\mathbb{Z}}\min\{1,2^{\frac{l}{3}}\}\|\Delta_{l}(\nabla u)\|_{L^{16}_{t}L^{8/3}_{x}}.
\end{equation}

It follows from \eqref{F-L10} and \eqref{F-L20} that \eqref{f-l nonlinear} holds for $k=0$, then by scaling, we know it holds for any $k\in\mathbb{Z}$, this completes the proof of \eqref{f-l nonlinear}.

We can rewrite \eqref{f-l nonlinear} as
\begin{eqnarray}\label{re f-l nonlinear}
&&2^{\frac{k}{4}}\|\Delta_{k}(\nabla(g_{z\backslash\bar{z}}(v)u))\|_{L^{\frac{16}{3}}(I,L^{8/5}(\mathbb{R}^{7}))} \\
\nonumber&\leq&C\|v\|_{\dot{W}}^{4/3}\sum_{l\in\mathbb{Z}}\min\{2^{\frac{k-l}{4}},2^{\frac{l-k}{12}}\}2^{\frac{l}{4}}\|\Delta_{l}(\nabla u)\|_{L^{16}(I,L^{8/3}(\mathbb{R}^{7}))}
\end{eqnarray}
for any $k\in\mathbb{Z}$. Therefore, we deduce from \eqref{exotic norms 1}, \eqref{exotic norms 2}, \eqref{re f-l nonlinear} and Schur's test the following nonlinear estimate:
\begin{equation}\label{nonlinear estimate}
    \|g_{z\backslash\bar{z}}(v)u\|_{V}\leq C\|v\|_{\dot{W}}^{4/3}\|u\|_{U},
\end{equation}
whenever the right-hand side makes sense.

Note that $(16,\frac{56}{27})$ is an admissible pair and $\dot{H}^{\frac{3}{4},\frac{56}{27}}(\mathbb{R}^{7})\hookrightarrow L^{\frac{8}{3}}(\mathbb{R}^{7})$, we deduce from exotic Strichartz estimate (\ref{exotic Strichartz estimate}) combined with nonlinear estimate (\ref{nonlinear estimate}) that
\begin{eqnarray*}
% \nonumber to remove numbering (before each equation)
  \|u_{n}-u\|_{U}&\leq& C(\sum_{k\in\mathbb{Z}}\|e^{i\cdot\Delta}(\Delta_{k}(\phi_{n}-\phi))\|_{L^{16}(I,\dot{H}^{2,\frac{56}{27}})}^{2})^{1/2}
  +C\|(g(u_{n})-g(u))\|_{V}\\
&\leq&C\|\phi_{n}-\phi\|_{\dot{H}^{2}}+C(\|u_{n}\|^{4/3}_{\dot{W}}+\|u\|^{4/3}_{\dot{W}}) \|u_{n}-u\|_{U}.
\end{eqnarray*}
Thus by (\ref{continuity2}), we have for all sufficiently large $n$,
\[\|u_{n}-u\|_{U}\leq C\|\phi_{n}-\phi\|_{\dot{H}^{2}}+C\delta^{4/3}\|u_{n}-u\|_{U}.\]
Therefore, by possibly choosing $\delta$ smaller, and choosing $T$ smaller if necessary, such that
\begin{equation}\label{delta2}
    C\delta^{4/3}<1/2,
\end{equation}
we can deduce that as $n\rightarrow \infty$, $\|u_{n}-u\|_{U}\rightarrow 0$.

Since $\dot{H}^{\frac{5}{4},\frac{8}{3}}(\mathbb{R}^{7})\hookrightarrow \dot{H}^{1,\frac{56}{19}}(\mathbb{R}^{7})$, we deduce from (\ref{exotic norms 1}) and Minkowski's inequality that
\begin{equation}\label{convergence}
\|u_{n}-u\|_{L^{16}(I,\dot{H}^{1,\frac{56}{19}})}\leq C\|u_{n}-u\|_{L^{16}(I,\dot{H}^{\frac{5}{4},\frac{8}{3}})}\leq C\|u_{n}-u\|_{U}\rightarrow 0,
\end{equation}
as $n\rightarrow \infty$. Note that $(\gamma,\rho)=(\frac{10}{3},\frac{70}{29})$ and $\frac{2}{\gamma}=7(\frac{1}{2}-\frac{1}{\rho^{\ast}_{1}})-1$, it follows from (\ref{continuity2}), (\ref{convergence}) and interpolations that
\[\|\nabla(u_{n}-u)\|_{L^{\gamma}((-T,T),L^{\rho^{\ast}_{1}})}\rightarrow 0, \quad as \quad n\rightarrow \infty.\]
Hence, from $\dot{H}^{1,\rho^{\ast}_{1}}(\mathbb{R}^{7})\hookrightarrow L^{\rho^{\ast}_{0}}(\mathbb{R}^{7})$ and the convergence of Lebesgue norms, we deduce that as $n\rightarrow\infty$, $u_{n}\rightarrow u$ in $L^{\gamma}((-T,T),H^{j,\rho^{\ast}_{j}}(\mathbb{R}^{7}))$ for $j=0,1$, this closes our proof.
\end{proof}

Now let us fix $\delta>0$ small enough so that (\ref{delta0}), (\ref{delta1}) and (\ref{delta2}) hold. For sufficiently small $T>0$ satisfying $\|e^{i\cdot\Delta}\phi\|_{L^{\gamma}((-T,T),B^{s}_{\rho,2})}<\delta$, we are to show that $\|u_{n}-u\|_{L^{q}(I,B^{s}_{r,2})}\rightarrow 0$ as $n\rightarrow\infty$, where $I=(-T,T)$ and $(q,r)$ is an arbitrary admissible pair.

Let $Y_{n}(I):=\|u_{n}-u\|_{L^{\gamma}(I,B^{s}_{\rho,2})}$, arguing the same as in the subcritical case and noting that $T^{\sigma}=T^{\frac{4-\alpha(N-2s)}{4}}=1$, by (\ref{continuity1}), we have the following two cases(see the proof of Case I, more precisely, see (\ref{X1}), (\ref{X2}), (\ref{X3}) and (\ref{X4})): \\
Case (i) \, if $g_{1}(u)$ is a polynomial in $u$ and $\bar{u}$, or if not, we assume further that $\alpha\geq [s]+1$, then
\begin{equation}\label{Yn0}
Y_n\leq C\|\phi_{n}-\phi\|_{H^{s}}+C(4\delta)^{\alpha}Y_n;
\end{equation}
Case (ii) \, if $g_{1}$ is not a polynomial and $[s]<\alpha<[s]+1$, then
\begin{equation}\label{Yn1}
 Y_n\leq C\|\phi_{n}-\phi\|_{H^{s}}+C\{(4\delta)^{\alpha}Y_n+\|K(u_{n},u)\|_{L^{\gamma'}(-T,T)}\};
\end{equation}
where $n$ is large enough and $K(u_{n},u)$ is defined by \eqref{Kato} and \eqref{Kato2}.\\

For case (i), it follows from \eqref{delta0} that $C(4\delta)^{\alpha}\leq 1/2$, so we can deduce from \eqref{Yn0} that $Y_{n}\leq C\|\phi_{n}-\phi\|_{H^{s}}\rightarrow 0$, as $n\rightarrow\infty$, which implies that
\[\|g(u_{n})-g(u)\|_{L^{\gamma'}(I,B^{s}_{\rho',2})}\leq C(4\delta)^{\alpha}Y_n\leq C\|\phi_{n}-\phi\|_{H^{s}}\rightarrow 0,\]
as $n\rightarrow \infty$. Applying Strichartz's estimates, we conclude that for all admissible pair $(q,r)$, $\|u_{n}-u\|_{L^{q}(I,B^{s}_{r,2})}\rightarrow 0$, as $n\rightarrow \infty$, and the continuous dependence is locally Lipschitz.\\

As to case (ii), by making use of the convergence stated in Claim \ref{claim1} and Claim \ref{claim2}, we certainly can apply the contradiction argument as in the subcritical case(see Case I), which won't rely on the assumption (\ref{class2}) if $s$ is not an integer. It is completely similar to the subcritical case, so we omit the details. Here, for the sake of simplicity, we prefer to deal with the problem in a simple and direct way, which will rely on (\ref{class2}).

Indeed, it follows from \eqref{delta0} that $C(4\delta)^{\alpha}\leq 1/2$, so we can deduce from (\ref{Kato1}), (\ref{KATO}), (\ref{continuity1}) and (\ref{Yn1}) that
\begin{equation}\label{Yn2}
Y_n\leq C\{\|\phi_{n}-\phi\|_{H^{s}}+\|u_{n}-u\|^{\alpha-[s]}_{L^{\gamma}((-T,T),L^{\rho^{\ast}_{0}})}\}
\end{equation}
for $n$ large enough. As a consequence of Claim \ref{claim1} and Claim \ref{claim2}, we know
\[\|u_{n}-u\|_{L^{\gamma}((-T,T),L^{\rho^{\ast}_{0}})}\rightarrow 0, \quad as \quad n\rightarrow \infty,\]
which yields the desired convergence, by \eqref{Yn2}. The convergence for arbitrary admissible pair $(q,r)$ follows from Strichartz's estimates.

Thus we have proved that in both case (i) and (ii), there exists $0<T<T_{max}(\phi),T_{min}(\phi)$ sufficiently small such that if $\phi_{n}\rightarrow \phi$ in $H^{s}(\mathbb{R}^{N})$, then $T_{max}(\phi_{n}),T_{min}(\phi_{n})>T$ for all sufficiently large $n$ and $u_{n}\rightarrow u$ in $L^{q}((-T,T),B^{s}_{r,2}(\mathbb{R}^{N}))$ as $n\rightarrow \infty$ for every admissible pair $(q,r)$.

For simplicity, we only argue forwards in time, as arguing backwards in time can be handled similarly. Let $\tilde{T}_{+}$(resp., $\tilde{T}_{-}$) be the supremum of all $0<T<T_{max}(\phi)$(resp., $0<T<T_{min}(\phi)$) such that if $\phi_{n}\rightarrow \phi$ in $H^{s}(\mathbb{R}^{N})$, then $T_{max}(\phi_{n})>T$(resp., $T_{min}(\phi_{n})>T$) for all sufficiently large $n$ and $u_{n}\rightarrow u$ in $L^{q}((0,T),B^{s}_{r,2}(\mathbb{R}^{N}))$(resp., $L^{q}((-T,0),B^{s}_{r,2}(\mathbb{R}^{N}))$) as $n\rightarrow \infty$ for every admissible pair $(q,r)$. We have already proved that $\tilde{T}_{+}>0$, and by definition $\tilde{T}_{+}\leq T_{max}(\phi)$. We claim that $\tilde{T}_{+}=T_{max}(\phi)$. If not, then we have $\tilde{T}_{+}< T_{max}(\phi)$. Since $u\in C([0,\tilde{T}_{+}],H^{s}(\mathbb{R}^{N}))$, $\bigcup_{0\leq t\leq \tilde{T}_{+}}\{u(t)\}$ is a compact subset of $H^{s}(\mathbb{R}^{N})$ and hence $\bigcup_{0\leq t\leq \tilde{T}_{+}}\{u(t)\}$ has a finite ``$\varepsilon$-net" in $H^{s}(\mathbb{R}^{N})$ for any $\varepsilon>0$. Therefore, it follows from Strichartz's estimates that there exists $0<\tau<T_{max}(\phi)-\tilde{T}_{+}$ such that
\[\sup_{0\leq t\leq \tilde{T}_{+}}\|e^{i\cdot \Delta}u(t)\|_{L^{\gamma}((0,\tau),B^{s}_{\rho,2})}<\delta,\]
where $\delta>0$ is fixed to be small enough so that (\ref{delta0}), (\ref{delta1}) and (\ref{delta2}) hold. As $0<\tilde{T}_{+}< T_{max}(\phi)$, we can fix $0<T'<\tilde{T}_{+}$ such that $\tilde{T}_{+}<T'+\tau<T_{max}(\phi)$, then $T_{max}(\phi_{n})>T'$ for all sufficiently large $n$ and $u_{n}(T')\rightarrow u(T')$ in $H^{s}(\mathbb{R}^{N})$ as $n\rightarrow \infty$. Thus we can deduce from
\[\|e^{i\cdot \Delta}u(T')\|_{L^{\gamma}((0,\tau),B^{s}_{\rho,2})}<\delta\]
and the preceding arguments(apply forwards in time) that $T_{max}(u_{n}(T'))>\tau$ for all $n$ large enough and that $u_{n}(T'+\cdot)\rightarrow u(T'+\cdot)$ in $L^{q}((0,\tau),B^{s}_{r,2}(\mathbb{R}^{N}))$ for all admissible pair $(q,r)$. So we have $T_{max}(\phi_{n})>T'+\tau$ for all $n$ large enough and $u_{n}\rightarrow u$ in
$L^{q}((0,T'+\tau),B^{s}_{r,2}(\mathbb{R}^{N}))$ as $n\rightarrow \infty$ for all admissible pair $(q,r)$. Then by the definition of $\tilde{T}_{+}$, we obtain $T'+\tau\leq \tilde{T}_{+}$, which is a contradiction. Thus
$\tilde{T}_{+}=T_{max}(\phi)$. By a similar argument, we can show that $\tilde{T}_{-}=T_{min}(\phi)$.

Thus we see that for any $0<T<T_{max}(\phi)$ and $0<S<T_{min}(\phi)$, if $\phi_{n}\rightarrow \phi$ in $H^{s}(\mathbb{R}^{N})$, then $T_{max}(\phi_{n})>T$ and $T_{min}(\phi_{n})>S$ for all sufficiently large $n$ and $u_{n}\rightarrow u$ in $L^{q}((-S,T),B^{s}_{r,2}(\mathbb{R}^{N}))$ as $n\rightarrow \infty$ for all admissible pair $(q,r)$, which proves the conclusions $(i)$ and $(ii)$ of the Theorem \ref{th1}. Moreover, if $g_{1}(u)$ is a polynomial in $u$ and $\bar{u}$, or if not, we assume further that $[s]+1\leq\alpha=\frac{4}{N-2s}$, then the continuous dependence is locally Lipschitz. \\
\\
Case III. The borderline case $s=N/2$, $\alpha<\infty$.\\

When $s=N/2$, the embedding $H^{\frac{N}{2}}(\mathbb{R}^{N})\hookrightarrow L^{p}(\mathbb{R}^{N})$ for all $2\leq p<\infty$ makes if possible to obtain local existence for (sufficiently regular) nonlinearities with arbitrary polynomial growth. In particular, there is local existence in the model case $g(u)=\lambda|u|^{\alpha}u$ for any $\alpha>0$ such that $\alpha>[s]$(see \cite{b7}). Using Trudinger's inequality, one can also consider nonlinearities of exponential growth(see \cite{b17}).

The case $s=N/2$ may be regarded as the ``borderline" in two aspects. First, different from the cases $s<N/2$, no power behavior of interaction amounts to the critical nonlinearity at the level of $H^{\frac{N}{2}}$. Second, since the embedding $H^{\frac{N}{2}}\hookrightarrow L^{\infty}$ falls down, pointwise control of solutions falls beyond the scope of the $H^{\frac{N}{2}}$-theory, therefore, without any specific control on the growth of interaction, any argument similar to that of the $H^{s}$-theory with $s>N/2$ breaks down, even for local theory.

We will carry out our discussion by considering the following two cases respectively: N is even and N is odd.

Since $g_{1}$ is of class $\mathcal{C}(\alpha,s)$ with $s=N/2$ and $0<\alpha<\infty$, we can fix $\delta>0$ sufficiently small so that
\[\left\{
  \begin{array}{ll}
   0<\delta<\min\{\chi(\alpha-[s]),\frac{8}{N(\alpha+2)-4},\frac{2}{N-1}\}, & if \,\,\, N\geq3,\\
   \,&\\
   0<\delta<\min\{\alpha,\frac{4}{\alpha},1\}, & if \,\,\, N=2,\\
   \,&\\
   0<\delta<\min\{\alpha,\frac{8}{(\alpha-2)^{+}},1\}, & if \,\,\, N=1,
  \end{array}
\right.\]
where $\chi(x):=x$ if $x>0$ and $\chi(x):=1$ if $x\leq0$. Let $r=2+\delta$, $(q,r)$ be the corresponding admissible pair and let $p=(1+\frac{2}{\delta})(\alpha-[s])$ provided $\alpha>[s]$, then we have $r<N/[s]$ for $N\geq3$, $q>\alpha+2$, $r<p<\infty$ and
\begin{equation}\label{eq3.0}
\frac{1}{r'}=\frac{\alpha-[s]}{p}+\frac{1}{r}.
\end{equation}
Let us define
\[\tilde{X}_{n}(I):=\|u_{n}(t)-u(t)\|_{L^{\infty}(I,H^{s})}+\|u_{n}(t)-u(t)\|_{L^{q}(I,B^{s}_{r,2})},\]
where $I=(-T,T)$. Since $\|u_{n}\|_{L^{q}(I,B^{s}_{r,2})}$ is bounded, we can define
\begin{equation}\label{eq3.2}
M:=\|u\|_{L^{q}(I,B^{s}_{r,2})}+\sup_{n\geq 1}\|u_{n}\|_{L^{q}(I,B^{s}_{r,2})}< \infty
\end{equation}
such that $\|u\|_{L^{q}(I,B^{s}_{r,2})}\leq M$ and $\|u_{n}\|_{L^{q}(I,B^{s}_{r,2})}\leq M$ for any $n\geq1$.

First, we will consider the cases that the spatial dimension $N$ is even. Note that $[s]=\frac{N}{2}-1$ in such cases. By Strichartz's estimate, we have
\begin{equation}\label{eq3.1}
\tilde{X}_{n}\leq C\|\phi_{n}-\phi\|_{H^{s}}+C\{T\tilde{X}_{n}+\|g_{1}(u_{n})-g_{1}(u)\|_{L^{q'}(I,H^{s,r'})}\}.
\end{equation}
To get our conclusion, we are to prove that $\tilde{X}_{n}\rightarrow0$ as $n\rightarrow\infty$, thus we need to estimate $\|g_{1}(u_{n})-g_{1}(u)\|_{L^{q'}(I,H^{s,r'})}$.

For a start, we assume that $g_{1}$ is not a polynomial, then by Definition \ref{class}, we have $[s]<\alpha<\infty$, thus the exponent $p$ is well defined. Note that by \eqref{class1}, \eqref{eq3.0} and H\"{o}lder's inequality, we have
\[\|g_{1}(u_{n})-g_{1}(u)\|_{L^{r'}}\leq C(\|u_{n}\|_{L^{p}}+\|u\|_{L^{p}})^{\alpha+1-\frac{N}{2}}(\|u_{n}\|_{L^{\infty}}+\|u\|_{L^{\infty}})^{\frac{N}{2}-1}\|u_{n}-u\|_{L^{r}},\]
thus we deduce from H\"{o}lder's estimate on time, \eqref{eq3.2} and the embedding $B^{s}_{r,2}\hookrightarrow L^{p}$, $B^{s}_{r,2}\hookrightarrow L^{\infty}$ that
\begin{equation}\label{eq3.3}
\|g_{1}(u_{n})-g_{1}(u)\|_{L^{q'}(I,L^{r'})}\leq CT^{1-\frac{\alpha+2}{q}}M^{\alpha}\|u_{n}-u\|_{L^{q}(I,L^{r})}.
\end{equation}

For the part of derivative estimate, by $L^{p}$ boundedness of Riesz transforms, we have the formula
\begin{eqnarray}\label{eq3.4}
&&\||\nabla|^{\frac{N}{2}}(g_{1}(u_{n})-g_{1}(u))\|_{L^{r'}}\\
\nonumber&\leq&\sum_{k=1}^{N/2}\sum_{|\beta_1|+\cdots+|\beta_k|=\frac{N}{2},|\beta_j|\geq1}\{\|(g_{1,z\backslash\bar{z}}^{(k)}(u_n)-g_{1,z\backslash\bar{z}}^{(k)}(u))
\prod_{j=1}^k\partial^{\beta_j}(u\backslash\bar{u})\|_{L^{r'}}\\
\nonumber&&+\sum_{(w_{1},\cdots,w_{k})}\|g_{1,z\backslash\bar{z}}^{(k)}(u_{n})\prod_{j=1}^k\partial^{\beta_j}w_j\|_{L^{r'}}\}=:S_{1}+S_{2},
\end{eqnarray}
where $k\in\{1,\cdots,N/2\}$ and the $\beta_{j}$'s are multi-indices such that $\frac{N}{2}=|\beta_{1}|+\cdots+|\beta_{k}|$ and $|\beta_{j}|\geq 1$ for $j=1,\cdots,k$, all the $w_j$'s are equal to $u_{n},\bar{u}_{n}$ or $u,\bar{u}$, except one which is equal to $u_{n}-u$ or its conjugate. For any $k\in\{1,\cdots,N/2\}$, we define indices $r_{j}$($1\leq j\leq k$) by $r_{j}=\frac{Nr}{2|\beta_{j}|}$, then we have
\begin{equation}\label{eq3.5}
\frac{1}{r'}=\frac{\alpha+1-\frac{N}{2}}{p}+\sum_{j=1}^{k}\frac{1}{r_{j}}.
\end{equation}

Since $g_{1}$ is of class $\mathcal{C}(\alpha,s)$, for the estimate of $S_{1}$, we have two cases:\\
(i)\, if we assume further that $\alpha\geq [s]+1=\frac{N}{2}$, then we deduce from \eqref{class1}, \eqref{class2}, \eqref{eq3.5}, H\"{o}lder's estimate and the embedding $B^{s}_{r,2}\hookrightarrow L^{p}$, $B^{s}_{r,2}\hookrightarrow L^{\infty}$, $B^{s}_{r,2}\hookrightarrow H^{|\beta_{j}|,r_{j}}$ that for any $k\in\{1,\cdots,N/2\}$,
\begin{eqnarray*}
&&\|(g_{1,z\backslash\bar{z}}^{(k)}(u_n)-g_{1,z\backslash\bar{z}}^{(k)}(u))\prod_{j=1}^k\partial^{\beta_j}(u\backslash\bar{u})\|_{L^{r'}}\\
&\leq&C(\|u_{n}\|_{L^{p}}+\|u\|_{L^{p}})^{\alpha-\frac{N}{2}}(\|u_{n}\|_{L^{\infty}}+\|u\|_{L^{\infty}})^{\frac{N}{2}-k}\|u_{n}-u\|_{L^{p}}
\prod_{j=1}^k\|\partial^{\beta_j}u\|_{L^{r_{j}}}\\
&\leq&C(\|u_{n}\|^{\alpha}_{B^{s}_{r,2}}+\|u\|^{\alpha}_{B^{s}_{r,2}})\|u_{n}-u\|_{B^{s}_{r,2}},
\end{eqnarray*}
thus
\begin{equation}\label{eq3.6}
S_{1}\leq C(\|u_{n}\|^{\alpha}_{B^{s}_{r,2}}+\|u\|^{\alpha}_{B^{s}_{r,2}})\|u_{n}-u\|_{B^{s}_{r,2}};
\end{equation}
(ii) if $\frac{N}{2}-1=[s]<\alpha<[s]+1=N/2$, by (\ref{class1}), we argue as above and get that, for any $k\in\{1,\cdots,\frac{N}{2}-1\}$,
\[\|(g_{1,z\backslash\bar{z}}^{(k)}(u_n)-g_{1,z\backslash\bar{z}}^{(k)}(u))\prod_{j=1}^k\partial^{\beta_j}(u\backslash\bar{u})\|_{L^{r'}}
\leq C(\|u_{n}\|^{\alpha}_{B^{s}_{r,2}}+\|u\|^{\alpha}_{B^{s}_{r,2}})\|u_{n}-u\|_{B^{s}_{r,2}},\]
except when $k=N/2$, we deduce from \eqref{eq3.5} and H\"{o}lder's inequality that
\begin{eqnarray*}
&&\|(g_{1,z\backslash\bar{z}}^{(\frac{N}{2})}(u_n)-g_{1,z\backslash\bar{z}}^{(\frac{N}{2})}(u))\prod_{j=1}^{N/2}\partial^{\beta_j}(u\backslash\bar{u})\|_{L^{r'}}\\
&\leq&C\|D^{\frac{N}{2}}(g_{1}(u_{n})-g_{1}(u))\|_{L^{\frac{p}{\alpha-[s]}}}\prod_{j=1}^{N/2}\|\partial^{\beta_j}u\|_{L^{r_{j}}}\\
&\leq&C\|u\|_{B^{s}_{r,2}}^{[s]+1}\|D^{[s]+1}(g_{1}(u_{n})-g_{1}(u))\|_{L^{\frac{p}{\alpha-[s]}}},
\end{eqnarray*}
thus
\begin{equation}\label{eq3.7}
\begin{array}{ll}
S_1\leq C\{(\|u_{n}\|^{\alpha}_{B^{s}_{r,2}}+\|u\|^{\alpha}_{B^{s}_{r,2}})\|u_{n}-u\|_{B^{s}_{r,2}}
&\\
\qquad\qquad+\|D^{[s]+1}(g_{1}(u_{n})-g_{1}(u))\|_{L^{\frac{p}{\alpha-[s]}}}\|u\|^{[s]+1}_{B^{s}_{r,2}}\}.
\end{array}
\end{equation}

As to the estimate of $S_{2}$, note that all the $w_j$'s are equal to $u_{n},\bar{u}_{n}$ or $u,\bar{u}$, except one which is equal to $u_{n}-u$ or its conjugate, we deduce from \eqref{class1}, \eqref{eq3.5}, H\"{o}lder's inequality and Sobolev embedding that
\begin{eqnarray*}
S_{2}&\leq&C\sum_{k=1}^{N/2}\sum_{\sum_{j=1}^{k}|\beta_j|=\frac{N}{2},|\beta_j|\geq1}\sum_{(w_{1},\cdots,w_{k})}
\|u_{n}\|_{L^{p}}^{\alpha+1-\frac{N}{2}}\|u_{n}\|_{L^{\infty}}^{\frac{N}{2}-k}\prod_{j=1}^{k}\|\partial^{\beta_j}w_j\|_{L^{r_{j}}}\\
&\leq&C(\|u_{n}\|^{\alpha}_{B^{s}_{r,2}}+\|u\|^{\alpha}_{B^{s}_{r,2}})\|u_{n}-u\|_{B^{s}_{r,2}}.
\end{eqnarray*}

By inserting the estimates for $S_1$ and $S_{2}$ into \eqref{eq3.4} and applying H\"{o}lder's estimates on time, we have the following two cases for the estimate of $\|g_{1}(u_{n})-g_{1}(u)\|_{L^{q'}(I,\dot{H}^{s,r'})}$:\\
(i)\, if $g_{1}$ is not a polynomial and $\alpha\geq[s]+1=\frac{N}{2}$, then
\begin{eqnarray}\label{eq3.8}
&&\|g_{1}(u_{n})-g_{1}(u)\|_{L^{q'}(I,\dot{H}^{s,r'})}\\
\nonumber&\leq&CT^{1-\frac{\alpha+2}{q}}(\|u_{n}\|^{\alpha}_{L^{q}(I,B^{s}_{r,2})}+\|u\|^{\alpha}_{L^{q}(I,B^{s}_{r,2})})
\|u_{n}-u\|_{L^{q}(I,B^{s}_{r,2})};
\end{eqnarray}
(ii) if $g_{1}$ is not a polynomial and $\frac{N}{2}-1=[s]<\alpha<[s]+1=N/2$, then
\begin{equation}\label{eq3.9}
\begin{array}{ll}
\|g_{1}(u_{n})-g_{1}(u)\|_{L^{q'}(I,\dot{H}^{s,r'})}\leq C\{T^{1-\frac{\alpha+2}{q}}(\|u_{n}\|^{\alpha}_{L^{q}(I,B^{s}_{r,2})}&\\
\qquad\qquad+\|u\|^{\alpha}_{L^{q}(I,B^{s}_{r,2})})\|u_{n}-u\|_{L^{q}(I,B^{s}_{r,2})}+\|K(u_{n},u)\|_{L^{q'}(I)}\},&
\end{array}
\end{equation}
where the Kato's remainder term $K(u_{n},u)$ can be defined by
\begin{equation}\label{eq3.10}
K(u_{n},u)=\|D^{[s]+1}g_{1}(u_{n})-D^{[s]+1}g_{1}(u)\|_{L^{\frac{p}{\alpha-[s]}}}\|u\|^{[s]+1}_{B^{s}_{r,2}}.
\end{equation}

As to the cases that $g_{1}(u)$ is a polynomial in $u$ and $\bar{u}$, note that $\alpha\geq1$ in such cases, we define $\tilde{p}=1+\frac{2}{\delta}$, then we have $r<\tilde{p}<\infty$, $B^{s}_{r,2}\hookrightarrow L^{\tilde{p}}$ and
\[\frac{1}{r'}=\frac{1}{\tilde{p}}+\frac{1}{r}.\]
Then similar to the cases that $g_{1}$ is not a polynomial as above, by H\"{o}lder's estimate, it follows from a similar but much simpler argument that
\begin{eqnarray}\label{eq3.100}
&&\|g_{1}(u_{n})-g_{1}(u)\|_{L^{q'}(I,{H}^{s,r'})}\\
\nonumber&\leq&CT^{1-\frac{\alpha+2}{q}}(\|u_{n}\|^{\alpha}_{L^{q}(I,B^{s}_{r,2})}+\|u\|^{\alpha}_{L^{q}(I,B^{s}_{r,2})})
\|u_{n}-u\|_{L^{q}(I,B^{s}_{r,2})}.
\end{eqnarray}

Therefore, by inserting \eqref{eq3.3}, \eqref{eq3.8}, \eqref{eq3.9} and \eqref{eq3.100} into the original Strichartz's estimate \eqref{eq3.1} and using \eqref{eq3.2}, we can obtain estimate for $\tilde{X}_n$ by considering the following two cases respectively:\\
Case (i) \,\, if $g_{1}(u)$ is a polynomial in $u$ and $\bar{u}$, or if not, we assume further that $\alpha\geq [s]+1$, then
\begin{equation}\label{eq3.11}
\tilde{X}_n\leq C\|\phi_{n}-\phi\|_{H^{s}}+CT\tilde{X}_n+CT^{\sigma}M^{\alpha}\tilde{X}_n;
\end{equation}
Case (ii) \, if $g_{1}$ is not a polynomial and $[s]<\alpha<[s]+1$, then
\begin{equation}\label{eq3.12}
\tilde{X}_n\leq C\|\phi_{n}-\phi\|_{H^{s}}+C\{T\tilde{X}_n+T^{\sigma}M^{\alpha}\tilde{X}_n+\|K(u_{n},u)\|_{L^{q'}(-T,T)}\};
\end{equation}
where $\sigma=1-\frac{\alpha+2}{q}>0$.

For case $(i)$, we can choose $T$ sufficiently small so that $C(T+T^{\sigma}M^{\alpha})\leq\frac{1}{2}$, and deduce from \eqref{eq3.11} that as $n\rightarrow \infty$,
\[\tilde{X}_{n}=\|u_{n}(t)-u(t)\|_{L^{\infty}(I,H^{s})}+\|u_{n}(t)-u(t)\|_{L^{q}(I,B^{s}_{r,2})}\leq C\|\phi_{n}-\phi\|_{H^{s}}\rightarrow 0,\]
and hence the solution flow is locally Lipschitz.

As to case $(ii)$, we can see that the key point is to show that the Kato's remainder term $\|K(u_{n},u)\|_{L^{q'}(-T,T)}\rightarrow0$ as $n\rightarrow\infty$, which will yield the desired convergence of $\tilde{X}_{n}$ for $T$ small enough immediately. To this end, there are two different method for us to choose.

First, entirely similar to the subcritical case for $s<N/2$(see Case I), we can prove the convergence of $K(u_{n},u)$ by a standard contradiction argument based on dominated convergence theorem, which will only rely on (\ref{class1}). Since we have known that $u_{n}\rightarrow u$ in $L^{q}(I,B^{s-\varepsilon}_{r,2})$ for arbitrary $\varepsilon>0$, we deduce easily that $u_{n}\rightarrow u$ in $L^{q}(I,L^{p})$. Note that by (\ref{class1}) and \eqref{eq3.10}, we have
\[K(u_{n},u)\leq C\|u\|^{[s]+1}_{B^{s}_{r,2}}(\|u_{n}\|^{\alpha-[s]}_{L^{p}}+\|u\|^{\alpha-[s]}_{L^{p}}),\]
thus by $u_{n}\rightarrow u$ in $L^{q}(I,L^{p})$ and $q>\alpha+2$, we can construct a control function $\omega\in
L^{\frac{(\alpha-[s])q}{q-[s]-2}}(I,L^{p}(\mathbb{R}^{N}))$ such that $|u_{n}|\leq \omega$ a.e. on $(-T,T)\times\mathbb{R}^{N}$ and
\[K(u_{n},u)^{q'}\leq C(\|\omega\|^{\frac{(\alpha-[s])q}{q-[s]-2}}_{L^{p}}+\|u\|^{\frac{(\alpha-[s])q}{q-[s]-2}}_{L^{p}}
+\|u\|^{q}_{B^{s}_{r,2}})\in L^{1}((-T,T)).\]
Then we can apply the dominated convergence theorem to obtain a contradiction, we omit the rest of details here(refer to Case I).

Second, we can make use of assumption (\ref{class2}) to obtain a H\"{o}lder type estimate of $K(u_{n},u)$:
\[\|K(u_{n},u)\|_{L^{q'}(I)}\leq CT^{1-\frac{\alpha+2}{q}}\|u\|^{[s]+1}_{L^{q}(I,B^{s}_{r,2})}\|u_{n}-u\|^{\alpha-[s]}_{L^{q}(I,L^{p})}\]
and hence, it follows from $u_{n}\rightarrow u$ in $L^{q}(I,L^{p})$ that
\[\|K(u_{n},u)\|_{L^{q'}(-T,T)}\rightarrow0, \,\,\,\, as \,\,\,\, n\rightarrow\infty.\]

Therefore, for the cases that $N$ is even, we have proved $\tilde{X}_{n}\rightarrow 0$ as $n\rightarrow \infty$ if $T$ is sufficiently small in both case (i) and (ii).

Now we turn to the cases that the spatial dimension $N$ is odd. Note that $[s]=\frac{N-1}{2}$ in such cases. By Strichartz's estimate, we have
\begin{equation}\label{eq3.13}
\tilde{X}_{n}\leq C\|\phi_{n}-\phi\|_{H^{s}}+C\{T\tilde{X}_{n}+\|g_{1}(u_{n})-g_{1}(u)\|_{L^{q'}(I,B^{s}_{r',2})}\}.
\end{equation}
To get our conclusion, we are to prove that $\tilde{X}_{n}\rightarrow0$ as $n\rightarrow\infty$, thus we need to estimate $\|g_{1}(u_{n})-g_{1}(u)\|_{L^{q'}(I,B^{s}_{r',2})}$. By Lemma \ref{alternative}, we have
\begin{equation}\label{eq3.14}
\|f\|_{B^{s}_{r',2}}\sim\|f\|_{L^{r'}}+\sum_{j=1}^{N}(\int_{0}^{\infty}(t^{-\frac{1}{2}}\sup_{|y|\leq t}\|\Delta_{y}\partial_{x_{j}}^{[s]}f\|_{L^{r'}})^{2}\frac{dt}{t})^{1/2},
\end{equation}
where $\Delta_{y}$ denotes the first order difference operator. Therefore, we can see that the key ingredient is how to estimate $\|\Delta_{y}\partial_{x_{j}}^{[s]}(g_{1}(u_{n})-g_{1}(u))\|_{L^{r'}}$ for $j=1,\cdots,N$.

For a start, we assume that $g_{1}$ is not a polynomial, then by Definition \ref{class}, we have $[s]<\alpha<\infty$, thus the exponent $p$ is well defined.

If $N=1$, $s=\frac{1}{2}$, by \eqref{class1}, \eqref{eq3.0}, the fundamental theorem of calculus and H\"{o}lder's estimate, we have
\begin{eqnarray}\label{eq3.15}
&&\|\Delta_{y}(g_{1}(u_{n})-g_{1}(u))\|_{L^{r'}}\\
\nonumber&\leq&C\{\|\Delta_{y}(u_{n}-u)\int_{0}^{1}Dg_{1}([\tau_{y}u_{n},u_{n}]_{\theta})d\theta\|_{L^{r'}}\\
\nonumber&&+\|\Delta_{y}u\int_{0}^{1}[Dg_{1}([\tau_{y}u_{n},u_{n}]_{\theta})-Dg_{1}([\tau_{y}u,u]_{\theta})]d\theta\|_{L^{r'}}\}\\
\nonumber&\leq&C\{\|\Delta_{y}(u_{n}-u)\|_{L^{r}}\|u_{n}\|_{B^{s}_{r,2}}^{\alpha}\\
\nonumber&&+\|\Delta_{y}u\|_{L^{r}}\int_{0}^{1}\|Dg_{1}([\tau_{y}u_{n},u_{n}]_{\theta})-Dg_{1}([\tau_{y}u,u]_{\theta})\|_{L^{\frac{p}{\alpha}}}d\theta\},
\end{eqnarray}
where $\tau_{y}F(x):=F(x-y)$ is the translation operator and $[v,u]_{\theta}:=\theta v+(1-\theta)u$ for $0\leq\theta\leq1$. Moreover, if we assume further that $\alpha\geq[s]+1=1$, then by \eqref{class2} and $B^{s}_{r,2}\hookrightarrow L^{p}$, we have
\begin{equation}\label{eq3.150}
\begin{array}{ll}
\|\Delta_{y}(g_{1}(u_{n})-g_{1}(u))\|_{L^{r'}}\leq C\{\|u_{n}\|_{B^{s}_{r,2}}^{\alpha}\|\Delta_{y}(u_{n}-u)\|_{L^{r}}&\\
\qquad\qquad\qquad\qquad+\|\Delta_{y}u\|_{L^{r}}(\|u_{n}\|^{\alpha-1}_{B^{s}_{r,2}}+\|u\|^{\alpha-1}_{B^{s}_{r,2}})\|u_{n}-u\|_{B^{s}_{r,2}}\}.&
\end{array}
\end{equation}

Now we assume $N\geq3$, for any $j=1,\cdots,N$, let $\partial_{j}$ denotes the partial derivative $\partial_{x_{j}}$, by simple calculations, we get the following estimate
\begin{eqnarray*}
&&\|\Delta_{y}\partial_{j}^{[s]}(g_{1}(u_{n})-g_{1}(u))\|_{L^{r'}}\\
\nonumber&\leq&\sum_{k=1}^{[s]}\sum_{\sum_{1}^{k}\beta_i=[s],\beta_i\geq1}
\{\|\Delta_{y}(u_{n}-u)\int_{0}^{1}D^{k+1}g_{1}([\tau_{y}u_{n},u_{n}]_{\theta})d\theta\prod_{i=1}^{k}\partial_{j}^{\beta_{i}}\tau_{y}u\|_{L^{r'}}\\
\nonumber&&+\|\Delta_{y}u\int_{0}^{1}[D^{k+1}g_{1}([\tau_{y}u_{n},u_{n}]_{\theta})-D^{k+1}g_{1}([\tau_{y}u,u]_{\theta})]d\theta
\prod_{i=1}^{k}\partial_{j}^{\beta_{i}}\tau_{y}u\|_{L^{r'}}\\
\nonumber&&+\sum_{(v_{1},\cdots,v_{k})}\|(D^{k}g_{1}(\tau_{y}u_{n})-D^{k}g_{1}(u_{n}))\prod_{i=1}^{k}\partial_{j}^{\beta_{i}}\tau_{y}v_{i}\|_{L^{r'}}\\
\nonumber&&+\sum_{(w_{1},\cdots,w_{k})}\|(D^{k}g_{1}(u_{n})-D^{k}g_{1}(u))\prod_{i=1}^{k}\partial_{j}^{\beta_{i}}w_{i}\|_{L^{r'}}\\
\nonumber&&+\sum_{(\mu_{1},\cdots,\mu_{k})}\|D^{k}g_{1}(u)\prod_{i=1}^{k}\partial_{j}^{\beta_{i}}\mu_{i}\|_{L^{r'}}
+\sum_{(\nu_{1},\cdots,\nu_{k})}\|D^{k}g_{1}(u)\prod_{i=1}^{k}\partial_{j}^{\beta_{i}}\nu_{i}\|_{L^{r'}}\}\\
\nonumber&=:&G_{1}+G_{2}+G_{3}+G_{4}+G_{5}+G_{6},
\end{eqnarray*}
where $\tau_{y}F(x):=F(x-y)$ is the translation operator; all the $v_{i}$'s are equal to $u_{n}$ or $u$, except one which is equal to $u_{n}-u$; all the $w_{i}$'s are equal to $\tau_{y}u_{n}$ or $u_{n}$, except one which is equal to $\Delta_{y}u_{n}$; if $k=1$, $\mu_{1}=0$, if $2\leq k\leq[s]$, $\mu_{i}$'s are equal to $u_{n}$, $u$ or $\tau_{y}u_{n}$, $\tau_{y}u$, except two of them, one is equal to $\Delta_{y}u_{n}$, the other one is equal to $\tau_{y}(u_{n}-u)$ or $u_{n}-u$; all the $\nu_{i}$'s are equal to $\tau_{y}u$ or $u$, except one which is equal to $\Delta_{y}(u_{n}-u)$.

Since $g_{1}$ is of class $\mathcal{C}(\alpha,s)$, by \eqref{class1}, \eqref{eq3.0}, H\"{o}lder's estimates and the embedding $B^{s}_{r,2}\hookrightarrow L^{p}$, $B^{s}_{r,2}\hookrightarrow L^{\infty}$, $B^{s}_{r,2}\hookrightarrow H^{\beta_{i},\frac{N}{\beta_{i}}}$, we get the following estimate for $G_{1}$:
\begin{equation*}
\begin{array}{ll}
G_{1}\leq C\|\Delta_{y}(u_{n}-u)\|_{L^{\frac{rN}{N-r[s]}}}\|u_{n}\|_{L^{p}}^{\alpha-[s]}\sum_{k=1}^{[s]}\{\|u_{n}\|_{L^{\infty}}^{[s]-k}\sum_{\beta_1+\cdots+\beta_{k}=[s],\beta_i\geq1}&\\
\qquad\prod_{i=1}^{k}\|\partial_{j}^{\beta_{i}}u\|_{L^{N/\beta_{i}}}\}\leq C(\|u_{n}\|_{B^{s}_{r,2}}^{\alpha}+\|u\|_{B^{s}_{r,2}}^{\alpha})\sum_{j=1}^{N}\|\Delta_{y}\partial_{j}^{[s]}(u_{n}-u)\|_{L^{r}}.&
\end{array}
\end{equation*}

By \eqref{class1}, \eqref{class2}, \eqref{eq3.0} and H\"{o}lder's estimates, we get two cases for the estimate of $G_{2}$ respectively:\\
(i)\, if $g_{1}$ is not a polynomial and $\alpha\geq[s]+1=\frac{N+1}{2}$, then
\begin{equation*}
\begin{array}{ll}
G_{2}\leq C\|\Delta_{y}u\|_{L^{\frac{rN}{N-r[s]}}}(\|u_{n}\|_{L^{p}}+\|u\|_{L^{p}})^{\alpha-[s]-1}\|u_{n}-u\|_{L^{p}}\sum_{k=1}^{[s]}\{&\\
\qquad\qquad(\|u_{n}\|_{L^{\infty}}+\|u\|_{L^{\infty}})^{[s]-k}\sum_{\beta_1+\cdots+\beta_{k}=[s],\beta_i\geq1}
\prod_{i=1}^{k}\|\partial_{j}^{\beta_{i}}u\|_{L^{N/\beta_{i}}}\}&\\
\qquad\qquad\qquad\leq C(\|u_{n}\|_{B^{s}_{r,2}}+\|u\|_{B^{s}_{r,2}})^{\alpha-1}
\|u_{n}-u\|_{B^{s}_{r,2}}\sum_{j=1}^{N}\|\Delta_{y}\partial_{j}^{[s]}u\|_{L^{r}};&
\end{array}
\end{equation*}
(ii) if $g_{1}$ is not a polynomial and $\frac{N-1}{2}=[s]<\alpha<[s]+1=\frac{N+1}{2}$, then
\begin{equation*}
\begin{array}{ll}
G_{2}\leq C\{(\|u_{n}\|_{B^{s}_{r,2}}+\|u\|_{B^{s}_{r,2}})^{\alpha-1}
\|u_{n}-u\|_{B^{s}_{r,2}}+\int_{0}^{1}\|D^{[s]+1}g_{1}([\tau_{y}u_{n},u_{n}]_{\theta})&\\
\qquad\qquad-D^{[s]+1}g_{1}([\tau_{y}u,u]_{\theta})\|_{L^{\frac{p}{\alpha-[s]}}}d\theta
\|u\|^{[s]}_{B^{s}_{r,2}}\}\sum_{j=1}^{N}\|\Delta_{y}\partial_{j}^{[s]}u\|_{L^{r}}.&
\end{array}
\end{equation*}

For $G_{3}$, note that all the $v_{i}$'s are equal to $u_{n}$ or $u$, except one which is equal to $u_{n}-u$, we deduce from \eqref{class1}, \eqref{eq3.0} and H\"{o}lder's estimates that
\begin{equation*}
\begin{array}{ll}
G_{3}\leq C\|\Delta_{y}u_{n}\|_{L^{\frac{rN}{N-r[s]}}}\|u_{n}\|_{L^{p}}^{\alpha-[s]}\sum_{k=1}^{[s]}\{\|u_{n}\|_{L^{\infty}}^{[s]-k}&\\
\qquad\qquad\sum_{\beta_{1}+\cdots+\beta_k=[s],\beta_i\geq1}
\sum_{(v_{1},\cdots,v_{k})}\prod_{i=1}^{k}\|\partial_{j}^{\beta_{i}}v_{i}\|_{L^{N/\beta_{i}}}\}&\\
\qquad\qquad\qquad\leq C(\|u_{n}\|_{B^{s}_{r,2}}+\|u\|_{B^{s}_{r,2}})^{\alpha-1}
\|u_{n}-u\|_{B^{s}_{r,2}}\sum_{j=1}^{N}\|\Delta_{y}\partial_{j}^{[s]}u_{n}\|_{L^{r}}.&
\end{array}
\end{equation*}

For $k\in\{1,\cdots,[s]\}$ and $\beta_{i}\geq1$($1\leq i\leq k$) such that $\sum_{i=1}^{k}\beta_{i}=[s]$, we define indices $p_{i}$($1\leq i\leq k$) by $p_{1}=\frac{rN}{N-r([s]-\beta_{1})}$ and $p_{i}=\frac{N}{\beta_{i}}$ for $2\leq i\leq k$.

For the estimate of $G_{4}$, without loss of generality, we may assume $w_{1}=\Delta_{y}u_{n}$, then we deduce from \eqref{class1}, \eqref{eq3.0}, H\"{o}lder's estimates and Sobolev embedding $\dot{H}^{[s],r}\hookrightarrow\dot{H}^{\beta_{1},p_{1}}$, $B^{s}_{r,2}\hookrightarrow H^{\beta_{i},p_{i}}$ that
\begin{equation*}
\begin{array}{ll}
G_{4}\leq C\|u_{n}-u\|_{L^{\infty}}(\|u_{n}\|_{L^{p}}+\|u\|_{L^{p}})^{\alpha-[s]}\sum_{k=1}^{[s]}\{(\|u_{n}\|_{L^{\infty}}+\|u\|_{L^{\infty}})^{[s]-k}&\\
\qquad\qquad\sum_{\beta_{1}+\cdots+\beta_k=[s],\beta_i\geq1}\sum_{(w_{1},\cdots,w_{k})}\|\partial_{j}^{\beta_{1}}\Delta_{y}u_{n}\|_{L^{p_{1}}}
\prod_{i=2}^{k}\|\partial_{j}^{\beta_{i}}w_{i}\|_{L^{p_{i}}}\}&\\
\qquad\qquad\qquad\leq C(\|u_{n}\|_{B^{s}_{r,2}}+\|u\|_{B^{s}_{r,2}})^{\alpha-1}
\|u_{n}-u\|_{B^{s}_{r,2}}\sum_{j=1}^{N}\|\Delta_{y}\partial_{j}^{[s]}u_{n}\|_{L^{r}}.&
\end{array}
\end{equation*}

For $G_{5}$, we know if $k=1$, $\mu_{1}=0$; if $2\leq k\leq[s]$, we may assume without loss of generality that $\mu_{1}=\Delta_{y}u_{n}$ and $\mu_{2}$ is equal to $\tau_{y}(u_{n}-u)$ or $u_{n}-u$. It follows from \eqref{class1}, \eqref{eq3.0}, H\"{o}lder's estimates and Sobolev embedding that
\begin{equation*}
\begin{array}{ll}
G_{5}\leq C\|u\|_{L^{p}}^{\alpha-[s]}\sum_{k=2}^{[s]}\{\|u\|_{L^{\infty}}^{[s]+1-k}\sum_{\beta_{1}+\cdots+\beta_k=[s],\beta_i\geq1}\sum_{(\mu_{1},\cdots,\mu_{k})}&\\
\qquad\qquad\qquad\|\partial_{j}^{\beta_{1}}\Delta_{y}u_{n}\|_{L^{p_{1}}}
\prod_{i=2}^{k}\|\partial_{j}^{\beta_{i}}\mu_{i}\|_{L^{p_{i}}}\}&\\
\qquad\qquad\leq C(\|u_{n}\|_{B^{s}_{r,2}}+\|u\|_{B^{s}_{r,2}})^{\alpha-1}
\|u_{n}-u\|_{B^{s}_{r,2}}\sum_{j=1}^{N}\|\Delta_{y}\partial_{j}^{[s]}u_{n}\|_{L^{r}}.&
\end{array}
\end{equation*}

As to $G_{6}$, similarly, we may assume without loss of generality that $\nu_{1}=\Delta_{y}(u_{n}-u)$, then we deduce from \eqref{class1}, \eqref{eq3.0}, H\"{o}lder's estimates and Sobolev embedding that
\begin{equation*}
\begin{array}{ll}
G_{6}\leq C\|u\|_{L^{p}}^{\alpha-[s]}\sum_{k=1}^{[s]}\{\|u\|_{L^{\infty}}^{[s]+1-k}\sum_{\beta_{1}+\cdots+\beta_k=[s],\beta_i\geq1}\sum_{(\nu_{1},\cdots,\nu_{k})}&\\
\qquad\qquad\|\partial_{j}^{\beta_{1}}\Delta_{y}(u_{n}-u)\|_{L^{p_{1}}}
\prod_{i=2}^{k}\|\partial_{j}^{\beta_{i}}\nu_{i}\|_{L^{p_{i}}}\}&\\
\qquad\qquad\qquad\leq C\|u\|_{B^{s}_{r,2}}^{\alpha}\sum_{j=1}^{N}\|\Delta_{y}\partial_{j}^{[s]}(u_{n}-u)\|_{L^{r}}.&
\end{array}
\end{equation*}

Therefore, under the assumption that $g_{1}$ is not a polynomial, by combining the estimates of $G_{1}$-$G_{6}$, we get two cases for the estimate of $\|\Delta_{y}\partial_{x_{j}}^{[s]}(g_{1}(u_{n})-g_{1}(u))\|_{L^{r'}}$ for $N\geq3$, that is, for any $j=1,\cdots,N$,\\
(i)\, if $g_{1}$ is not a polynomial and $\alpha\geq[s]+1=\frac{N+1}{2}$, then
\begin{equation}\label{eq3.16}
\begin{array}{ll}
\|\Delta_{y}\partial_{x_{j}}^{[s]}(g_{1}(u_{n})-g_{1}(u))\|_{L^{r'}}\leq C\{(\|u_{n}\|_{B^{s}_{r,2}}^{\alpha}+\|u\|_{B^{s}_{r,2}}^{\alpha})&\\
\qquad\sum_{j=1}^{N}\|\Delta_{y}\partial_{j}^{[s]}(u_{n}-u)\|_{L^{r}}
+(\|u_{n}\|_{B^{s}_{r,2}}+\|u\|_{B^{s}_{r,2}})^{\alpha-1}&\\
\qquad\qquad\sum_{j=1}^{N}(\|\Delta_{y}\partial_{j}^{[s]}u_{n}\|_{L^{r}}
+\|\Delta_{y}\partial_{j}^{[s]}u\|_{L^{r}})\|u_{n}-u\|_{B^{s}_{r,2}}\};&
\end{array}
\end{equation}
(ii) if $g_{1}$ is not a polynomial and $\frac{N-1}{2}=[s]<\alpha<[s]+1=\frac{N+1}{2}$, then
\begin{equation}\label{eq3.17}
\begin{array}{ll}
\|\Delta_{y}\partial_{x_{j}}^{[s]}(g_{1}(u_{n})-g_{1}(u))\|_{L^{r'}}\leq C\{(\|u_{n}\|_{B^{s}_{r,2}}^{\alpha}+\|u\|_{B^{s}_{r,2}}^{\alpha})&\\
\qquad\sum_{j=1}^{N}\|\Delta_{y}\partial_{j}^{[s]}(u_{n}-u)\|_{L^{r}}
+(\|u_{n}\|_{B^{s}_{r,2}}+\|u\|_{B^{s}_{r,2}})^{\alpha-1}
\sum_{j=1}^{N}(\|\Delta_{y}\partial_{j}^{[s]}u_{n}\|_{L^{r}}&\\
\qquad\qquad+\|\Delta_{y}\partial_{j}^{[s]}u\|_{L^{r}})\|u_{n}-u\|_{B^{s}_{r,2}}
+\int_{0}^{1}\|D^{[s]+1}g_{1}([\tau_{y}u_{n},u_{n}]_{\theta})&\\
\qquad\qquad\qquad-D^{[s]+1}g_{1}([\tau_{y}u,u]_{\theta})\|_{L^{\frac{p}{\alpha-[s]}}}d\theta
\|u\|^{[s]}_{B^{s}_{r,2}}\sum_{j=1}^{N}\|\Delta_{y}\partial_{j}^{[s]}u\|_{L^{r}}\}.&
\end{array}
\end{equation}

Therefore, if $N$ is odd and the nonlinearity $g_{1}\in\mathcal{C}(\alpha,s)$ is not a polynomial, by inserting \eqref{eq3.15}, \eqref{eq3.150}, \eqref{eq3.16} and \eqref{eq3.17} into \eqref{eq3.14} and applying H\"{o}lder's inequality on time, we have the following two cases for the estimate of $\|g_{1}(u_{n})-g_{1}(u)\|_{L^{q'}(I,B^{s}_{r',2})}$:\\
(i)\, if $g_{1}$ is not a polynomial and $\alpha\geq[s]+1=\frac{N+1}{2}$, then
\begin{eqnarray}\label{eq3.18}
&&\|g_{1}(u_{n})-g_{1}(u)\|_{L^{q'}(I,B^{s}_{r',2})}\\
\nonumber&\leq&CT^{1-\frac{\alpha+2}{q}}(\|u_{n}\|^{\alpha}_{L^{q}(I,B^{s}_{r,2})}+\|u\|^{\alpha}_{L^{q}(I,B^{s}_{r,2})})
\|u_{n}-u\|_{L^{q}(I,B^{s}_{r,2})};
\end{eqnarray}
(ii) if $g_{1}$ is not a polynomial and $\frac{N-1}{2}=[s]<\alpha<[s]+1=\frac{N+1}{2}$, then
\begin{equation}\label{eq3.19}
\begin{array}{ll}
\|g_{1}(u_{n})-g_{1}(u)\|_{L^{q'}(I,B^{s}_{r',2})}\leq C\{T^{1-\frac{\alpha+2}{q}}(\|u_{n}\|^{\alpha}_{L^{q}(I,B^{s}_{r,2})}&\\
\qquad\qquad+\|u\|^{\alpha}_{L^{q}(I,B^{s}_{r,2})})\|u_{n}-u\|_{L^{q}(I,B^{s}_{r,2})}+\|K(u_{n},u)\|_{L^{q'}(I)}\},&
\end{array}
\end{equation}
where the Kato's remainder term $K(u_{n},u)$ can be defined by
\begin{equation}\label{eq3.20}
\begin{array}{ll}
K(u_{n},u):=\|u\|^{[s]+1}_{B^{s}_{r,2}}\int_{0}^{1}\|D^{[s]+1}g_{1}([\tau_{y}u_{n},u_{n}]_{\theta})&\\
\qquad\qquad\qquad\qquad\qquad\qquad-D^{[s]+1}g_{1}([\tau_{y}u,u]_{\theta})\|_{L^{\frac{p}{\alpha-[s]}}}d\theta.&
\end{array}
\end{equation}

As to the cases that $g_{1}(u)$ is a polynomial in $u$ and $\bar{u}$, note that $\alpha\geq1$ in such cases, we define $\tilde{p}=1+\frac{2}{\delta}$, then we have $r<\tilde{p}<\infty$, $B^{s}_{r,2}\hookrightarrow L^{\tilde{p}}$ and
\[\frac{1}{r'}=\frac{1}{\tilde{p}}+\frac{1}{r}.\]
Then similar to the arguments for the cases that $g_{1}$ is not a polynomial as above, by H\"{o}lder's estimates, it follows from a similar but much simpler argument that
\begin{eqnarray}\label{eq3.21}
&&\|g_{1}(u_{n})-g_{1}(u)\|_{L^{q'}(I,{B}^{s}_{r',2})}\\
\nonumber&\leq&CT^{1-\frac{\alpha+2}{q}}(\|u_{n}\|^{\alpha}_{L^{q}(I,B^{s}_{r,2})}+\|u\|^{\alpha}_{L^{q}(I,B^{s}_{r,2})})
\|u_{n}-u\|_{L^{q}(I,B^{s}_{r,2})}.
\end{eqnarray}

Therefore, by inserting \eqref{eq3.18}, \eqref{eq3.19} and \eqref{eq3.21} into the original Strichartz's estimate \eqref{eq3.13} and using \eqref{eq3.2}, we can obtain estimate for $\tilde{X}_n$ by considering the following two cases respectively:\\
Case (i) \,\, if $g_{1}(u)$ is a polynomial in $u$ and $\bar{u}$, or if not, we assume further that $\alpha\geq [s]+1$, then
\begin{equation}\label{eq3.11}
\tilde{X}_n\leq C\|\phi_{n}-\phi\|_{H^{s}}+CT\tilde{X}_n+CT^{\sigma}M^{\alpha}\tilde{X}_n;
\end{equation}
Case (ii) \, if $g_{1}$ is not a polynomial and $[s]<\alpha<[s]+1$, then
\begin{equation}\label{eq3.12}
\tilde{X}_n\leq C\|\phi_{n}-\phi\|_{H^{s}}+C\{T\tilde{X}_n+T^{\sigma}M^{\alpha}\tilde{X}_n+\|K(u_{n},u)\|_{L^{q'}(-T,T)}\};
\end{equation}
where $\sigma=1-\frac{\alpha+2}{q}>0$.

For case $(i)$, we can choose $T$ sufficiently small so that $C(T+T^{\sigma}M^{\alpha})\leq\frac{1}{2}$, and deduce from \eqref{eq3.11} that as $n\rightarrow \infty$,
\[\tilde{X}_{n}=\|u_{n}(t)-u(t)\|_{L^{\infty}(I,H^{s})}+\|u_{n}(t)-u(t)\|_{L^{q}(I,B^{s}_{r,2})}\leq C\|\phi_{n}-\phi\|_{H^{s}}\rightarrow 0,\]
and hence the solution flow is locally Lipschitz.

As to case $(ii)$, we can see that the key point is to show that the Kato's remainder term $\|K(u_{n},u)\|_{L^{q'}(-T,T)}\rightarrow0$ as $n\rightarrow\infty$, which will yield the desired convergence of $\tilde{X}_{n}$ for $T$ small enough immediately. By making use of \eqref{class1} and the convergence $u_{n}\rightarrow u$ in $L^{q}(I,L^{p})$, we certainly can construct a control function $\omega\in L^{\frac{(\alpha-[s])q}{q-[s]-2}}(I,L^{p}(\mathbb{R}^{N}))$ and apply the contradiction argument as in Case I and the even dimensional case, which won't rely on the assumption (\ref{class2}). It is completely similar to the even dimensional case and Case I, so we omit the details here. We can also make use of assumption (\ref{class2}) to obtain a H\"{o}lder type estimate of $K(u_{n},u)$:
\[\|K(u_{n},u)\|_{L^{q'}(I)}\leq CT^{1-\frac{\alpha+2}{q}}\|u\|^{[s]+1}_{L^{q}(I,B^{s}_{r,2})}\|u_{n}-u\|^{\alpha-[s]}_{L^{q}(I,L^{p})}\]
and hence, it follows from $u_{n}\rightarrow u$ in $L^{q}(I,L^{p})$ that
\[\|K(u_{n},u)\|_{L^{q'}(-T,T)}\rightarrow0, \,\,\,\, as \,\,\,\, n\rightarrow\infty.\]

Therefore, for the cases that $N$ is odd, we have proved $\tilde{X}_{n}\rightarrow 0$ as $n\rightarrow \infty$ if $T$ is sufficiently small in both case (i) and (ii).

In a word, we have proved $\tilde{X}_{n}\rightarrow 0$ as $n\rightarrow \infty$ if $T$ is sufficiently small for both $N$ even and odd. The convergence for arbitrary admissible pair follows from Strichartz's estimates. The conclusions $(i)$ and $(ii)$ of Theorem \ref{th1} follow by iterating this property to cover any compact subset of $(-T_{min},T_{max})$. Moreover, if $g_{1}(u)$ is a polynomial in $u$ and $\bar{u}$, or if not, we assume further that $\alpha\geq [s]+1$, then the continuous dependence is locally Lipschitz. \\
\\
Case IV. $s>\frac{N}{2}$.\\

Since $s>N/2$, we know that $H^{s}=B^{s}_{2,2}\hookrightarrow L^{\infty}$ and $H^{s}$ is an algebra. Let $Z_n(I):=\|u_n-u\|_{L^\infty(I,H^{s})}$, where $I=(-T,T)$. We deduce from Strichartz's estimate that
\begin{equation}\label{eq4.1}
Z_n\leq C\|\varphi_n-\varphi\|_{H^s}+CTZ_n+C\|g_1(u_n)-g_1(u)\|_{L^{1}(I,H^{s})}.
\end{equation}
To get our conclusion, we are to prove that $Z_{n}\rightarrow0$ as $n\rightarrow\infty$, thus we need to estimate $\|g_1(u_n)-g_1(u)\|_{L^{1}(I,H^{s})}$.

Since $\|u_{n}\|_{L^\infty(I,H^s)}$ is bounded, we define
\begin{equation}\label{eq4.2}
M:=\|u\|_{L^\infty(I,H^s)}+\sup_{n\geq 1}\|u_{n}\|_{L^\infty(I,H^s)}< \infty
\end{equation}
such that $\|u\|_{L^\infty(I,H^s)}\leq M$ and $\|u_{n}\|_{L^\infty(I,H^s)}\leq M$ for any $n\geq1$.

First, let us consider the simpler case $s\in\mathbb{N}$. We have the formula
\begin{eqnarray}\label{eq4.3}
&&\||\nabla|^{s}(g_{1}(u_{n})-g_{1}(u))\|_{L^{2}}\\
\nonumber&\leq&\sum_{k=1}^{s}\sum_{|\beta_1|+\cdots+|\beta_k|=s,|\beta_j|\geq1}\{\|(g_{1,z\backslash\bar{z}}^{(k)}(u_n)-g_{1,z\backslash\bar{z}}^{(k)}(u))
\prod_{j=1}^k\partial^{\beta_j}(u\backslash\bar{u})\|_{L^{2}}\\
\nonumber&&+\sum_{(w_{1},\cdots,w_{k})}\|g_{1,z\backslash\bar{z}}^{(k)}(u_{n})\prod_{j=1}^k\partial^{\beta_j}w_j\|_{L^{2}}\}=:(I)+(II),
\end{eqnarray}
where $k\in \{1,\cdots,s\}$ and the $\beta_{j}$'s are multi-indices such that $s=|\beta_{1}|+\cdots+|\beta_{k}|$ and $|\beta_{j}|\geq 1$ for $j=1,\cdots,k$, all the $w_j$'s are equal to $u_{n},\bar{u}_{n}$ or $u,\bar{u}$, except one which is equal to $u_{n}-u$ or its conjugate. Since $g_{1}$ is of class $\mathcal{C}(\alpha,s)$, for the estimate of $(I)$, we have two cases:\\
(i)\, if $g_{1}\in C^{[s]+2}(\mathbb{C},\mathbb{C})$, then $D^{k}g_{1}$($k=1,\cdots,s$) is Lipschitz on bounded sets, thus we deduce from \eqref{eq4.3} and $H^{s}\hookrightarrow H^{|\beta_{j}|,\frac{2s}{|\beta_{j}|}}$ that
\begin{eqnarray*}
(I)&\leq&C_{M}\|u_{n}-u\|_{L^{\infty}}\sum_{k=1}^{s}\sum_{|\beta_1|+\cdots+|\beta_k|=s,|\beta_j|\geq1}\prod_{j=1}^{k} \|\partial^{\beta_j}u\|_{L^{2s/|\beta_{j}|}}\\
&\leq&C_{M}\|u_{n}-u\|_{H^{s}};
\end{eqnarray*}
(ii) otherwise, since $g_{1}$ is of class $\mathcal{C}(\alpha,s)$, we have $D^{k}g_{1}$($k=1,\cdots,s-1$) is Lipschitz on bounded sets, except $D^{s}g_{1}$, which is uniformly continuous on bounded sets, so we deduce from \eqref{eq4.3} and $H^{s}\hookrightarrow H^{|\beta_{j}|,\frac{2s}{|\beta_{j}|}}$ that
\begin{equation*}
(I)\leq C_{M}\{\|u_{n}-u\|_{H^{s}}+\varepsilon_{M}(\|u_{n}-u\|_{L^{\infty}})\},
\end{equation*}
where $\varepsilon_{M}(t)\rightarrow0$ as $t\downarrow0$. As to $(II)$, note that all the $w_j$'s are equal to $u_{n},\bar{u}_{n}$ or $u,\bar{u}$, except one which is equal to $u_{n}-u$ or its conjugate, by H\"{o}lder's eatimates, we have
\begin{eqnarray*}
(II)&\leq&C\sum_{k=1}^{s}\sum_{\sum_{j=1}^{k}|\beta_j|=s,|\beta_j|\geq1}\sum_{(w_{1},\cdots,w_{k})}\|g_{1,z\backslash
\bar{z}}^{(k)}(u_{n})\|_{L^{\infty}}\prod_{j=1}^{k} \|\partial^{\beta_j}w_{j}\|_{L^{2s/|\beta_{j}|}}\\
&\leq&C_{M}\|u_{n}-u\|_{H^{s}}.
\end{eqnarray*}
Therefore, by inserting the estimates of $(I)$ and $(II)$ into \eqref{eq4.3} and applying H\"{o}lder's inequality on time, we get two cases:\\
(i)\, if $g_{1}\in C^{[s]+2}(\mathbb{C},\mathbb{C})$, then
\begin{equation}\label{eq4.4}
\|g_1(u_n)-g_1(u)\|_{L^{1}(I,H^{s})}\leq CT\|u_{n}-u\|_{L^{\infty}(I,H^{s})};
\end{equation}
(ii) otherwise, we can deduce from the convergence in $L^{\infty}$ norm that
\begin{equation}\label{eq4.5}
\|g_1(u_n)-g_1(u)\|_{L^{1}(I,H^{s})}\leq\varepsilon_{n}+CT\|u_{n}-u\|_{L^{\infty}(I,H^{s})},
\end{equation}
where $\varepsilon_{n}\rightarrow0$ as $n\rightarrow\infty$.

Now we turn to the cases that $s$ is not an integer.

First, let us consider a special case, that is, the nonlinearity $g_{1}$ is of class $\mathcal{C}(\alpha,s)$ with a further assumption $g_{1}\in C^{[s]+2}(\mathbb{C},\mathbb{C})$. By the fundamental theorem of calculus, we can write
\begin{equation}\label{eq4.6}
g_{1}(u_{n})-g_{1}(u)=\int_{0}^{1}[g_{1z}([u_{n},u]_{\theta})(u_{n}-u)+g_{1\bar{z}}([u_{n},u]_{\theta})(\bar{u}_{n}-\bar{u})]d\theta,
\end{equation}
where $[u_{n},u]_{\theta}:=\theta u_{n}+(1-\theta)u$ for $0\leq\theta\leq1$. Thus, by Morse type inequality(see Lemma \ref{lem7}), we have
\begin{equation}\label{eq4.7}
\|g_{1}(u_{n})-g_{1}(u)\|_{\dot{H}^{s}}\leq C\int_{0}^{1}\|Dg_{1}([u_{n},u]_{\theta})\|_{\dot{H}^{s}\cap L^{\infty}}\|u_{n}-u\|_{H^{s}}d\theta.
\end{equation}

By resorting to the technique of multilinear paradifferential expansions(see \cite{Tataru}), we have the following Schauder estimate for homogeneous spaces.
\begin{lem}\label{Schauder}(Schauder estimate).
Let $V$ be a finite-dimensional normed vector space and let $f\in\dot{H}_{x}^{s}(\mathbb{R}^{N}\rightarrow V)\bigcap L_{x}^{\infty}(\mathbb{R}^{N}\rightarrow V)$ for some $s\geq0$. If $F\in C_{loc}^{[s]+1}(V\rightarrow V)$, then $F(f)\in\dot{H}_{x}^{s}(\mathbb{R}^{N}\rightarrow V)$ as well. Moreover, we have a bound of the form
\[\|F(f)\|_{\dot{H}_{x}^{s}(\mathbb{R}^{N})}\leq C_{F,\|f\|_{L_{x}^{\infty}},V,s,N}\|f\|_{\dot{H}_{x}^{s}(\mathbb{R}^{N})}.\]
\end{lem}

For the proof of Lemma \ref{Schauder}, we refer to Lemma A.9 in \cite{Tao3}.

Since $g_{1}\in C^{[s]+2}(\mathbb{C},\mathbb{C})$, we deduce from \eqref{eq4.2} and Lemma \ref{Schauder} that
\begin{equation}\label{eq4.8}
\|Dg_{1}([u_{n},u]_{\theta})\|_{\dot{H}^{s}}\leq C_{M,s,N}\|[u_{n},u]_{\theta}\|_{\dot{H}^{s}}\leq C.
\end{equation}
Note that $\|[u_{n},u]_{\theta}\|_{L^{\infty}}\leq CM$, so we have $\|Dg_{1}([u_{n},u]_{\theta})\|_{L^{\infty}}\leq C_{M}$, then it follows from \eqref{eq4.7} and \eqref{eq4.8} that
\begin{equation}
\|g_{1}(u_{n})-g_{1}(u)\|_{\dot{H}^{s}}\leq C\|u_{n}-u\|_{H^{s}}
\end{equation}
and hence, by applying H\"{o}lder's inequality on time, we get
\begin{equation}\label{eq4.9}
\|g_1(u_n)-g_1(u)\|_{L^{1}(I,H^{s})}\leq CT\|u_{n}-u\|_{L^{\infty}(I,H^{s})}.
\end{equation}

For the general case that $g_{1}$ is of class $\mathcal{C}(\alpha,s)$, by Lemma \ref{alternative}, we have
\begin{equation}\label{eq4.10}
\|f\|_{H^{s}}\sim\|f\|_{L^{2}}+\sum_{j=1}^{N}(\int_{0}^{\infty}(t^{-\{s\}}\sup_{|y|\leq t}\|\Delta_{y}\partial_{x_{j}}^{[s]}f\|_{L^{2}})^{2}\frac{dt}{t})^{1/2},
\end{equation}
where $\Delta_{y}f:=\tau_{y}f-f$ denotes the first order difference operator. Therefore, we can see that the key ingredient is how to estimate $\|\Delta_{y}\partial_{x_{j}}^{[s]}(g_{1}(u_{n})-g_{1}(u))\|_{L^{2}}$ for $j=1,\cdots,N$.

If $N=1$, $\frac{1}{2}<s<1$, since $g_{1}\in C^{1}(\mathbb{C},\mathbb{C})$, by \eqref{eq4.2}, the fundamental theorem of calculus and H\"{o}lder's estimate, we have
\begin{eqnarray}\label{eq4.11}
&&\|\Delta_{y}(g_{1}(u_{n})-g_{1}(u))\|_{L^{2}}\\
\nonumber&\leq&C\{\|\Delta_{y}(u_{n}-u)\int_{0}^{1}Dg_{1}([\tau_{y}u_{n},u_{n}]_{\theta})d\theta\|_{L^{2}}\\
\nonumber&&+\|\Delta_{y}u\int_{0}^{1}[Dg_{1}([\tau_{y}u_{n},u_{n}]_{\theta})-Dg_{1}([\tau_{y}u,u]_{\theta})]d\theta\|_{L^{2}}\}\\
\nonumber&\leq&C\{\|\Delta_{y}(u_{n}-u)\|_{L^{2}}+\varepsilon_{M}(\|u_{n}-u\|_{L^{\infty}})\|\Delta_{y}u\|_{L^{2}}\},
\end{eqnarray}
where $\tau_{y}F(x):=F(x-y)$ is the translation operator, $[v,u]_{\theta}:=\theta v+(1-\theta)u$ for $0\leq\theta\leq1$ and $\varepsilon_{M}(t)\rightarrow0$ as $t\downarrow0$.

Now we assume $s>\max\{1,N/2\}$ and let $\partial_{j}$ denotes the partial derivative $\partial_{x_{j}}$ for $j=1,\cdots,N$. By simple calculations, we get the following estimate for any $j=1,\cdots,N$,
\begin{eqnarray*}
&&\|\Delta_{y}\partial_{j}^{[s]}(g_{1}(u_{n})-g_{1}(u))\|_{L^{2}}\\
\nonumber&\leq&\sum_{k=1}^{[s]}\sum_{\sum_{1}^{k}\beta_i=[s],\beta_i\geq1}
\{\|\Delta_{y}(u_{n}-u)\int_{0}^{1}D^{k+1}g_{1}([\tau_{y}u_{n},u_{n}]_{\theta})d\theta\prod_{i=1}^{k}\partial_{j}^{\beta_{i}}\tau_{y}u\|_{L^{2}}\\
\nonumber&&+\|\Delta_{y}u\int_{0}^{1}[D^{k+1}g_{1}([\tau_{y}u_{n},u_{n}]_{\theta})-D^{k+1}g_{1}([\tau_{y}u,u]_{\theta})]d\theta
\prod_{i=1}^{k}\partial_{j}^{\beta_{i}}\tau_{y}u\|_{L^{2}}\\
\nonumber&&+\sum_{(v_{1},\cdots,v_{k})}\|(D^{k}g_{1}(\tau_{y}u_{n})-D^{k}g_{1}(u_{n}))\prod_{i=1}^{k}\partial_{j}^{\beta_{i}}\tau_{y}v_{i}\|_{L^{2}}\\
\nonumber&&+\sum_{(w_{1},\cdots,w_{k})}\|(D^{k}g_{1}(u_{n})-D^{k}g_{1}(u))\prod_{i=1}^{k}\partial_{j}^{\beta_{i}}w_{i}\|_{L^{2}}\\
\nonumber&&+\sum_{(\mu_{1},\cdots,\mu_{k})}\|D^{k}g_{1}(u)\prod_{i=1}^{k}\partial_{j}^{\beta_{i}}\mu_{i}\|_{L^{2}}
+\sum_{(\nu_{1},\cdots,\nu_{k})}\|D^{k}g_{1}(u)\prod_{i=1}^{k}\partial_{j}^{\beta_{i}}\nu_{i}\|_{L^{2}}\}\\
\nonumber&=:&H_{1}+H_{2}+H_{3}+H_{4}+H_{5}+H_{6},
\end{eqnarray*}
where $\tau_{y}F(x):=F(x-y)$ is the translation operator; all the $v_{i}$'s are equal to $u_{n}$ or $u$, except one which is equal to $u_{n}-u$; all the $w_{i}$'s are equal to $\tau_{y}u_{n}$ or $u_{n}$, except one which is equal to $\Delta_{y}u_{n}$; if $k=1$, $\mu_{1}=0$, if $2\leq k\leq[s]$, $\mu_{i}$'s are equal to $u_{n}$, $u$ or $\tau_{y}u_{n}$, $\tau_{y}u$, except two of them, one is equal to $\Delta_{y}u_{n}$, the other one is equal to $\tau_{y}(u_{n}-u)$ or $u_{n}-u$; all the $\nu_{i}$'s are equal to $\tau_{y}u$ or $u$, except one which is equal to $\Delta_{y}(u_{n}-u)$.

Since $g_{1}\in C^{[s]+1}(\mathbb{C},\mathbb{C})$, by \eqref{eq4.2}, H\"{o}lder's estimates and Sobolev embedding theorem, we get the following estimates for $H_{1}$-$H_{3}$:
\begin{eqnarray*}
H_{1}&\leq&C_{M}\|\Delta_{y}(u_{n}-u)\|_{L^{\frac{2s}{\{s\}}}}\sum_{k=1}^{[s]}\sum_{\sum_{i=1}^{k}\beta_i=[s],\beta_i\geq1}
\prod_{i=1}^{k}\|\partial_{j}^{\beta_{i}}u\|_{L^{2s/\beta_{i}}}\\
&\leq&C(\|\Delta_{y}(u_{n}-u)\|_{L^{2}}+\sum_{j=1}^{N}\|\Delta_{y}\partial_{j}^{[s]}(u_{n}-u)\|_{L^{2}});
\end{eqnarray*}
\begin{equation*}
H_{2}\leq C(\|\Delta_{y}u\|_{L^{2}}+\sum_{j=1}^{N}\|\Delta_{y}\partial_{j}^{[s]}u\|_{L^{2}})
(\|u_{n}-u\|_{H^{s}}+\varepsilon_{M}(\|u_{n}-u\|_{L^{\infty}}));
\end{equation*}
\begin{eqnarray*}
H_{3}&\leq&C_{M}\|\Delta_{y}u_{n}\|_{L^{\frac{2s}{\{s\}}}}\sum_{k=1}^{[s]}\sum_{\sum_{i=1}^{k}\beta_i=[s],\beta_i\geq1}
\sum_{(v_{1},\cdots,v_{k})}\prod_{i=1}^{k}\|\partial_{j}^{\beta_{i}}v_{i}\|_{L^{2s/\beta_{i}}}\\
&\leq&C(\|\Delta_{y}u_{n}\|_{L^{2}}+\sum_{j=1}^{N}\|\Delta_{y}\partial_{j}^{[s]}u_{n}\|_{L^{2}})\|u_{n}-u\|_{H^{s}};
\end{eqnarray*}
where $\varepsilon_{M}(t)\rightarrow0$ as $t\downarrow0$.

For $k\in\{1,\cdots,[s]\}$ and $\beta_{i}\geq1$($1\leq i\leq k$) such that $\sum_{i=1}^{k}\beta_{i}=[s]$, we will define indices $p_{i}$($1\leq i\leq k$) as following. If $[s]\geq\frac{N}{2}$, we let $p_{i}=\frac{2[s]}{\beta_{i}}$ for $1\leq i\leq k$; if $[s]<\frac{N}{2}$, let $p_{1}=\frac{2N}{N-2([s]-\beta_{1})}$ and $p_{i}=\frac{N}{\beta_{i}}$ for $2\leq i\leq k$.

For the estimate of $H_{4}$, without loss of generality, we may assume $w_{1}=\Delta_{y}u_{n}$, then we deduce from \eqref{eq4.2}, H\"{o}lder's estimates and Sobolev embedding that
\begin{eqnarray*}
H_{4}&\leq&C\|u_{n}-u\|_{L^{\infty}}\sum_{k=1}^{[s]}\sum_{\sum_{1}^{k}\beta_i=[s],\beta_i\geq1}
\sum_{(w_{1},\cdots,w_{k})}\|\partial_{j}^{\beta_{1}}\Delta_{y}u_{n}\|_{L^{p_{1}}}\prod_{i=2}^{k}\|\partial_{j}^{\beta_{i}}w_{i}\|_{L^{p_{i}}}\\
&\leq&C(\|\Delta_{y}u_{n}\|_{L^{2}}+\sum_{j=1}^{N}\|\Delta_{y}\partial_{j}^{[s]}u_{n}\|_{L^{2}})\|u_{n}-u\|_{H^{s}}.
\end{eqnarray*}

For $H_{5}$, we know if $k=1$, $\mu_{1}=0$; if $2\leq k\leq[s]$, we may assume without loss of generality that $\mu_{1}=\Delta_{y}u_{n}$ and $\mu_{2}$ is equal to $\tau_{y}(u_{n}-u)$ or $u_{n}-u$. By \eqref{eq4.2}, H\"{o}lder's estimates and Sobolev embedding, we have
\begin{eqnarray*}
H_{5}&\leq&C\sum_{k=2}^{[s]}\sum_{\sum_{i=1}^{k}\beta_i=[s],\beta_i\geq1}
\sum_{(\mu_{1},\cdots,\mu_{k})}\|\partial_{j}^{\beta_{1}}\Delta_{y}u_{n}\|_{L^{p_{1}}}\prod_{i=2}^{k}\|\partial_{j}^{\beta_{i}}\mu_{i}\|_{L^{p_{i}}}\\
&\leq&C(\|\Delta_{y}u_{n}\|_{L^{2}}+\sum_{j=1}^{N}\|\Delta_{y}\partial_{j}^{[s]}u_{n}\|_{L^{2}})\|u_{n}-u\|_{H^{s}}.
\end{eqnarray*}

As to $H_{6}$, similarly, we may assume without loss of generality that $\nu_{1}=\Delta_{y}(u_{n}-u)$, then we deduce from \eqref{eq4.2}, H\"{o}lder's estimates and Sobolev embedding that
\begin{eqnarray*}
H_{6}&\leq&C\sum_{k=1}^{[s]}\sum_{\sum_{i=1}^{k}\beta_i=[s],\beta_i\geq1}
\sum_{(\nu_{1},\cdots,\nu_{k})}\|\partial_{j}^{\beta_{1}}\Delta_{y}(u_{n}-u)\|_{L^{p_{1}}}\prod_{i=2}^{k}\|\partial_{j}^{\beta_{i}}\nu_{i}\|_{L^{p_{i}}}\\
&\leq&C(\|\Delta_{y}(u_{n}-u)\|_{L^{2}}+\sum_{j=1}^{N}\|\Delta_{y}\partial_{j}^{[s]}(u_{n}-u)\|_{L^{2}}).
\end{eqnarray*}

Combining the estimates of $H_{1}$-$H_{6}$, we get the estimate of $\|\Delta_{y}\partial_{x_{j}}^{[s]}(g_{1}(u_{n})-g_{1}(u))\|_{L^{2}}$ for $s>\max\{1,N/2\}$, that is, for any $j=1,\cdots,N$,
\begin{equation}\label{eq4.12}
\begin{array}{ll}
\|\Delta_{y}\partial_{x_{j}}^{[s]}(g_{1}(u_{n})-g_{1}(u))\|_{L^{2}}\leq C\{(\|\Delta_{y}(u_{n}-u)\|_{L^{2}}&\\
\qquad+\sum_{j=1}^{N}\|\Delta_{y}\partial_{x_{j}}^{[s]}(u_{n}-u)\|_{L^{2}})+[(\|\Delta_{y}u\|_{L^{2}}+\sum_{j=1}^{N}\|\Delta_{y}\partial_{x_{j}}^{[s]}u\|_{L^{2}})&\\
\qquad\qquad+(\|\Delta_{y}u_{n}\|_{L^{2}}+\sum_{j=1}^{N}\|\Delta_{y}\partial_{x_{j}}^{[s]}u_{n}\|_{L^{2}})]\|u_{n}-u\|_{H^{s}}&\\
\qquad\qquad\qquad+(\|\Delta_{y}u\|_{L^{2}}+\sum_{j=1}^{N}\|\Delta_{y}\partial_{x_{j}}^{[s]}u\|_{L^{2}})\varepsilon_{M}(\|u_{n}-u\|_{L^{\infty}})\},&
\end{array}
\end{equation}
where $\varepsilon_{M}(t)\rightarrow0$ as $t\downarrow0$.

Therefore, if $s>N/2$ is not an integer and the nonlinearity $g_{1}$ is of class $\mathcal{C}(\alpha,s)$, by inserting \eqref{eq4.11} and \eqref{eq4.12} into \eqref{eq4.10} and applying H\"{o}lder's inequality on time, we deduce from the convergence of $L^{\infty}$ norm that
\begin{equation}\label{eq4.13}
\|g_1(u_n)-g_1(u)\|_{L^{1}(I,H^{s})}\leq\varepsilon_{n}+CT\|u_{n}-u\|_{L^{\infty}(I,H^{s})},
\end{equation}
where $\varepsilon_{n}\rightarrow0$ as $n\rightarrow\infty$.

From \eqref{eq4.11} and the estimate of $H_{2}$, we can easily observe that if we assume further $g_{1}\in C^{[s]+2}(\mathbb{C},\mathbb{C})$, then $D^{[s]+1}g_{1}$ is Lipschitz on bounded sets, thus a Lipschitz type estimate \eqref{eq4.9} will hold, this gives another way to show \eqref{eq4.9}, which doesn't rely on the Schauder estimate(Lemma \ref{Schauder}).

By inserting \eqref{eq4.4}, \eqref{eq4.5}, \eqref{eq4.9} and \eqref{eq4.13} into the original Strichartz's estimate \eqref{eq4.1}, we have the following two cases:\\
Case (i)\, if $g_{1}\in C^{[s]+2}(\mathbb{C},\mathbb{C})$, then
\begin{equation}\label{eq4.14}
Z_n\leq C\|\varphi_n-\varphi\|_{H^s}+CTZ_{n};
\end{equation}
Case (ii) otherwise, we have
\begin{equation}\label{eq4.15}
Z_n\leq C\|\varphi_n-\varphi\|_{H^s}+CTZ_{n}+\varepsilon_{n},
\end{equation}
where $\varepsilon_{n}\rightarrow0$ as $n\rightarrow\infty$.

For case (i), by choosing T sufficiently small such that $CT\leq 1/2$, we get
\[Z_{n}\leq C\|\varphi_n-\varphi\|_{H^s},\]
hence $Z_{n}=\|u_{n}-u\|_{L^\infty(I,H^s)}\rightarrow 0$ as $n\rightarrow\infty$, and the solution flow is locally Lipshcitz in $H^{s}(\mathbb{R}^{N})$.

As to case (ii), by choosing T sufficiently small such that $CT\leq 1/2$, we infer from \eqref{eq4.15} that
\[Z_n\leq C\|\varphi_n-\varphi\|_{H^s}+\varepsilon_{n},\]
thus $\varepsilon_{n}\rightarrow0$ as $n\rightarrow\infty$ yields the desired convergence.

Thus we have proved $Z_{n}\rightarrow 0$ as $n\rightarrow \infty$ if $T$ is sufficiently small in both case (i) and (ii). The convergence for arbitrary admissible pair $(q,r)$ follows from Strichartz's estimates. The conclusions $(i)$ and $(ii)$ of Theorem \ref{th1} follow by iterating this property to cover any compact subset of $(-T_{min},T_{max})$. Moreover, if we assume further $g_{1}\in C^{[s]+2}(\mathbb{C},\mathbb{C})$, then the continuous dependence is locally Lipschitz.\\

This completes the proof of Theorem \ref{th1}.\\

{\bf Acknowledgements:} The authors would like to thank professor Daoyuan Fang and Daniel Tataru for their valuable comment. D. Cao was partially supported by Science Fund for Creative Research Groups of NSFC(No. 10721101). \\

\end{document}